\documentclass{article}
\PassOptionsToPackage{numbers, compress}{natbib}
\usepackage[main, final]{neurips_2026}

\usepackage[utf8]{inputenc}
\usepackage[T1]{fontenc}
\usepackage{hyperref}
\usepackage{booktabs}
\usepackage{graphicx}
\usepackage{mathtools}
\usepackage{amssymb}
\usepackage{amsthm}
\usepackage{enumitem}
\usepackage{pifont}
\usepackage{tikz}
\usepackage{subcaption}
\usepackage{multirow}
\usetikzlibrary{arrows.meta,positioning,patterns,calc}

\mathtoolsset{showonlyrefs=true}

\theoremstyle{plain}
\newtheorem{theorem}{Theorem}
\newtheorem{proposition}{Proposition}
\newtheorem{lemma}{Lemma}
\newtheorem{corollary}{Corollary}
\theoremstyle{remark}
\newtheorem{remark}{Remark}

\newtheorem{assumption}{Assumption}
\newtheorem{definition}{Definition}

\newcommand{\cmark}{\ding{51}}
\newcommand{\method}{GANICE}

\DeclareMathOperator{\diam}{diam}
\DeclareMathOperator{\TV}{TV}
\DeclareMathOperator*{\argmin}{arg\,min}
\DeclareMathOperator*{\argmax}{arg\,max}

\newcommand{\dd}{\mathrm{d}}
\newcommand{\E}{\mathbb{E}}
\newcommand{\Prob}{\mathbb{P}}
\newcommand{\R}{\mathbb{R}}
\newcommand{\1}{\boldsymbol{1}}
\newcommand{\cF}{\mathcal{F}}
\newcommand{\cG}{\mathcal{G}}
\newcommand{\cH}{\mathcal{H}}
\newcommand{\cM}{\mathcal{M}}
\newcommand{\cP}{\mathcal{P}}
\newcommand{\cT}{\mathcal{T}}
\newcommand{\cW}{\mathcal{W}}
\newcommand{\cX}{\mathcal{X}}
\newcommand{\cY}{\mathcal{Y}}
\newcommand{\Lip}{\mathrm{Lip}}
\newcommand{\Unif}{\mathrm{Unif}}
\newcommand{\indep}{\perp\!\!\!\perp}
\newcommand{\eW}{\mathrm{eW}}
\newcommand{\cont}{\mathrm{cont}}
\newcommand{\disc}{\mathrm{disc}}
\newcommand{\obs}{\mathrm{obs}}

\title{Extended Wasserstein-GAN Approach to Causal Distribution Learning: Density-Free Estimation and Minimax Optimality}

\author{%
  Shu Tamano$^{1,2}$
  \quad
  Masaaki Imaizumi$^{1,3,4}$
  \vspace{1em}
  \\
  $^{1}$ The University of Tokyo
  \\
  $^{2}$ Japan Institute for Health Security
  \\
  $^3$ RIKEN Center for Advanced Intelligence Project
  \\
  $^{4}$ Kyoto University
  \vspace{1em}
  \\
  \texttt{tamano-shu212@g.ecc.u-tokyo.ac.jp, imaizumi@g.ecc.u-tokyo.ac.jp}
}

\begin{document}

\maketitle

\begin{abstract}
    Distributional causal inference requires estimating not only average treatment effects but also interventional outcome distributions, including quantiles, tail risks, and policy-dependent uncertainty.
    As a method for distributional causal inference, generative adversarial network (GAN)-based counterfactual methods are flexible tools for this task.
    However, these methods have several limitations.
    First, the objectives of certain techniques do not coincide with the statistical risk of the identifiable causal target, and therefore provide limited theoretical guarantees regarding estimable counterfactual distributions or optimality.
    Second, implementing target-design risks under observational sampling can require density-ratio estimation.
    In this paper, we propose GANICE (GAN for Interventional Conditional Estimation) with several advantages:
    it
    (i) specifies the conditional interventional distribution for each treatment--covariate state as the causal estimation target;
    (ii) estimates the conditional distribution such that its averaged Wasserstein risk is minimized;
    (iii) attains minimax optimality.
    GANICE achieves these advantages through the extended Wasserstein distance and cell-normalized estimation, with minimax bounds based on Wasserstein GAN oracle inequalities and an anisotropic Assouad construction;
    Besov theory characterizes the cellwise critics.
    Our experiments show that GANICE outperforms existing methods on most benchmarks.
\end{abstract}

\section{Introduction}
\label{sec:introduction}

Causal inference often targets average treatment effects, while many scientific and policy questions are distributional.
Classical and recent work on counterfactual distributions, quantile treatment effects, conditional distributional treatment effects, and semiparametric counterfactual density estimation shows that such tasks require more than conditional means \citep{firpo2007efficient,chernozhukov2013inference,park2021conditional,kennedy2023semiparametric,kallus2023robust,byambadalai2024estimating,naf2026causal,jain2026conditional}.
Under standard potential-outcomes assumptions, the identifiable target of a causal generative model is the family of conditional interventional outcome distributions, one for each treatment--covariate state.
This family can be identified from observed outcomes, but the joint distribution of all potential outcomes cannot be identified without additional assumptions.

Generative adversarial networks (GANs) are promising for this problem because they can represent complex, non-Gaussian, and structured outcome distributions.
For treatment-effect estimation, GANITE introduced adversarial counterfactual imputation for discrete treatments \citep{yoon2018ganite}, SCIGAN extended this to continuous interventions \citep{bica2020estimating}, and MCGAN generalized it to binary, categorical, and continuous treatments \citep{ge2020conditional}.
Additionally, GAD uses adversarial deconfounding to learn balancing weights \citep{li2020continuous}.
Further models address sequential treatments and dosages, compromised nets, and mediation \citep{huan2024conditional,norimatsu2025edts,abdisa2026individualized,zhang2026gahmn}.

Despite this progress, the statistical foundations of GAN-based causal distribution learning remain limited.
First, existing adversarial objectives for causal inference (e.g., \citep{yoon2018ganite,bica2020estimating}) are algorithmically well defined, but their population values are rarely formulated as risks for the identifiable causal target.
As a result, these objectives neither characterize the estimable counterfactual distributions nor guarantee optimality for the identifiable interventional distribution.
Second, exact identification of a target-design interventional conditional distribution is naturally expressed through overlap reweighting.
Direct empirical implementation may require density-ratio or propensity-score estimation, which can be unstable in continuous or high-dimensional treatment--covariate spaces.
Similar limitations apply beyond adversarial methods.
While powerful, flow- and diffusion-based methods such as INFs \citep{melnychuk2023normalizing}, DiffPO \citep{ma2024diffpo}, and PO-Flow \citep{wu2025po} lack minimax guarantees for the target-design-averaged causal risk and typically require tractable densities, invertibility, or conditional score modeling.

To address these issues, we propose \method\ (Generative Adversarial Network for Interventional Conditional Estimation), an extended Wasserstein GAN for causal distribution learning.
It has three statistical advantages.
(i) It specifies the estimable causal target:
the generator learns the conditional interventional outcome law for each treatment--covariate state.
(ii) It estimates the conditional distribution such that its averaged Wasserstein risk is minimized.
The stratified estimator normalizes within treatment--covariate cells and aggregates their discrepancies under the target design, avoiding observational density, propensity score, or overlap density-ratio estimation.
(iii) It attains minimax-optimal convergence rates, up to logarithmic factors, for the resulting causal distributional risk.
The theory is density-free in the outcome space:
it uses latent pushforward models and Wasserstein risk, with no Lebesgue density, likelihood, or conditional score.
Overlap and latent-generator smoothness are retained;
continuous conditioning additionally requires Wasserstein regularity across states.

\method\ achieves these advantages via three theoretical designs.
First, the extended Wasserstein distance \citep{chemseddine2025conditional} uses a diagonal coupling to compare outcomes only at identical treatment--covariate states, aligning the loss causally and preventing cross-state transport.
Second, its dual yields finite-resolution cellwise outcome-Lipschitz critics. These compare intra-cell outcome laws, while cell-normalized training preserves target-design aggregation without estimating density ratios.
Third, we develop minimax theory for this objective.
Under finite conditioning, we statewise-extend the Wasserstein GAN (WGAN) construction of \citep{stephanovitch2024wasserstein} using \citep{tang2023minimax}'s lower-bound mechanism. Under continuous conditioning, we combine finite-resolution localization \citep{nobel1996histogram,sart2017estimating,li2022minimax,bilodeau2023minimax}, anisotropic Wasserstein regularity, and Besov control of discontinuous critics \citep{suzuki2021deep}.
Table~\ref{tab:intro-comparison} contrasts representative causal GANs;
see Appendix~\ref{app:detailed-related-work} for detailed related work.

\begin{table}[tb]
    \centering
    \caption{
    Comparison with representative GAN-based causal inference methods.
    Cond. Dist. marks explicit conditional-distribution targets, and No Ratio Est. marks methods avoiding density-ratio or propensity-score estimation.
    CE denotes task-specific adversarial cross-entropy objective;
    $W_1$ and $\eW_1$ denote Wasserstein-1 and extended Wasserstein distances.
    Empty entries indicate unspecified properties.
    }
    \label{tab:intro-comparison}
    \begin{tabular}{lllccc}
        \toprule
        Regime     & Study &
        Cond. Dist. & No Ratio Est. & Metric & Optimality \\

        \midrule\midrule
        (Baseline) & WGAN \citep{stephanovitch2024wasserstein} & --
        & -- & $W_1$ & \cmark \\

        \midrule
        Disc.      & GANITE \citep{yoon2018ganite} &
        &  & CE &  \\

        Disc.      & MCGAN \citep{ge2020conditional} & 
        \cmark &  & CE &  \\

        Disc.      & ITE-CAN \citep{abdisa2026individualized} & 
        \cmark &  & CE &  \\

        Disc.      & \method\ (ours) &
        \cmark & \cmark & $\eW_1$ & \cmark \\

        \midrule
        Cont.      & SCIGAN \citep{bica2020estimating} &
        &  & CE &  \\

        Cont.      & MCGAN \citep{ge2020conditional} & 
        \cmark &  & CE &  \\

        Cont.      & GAD \citep{li2020continuous} &
        & \cmark & CE &  \\

        Cont.      & \method\ (ours) &
        \cmark & \cmark & $\eW_1$ & \cmark \\
        \bottomrule
    \end{tabular}
\end{table}

Our contributions are as follows:
\begin{itemize}
    \item
    We characterize the conditional marginals identified by factual-coordinate adversarial imputation and formulate an adversarial objective equal to their target-design-averaged Wasserstein risk.
    \item
    Using the one-state WGAN construction of \citep{stephanovitch2024wasserstein}, we establish finite-state minimax rates under observational sampling and target-design shift.
    \item
    For continuous conditioning, we prove a Wasserstein perturbation oracle, a comparison between cell-normalized training and exact causal risk, and a matching anisotropic Assouad lower bound. A vector-valued Besov embedding characterizes the localized critics.
\end{itemize}

\subsection{Basic Notation}
\label{subsec:basic-notation}

Probability measures reside on standard Borel spaces, and conditional distributions are fixed regular versions.
Write $\cP(S)$ for Borel probability measures on $S$, $\mathcal{L}(Z)$ for the law of a random element $Z$, and $g_\#\xi$ for the pushforward of $\xi$ by measurable $g$.
For $U\sim\Unif([0,1]^{d_U})$, write $g_\#U=g_\#\mathcal{L}(U)$; $d_U$ is the latent dimension.
$W_1$ denotes Wasserstein-1 distance on $\cY$.
We use $\|\cdot\|$, $\diam(S)$, $\delta_s$, and $\Lip(f)$ for the Euclidean norm, diameter, Dirac measure at $s$, and optimal Lipschitz constant.
Write $\mathbb{N}_0=\{0,1,\ldots\}$, $a\vee b=\max\{a,b\}$, and $\1\{\cdot\}$ for indicators.
$A\lesssim B$ means $A\le CB$ for a model-parameter-dependent constant $C$, $A\asymp B$ means both comparisons hold, and $\tilde{O}(\cdot)$ hides logarithmic factors.

Let $X\in\cX$ be covariates with law $P_X$, $T\in\cT$ treatment, and $Y(t)\in\cY\subset\R^p$ the potential outcome under $t$.
For fixed $K_0>0$, the theoretical outcome space is the closed Euclidean ball $\cY=B^p(0,2K_0)$.
For an integer $n\ge2$, the i.i.d. observations $(W_i,Y_i^{\obs})$, $i=1,\ldots,n$, are drawn from $P_{\obs}$, with state $W=(X,T)\in\cW$ and observed outcome $Y^{\obs}$.
The conditioning space is finite $\cW=\{1,\ldots,M\}$ or continuous $\cW=[0,1]^{d_W}$; $p,d_W$ are the outcome and conditioning dimensions.
Let $Q_{\obs},Q_\rho$ be the observational and fixed known target marginals (designs) of $W$, and $w_\rho=\dd Q_\rho/\dd Q_{\obs}$ their density ratio when it exists.
For $w=(x,t)$ and Borel $A\subset\cY$, the conditional interventional law is $\mu_w^\ast(A)=\Prob(Y(t)\in A\mid X=x)$.
$P_{\obs}$ and $P^\ast$ denote the observational and target interventional joint laws.
For a generator $g:\cW\times[0,1]^{d_U}\to\cY$, the generated conditional law is $\nu_{g,w}=g(w,\cdot)_\#U$, yielding $P_g(\dd w,\dd y)=Q_\rho(\dd w)\nu_{g,w}(\dd y)$.

For a partition $\Pi$ of $\cW$ with cells $C\in\Pi$, its target and observational masses are $q_C=Q_\rho(C)$ and $\pi_C=Q_{\obs}(C)$.
For finite $\cW$, we abbreviate $q_j=Q_\rho(\{j\})$, $\pi_j=Q_{\obs}(\{j\})$, and $\mu_j^\ast=\mu_{w=j}^\ast$.
For an anisotropic dyadic partition $\Pi_{\boldsymbol{m}}$, $C_{\boldsymbol{m}}(w)$ is the cell containing $w$.
Appendix~\ref{app:notations} summarizes notation.

\section{Problem Setup}
\label{sec:problem-setup}

We work in the potential-outcomes framework \citep{neyman1923sur,rubin1974estimating}.
Our goal is to estimate the target interventional law $P^\ast$, equivalently the family $\{\mu_w^\ast\}_{w\in\cW}$ under target weighting $Q_\rho$.
All equalities involving conditional laws are interpreted up to the relevant design marginal.

\begin{assumption}[Consistency]
\label{ass:consistency}
    $Y^{\mathrm{obs}}=Y(T)$ almost surely (a.s.).
\end{assumption}

\begin{assumption}[Ignorability]
\label{ass:ignorability}
    There exist a standard Borel random element $V$ and a jointly measurable map $F$ such that $Y(t)=F(X,t,V)$ simultaneously for all $t$ almost surely, and $V\indep T\mid X$.
\end{assumption}

Use the version $\mu_{(x,t)}^\ast=F(x,t,\cdot)_\#\mathcal{L}(V\mid X=x)$.
Under Assumptions~\ref{ass:consistency}--\ref{ass:ignorability}, the observed conditional distribution of $Y^{\obs}$ given $W=w$ is $\mu_w^\ast$;
see Appendix~\ref{app:problem-setup-proofs}.
Hence, $P_{\obs}(\dd w,\dd y)=Q_{\obs}(\dd w)\mu_w^\ast(\dd y)$ and $P^\ast(\dd w,\dd y)=Q_\rho(\dd w)\mu_w^\ast(\dd y)$.

\begin{assumption}[Overlap]
\label{ass:overlap-general}
    $Q_\rho\ll Q_{\obs}$, and $0\le w_\rho(w)\le\kappa^{-1}$ $Q_{\obs}$-a.s. for some $\kappa\in(0,1]$.
\end{assumption}

For finite $\cW=\{1,\ldots,M\}$, overlap reduces to $q_j\le\kappa^{-1}\pi_j$ for every state $j$.
Assumption~\ref{ass:overlap-general} implies
$
\E_{P_{\mathrm{obs}}}[
    w_\rho(W)\varphi(W,Y^{\mathrm{obs}})
]
=
\E_{P^\ast}[\varphi(W,Y)]
$
for every bounded measurable $\varphi:\cW\times\cY\to\R$.
This identity is the causal identification step;
the remaining analysis is statistical estimation of the conditional distribution.

\subsection{What Existing Counterfactual GANs Identify}
\label{subsec:role-existing-gan}

Many adversarial causal generators form a completed potential-outcome vector $\bar{\boldsymbol{Y}}=(\bar{Y}(1),\ldots,\bar{Y}(A))$ for finite treatments $\cT=\{1,\ldots,A\}$ by imputing missing coordinates alongside the factual outcome, then train a discriminator to identify the factual treatment \citep{yoon2018ganite,ge2020conditional,huan2024conditional}.
Given a generator, the population discriminator estimates $r_a(x,\bar{y})=\Prob(T=a\mid X=x,\bar{\boldsymbol{Y}}=\bar{y})$.
Optimizing the generator drives the conditional distribution of $\bar{\boldsymbol{Y}}$ toward $T$-invariance given $X$.
Under exact factual reconstruction ($\bar{Y}(T)=Y^{\obs}$ a.s.), this invariance yields $\mathcal{L}(\bar{Y}(a)\mid X=x)=\mathcal{L}(Y(a)\mid X=x)$ $P_X$-almost surely for any treatment $a$ with positive propensity.
Thus, this approach identifies treatment-specific conditional marginals but not their conditional joint law: any measurable coupling of these marginals is compatible with the observed law, factual reconstruction, and invariance (see Appendix~\ref{app:problem-setup-proofs} for a formal proof).

\section{Proposed Method}
\label{sec:proposed-method}

\subsection{Preparation: Extended Wasserstein Distance}
\label{subsec:extended-wasserstein}

This subsection defines the distributional loss used to train the generator.
Since the identifiable target is the family of conditional interventional distributions, the loss must compare outcome distributions at the same treatment--covariate state.

Let $\mu_w,\nu_w$ be the conditional outcome kernels of $P,R\in\cP(\cW\times\cY)$.
Write $\Gamma_\Delta(P,R)$ for couplings of $P,R$ supported on $\{w=w'\}$.

\begin{definition}[Extended Wasserstein distance \citep{chemseddine2025conditional}]
\label{def:extended-wasserstein}
    For laws $P,R$ with common conditioning marginal $Q$, the extended Wasserstein-1 distance is
    \begin{equation}
    \label{eq:def-extended-w1}
        \eW_1(P,R)
        \coloneq
        \inf_{\alpha\in\Gamma_\Delta(P,R)}
        \int
        \|y-y'\|\,
        \alpha(\dd w,\dd y,\dd w',\dd y')
        .
    \end{equation}
\end{definition}

Formally, for a coupling $\alpha\in\Gamma_\Delta(P,R)$ on $(\cW\times\cY)^2$, the two state coordinates must satisfy $W'=W$ almost surely;
equivalently, the $(W,W')$-marginal is the pushforward $\Delta_\#Q$ under the diagonal map $\Delta(w)=(w,w)$.
This diagonal constraint is the causal restriction:
an outcome observed at state $w$ can be matched only to a generated outcome at the same $w$.
On standard Borel $\cW$ and compact $\cY$, it yields the statewise identity $\eW_1(P,R) = \int_{\cW} W_1(\mu_w,\nu_w)\,Q(\dd w)$.
Thus, $\eW_1$ is exactly the conditioning-marginal average of ordinary outcome-space Wasserstein distances;
details are in Appendix~\ref{app:extended-wasserstein-details}.

For the dual representation used below, fix an anchor $y_0\in\cY$ and define $\mathcal{L}_1\coloneq\{f:\cY\to\R:f(y_0)=0,\ \Lip(f)\le1\}$.
The critic class is defined as
$
    \cF
    \coloneq
    \{
        h:\cW\times\cY\to\R:
        h\text{ is jointly measurable},\
        h(w,\cdot)\in\mathcal{L}_1
        \text{ for }Q\text{-a.e. }w
    \}
$.
Anchoring loses nothing, since adding any $c(w)$ to $h$ cancels under the common $Q$, and compactness gives $|h(w,y)|\le\diam(\cY)$ for every $y$ and $Q$-almost every $w$.
For causal risks, $\cF$ uses $Q=Q_\rho$.
Extended Kantorovich--Rubinstein duality gives
\begin{equation}
    \label{eq:extended-w1-dual}
    \eW_1(P,R)
    =
    \sup_{h\in\cF}\{\E_P [h(W,Y)]-\E_R [h(W,Y)]\}
    .
\end{equation}

The constraint in $\eW_1$ is stronger than ordinary joint optimal transport.
If $\cW$ has a bounded metric $d_\cW$ and $W_1^\otimes$ denotes Wasserstein-1 distance on $\cW\times\cY$ with product cost $d_\cW(w,w')+\|y-y'\|$, then every diagonal coupling is an admissible joint coupling, yielding
\begin{equation}
\label{eq:ew-dominates-joint-w1}
    W_1^\otimes(P,R)
    \le
    \eW_1(P,R)
    .
\end{equation}
Hence, $\eW_1$ convergence implies joint Wasserstein convergence, while forbidding transport across distinct causal states.

\subsection{GANICE: GAN for Interventional Conditional Estimation}
\label{subsec:conditional-estimator}

This subsection presents the proposed method, \method, and introduces the population adversarial identity along with the stratified finite-resolution estimators.  
The population identity uses overlap reweighting to identify the exact causal risk, whereas the empirical estimators avoid estimating $w_\rho$ by performing normalization within cells.

For a critic $h$, define
$L(g,h)\coloneq\E_{P_{\obs}}[w_\rho(W)h(W,Y^{\obs})]-\E[h(\tilde{W},g(\tilde{W},U))]$,
where $\tilde{W}\sim Q_\rho$ independently of $U$.
By the identity in Section~\ref{sec:problem-setup} and \eqref{eq:extended-w1-dual},
\begin{equation}
\label{eq:population-identity-short}
    \sup_{h\in\cF} L(g,h)
    =
    \eW_1(P^\ast,P_g)
    .
\end{equation}
Thus the exact population adversarial risk is the extended Wasserstein distance to the identifiable causal target.

The empirical estimator uses partition-based conditional measures, normalizing each cell by its sample count and weighting by its target mass, following the same principle as histogram- and partition-based estimators \citep{nobel1996histogram,sart2017estimating,li2022minimax,bilodeau2023minimax}, with partitions defined as singletons for finite conditioning and anisotropic dyadic cells for continuous conditioning (see Appendix~\ref{app:dyadic-partition-details}).
For a finite measurable partition $\Pi$ of $\cW$ and a cell $C\in\Pi$, let $N_C^{\obs}=\sum_{i=1}^n\1\{W_i\in C\}$ be the observed count in $C$.
For $f:\cY\to\R$, define
\begin{equation}
\label{eq:def-cell-obs-average}
    \hat{E}_C^{\obs}[f]
    \coloneq
    \begin{cases}
        (N_C^{\obs})^{-1}
        \sum_{i:W_i\in C}
        f(Y_i^{\obs})
        ,
        & N_C^{\obs}
        >
        0
        ,
        \\
        f(y_0)
        ,
        & N_C^{\obs}
        =
        0
        .
    \end{cases}
\end{equation}
This cell-normalized average is the basic object that removes the need to estimate the density ratio.

\paragraph{Finite conditioning.}
When $\cW=\{1,\ldots,M\}$, take $\Pi=\{\{1\},\ldots,\{M\}\}$ and use $j$ to index the singleton cell $\{j\}$.
Let $\cG_N$ and $\mathcal{D}_N$ be the generator and critic classes from the unconditional WGAN estimator of \citep{stephanovitch2024wasserstein}, indexed by a deterministic sieve parameter $N\ge2$.
In the finite-conditioning setting, we take $N=n$ and set
\begin{equation*}
    \cG_n^{\disc}
    =
    \{g(j,u)=g_j(u):g_j\in\cG_n,\ \forall j\}
    ,
    \quad
    \cH_n^{\disc}
    =
    \{h(j,y)=D_j(y):D_j\in\mathcal{D}_n,\ \forall j\}
    .
\end{equation*}
Let $\{U_{\ell j}: \ell=1,\ldots, n, j=1,\ldots,M\}$ be i.i.d. copies of the noise variable $U$, independent of the observations.
For $h(j,y)=D_j(y)$, define $\hat{L}_n^{\disc}(g,h)$ and the finite-state \method\ estimator $\hat{g}_n^{\disc}$ as follows:
\begin{equation}
\label{eq:discrete-stratified-objective-main}
    \hat{L}_n^{\disc}(g,h)
    \coloneq
    \sum_{j=1}^M
    q_j
    \Biggl[
        \hat{E}_j^{\obs}[D_j]
        -
        \frac{1}{n}
        \sum_{\ell=1}^n
        D_j(g_j(U_{\ell j}))
    \Biggr]
    ,
    \quad
    \hat{g}_n^{\disc}
    \in
    \argmin_{g\in\cG_n^{\disc}}
    \sup_{h\in\cH_n^{\disc}}
    \hat{L}_n^{\disc}(g,h)
    .
\end{equation}
In finite states, stratification is exact:
conditional on $W=j$, observed outcomes are sampled from $\mu_j^\ast$, and $q_j$ only weights the statewise risks.

\paragraph{Continuous conditioning.}
When $\cW=[0,1]^{d_W}$, the exact class $\cF$ is too rich because its Kantorovich potential may vary freely with $w$.
We therefore localize critics on treatment--covariate cells while allowing the generator to remain a general conditional model.

Let $\boldsymbol{m}=(m_1,\ldots,m_{d_W})\in\mathbb{N}_0^{d_W}$ and $|\boldsymbol{m}|_1=\sum_jm_j$.
The anisotropic dyadic partition $\Pi_{\boldsymbol{m}}$ divides $[0,1]^{d_W}$ into axis-aligned rectangles with side length $2^{-m_j}$ in coordinate $j$.
The population finite-resolution critic class is
$\cF^{(\boldsymbol{m})}
\coloneq
\{
    h(w,y)=\sum_{C\in\Pi_{\boldsymbol{m}}}\1\{w\in C\}f_C(y):
    f_C\in\mathcal{L}_1
\}$.
It defines the cell-resolution version of $\eW_1$:
one outcome-Lipschitz potential is chosen per cell, so the induced loss compares cell-averaged outcome laws and averages the resulting $W_1$ distances using the target masses $q_C$.
For estimation, let $N_{\boldsymbol{m}}\coloneq\lceil\kappa n2^{-|\boldsymbol{m}|_1}\rceil\vee2$ be the deterministic effective cell sample size.
Using the one-state outcome-critic class $\mathcal{D}_{N_{\boldsymbol{m}}}$ from Appendix~\ref{app:one-state-sieve}, set
\begin{equation}
\label{eq:continuous-restricted-critic-main}
    \cH_{n,\boldsymbol{m}}^{\cont}
    \coloneq
    \bigl\{
        h(w,y)
        =
        \sum_{C\in\Pi_{\boldsymbol{m}}}
        \1\{w\in C\}D_C(y):
        D_C\in\mathcal{D}_{N_{\boldsymbol{m}}}
    \bigr\}
    .
\end{equation}
Thus $\cF^{(\boldsymbol{m})}$ defines the finite-resolution loss geometry, and $\cH_{n,\boldsymbol{m}}^{\cont}$ is its regularized empirical critic class.
Let $\cG_n^{\cont}$ be the implemented conditional generator class, with the measurability conditions specified in Appendix~\ref{app:single-network-transfer-proof}.

Draw $\tilde{W}_1,\ldots,\tilde{W}_n$ i.i.d. from $Q_\rho$ and independent copies $U_1,\ldots,U_n$ of $U$, with both samples independent of the observations and of each other.
For $C\in\Pi_{\boldsymbol{m}}$, let $M_C^\rho\coloneq \sum_{\ell=1}^n\1\{\tilde{W}_\ell\in C\}$ be the target-design count in $C$ and define the generated cell average for $f:\cY\to\R$:
\begin{equation}
\label{eq:def-cell-generated-average}
    \hat{E}_C^g[f]
    \coloneq
    \begin{cases}
        (M_C^\rho)^{-1}
        \sum_{\ell:\tilde{W}_\ell\in C}
        f(g(\tilde{W}_\ell,U_\ell))
        ,
        & M_C^\rho
        >
        0
        ,
        \\
        f(y_0)
        ,
        & M_C^\rho
        =
        0
        .
    \end{cases}
\end{equation}
For $h\in\cH_{n,\boldsymbol{m}}^{\cont}$ with cell components $D_C$, the continuous stratified objective $\hat{L}_{n,\boldsymbol{m}}^{\cont}(g,h)$ and the continuous-conditioning \method\ estimator $\hat{g}_{n,\boldsymbol{m}}^{\cont}$ are defined as follows:
\begin{equation}
\label{eq:continuous-estimator-main}
    \hat{L}_{n,\boldsymbol{m}}^{\cont}(g,h)
    \coloneq
    \sum_{C\in\Pi_{\boldsymbol{m}}}
    q_C
    \bigl[
        \hat{E}_C^{\obs}[D_C]
        -
        \hat{E}_C^g[D_C]
    \bigr]
    ,
    \quad
    \hat{g}_{n,\boldsymbol{m}}^{\cont}
    \in
    \argmin_{g\in\cG_{n}^{\cont}}
    \sup_{h\in\cH_{n,\boldsymbol{m}}^{\cont}}
    \hat{L}_{n,\boldsymbol{m}}^{\cont}(g,h)
    .
\end{equation}
The empirical objective learns from observed outcomes within cells and aggregates cellwise discrepancies under $Q_\rho$, without estimating $w_\rho$.

\section{Minimax Optimality: Discrete Conditioning}
\label{sec:discrete-minimax}

We first consider finite conditioning.
In this setting, stratification is exact:
conditional on $W=j$, observed outcomes are sampled from $\mu_j^\ast$, and no density-ratio estimate is needed.
Using the finite-state notation from Section~\ref{subsec:basic-notation},
\begin{equation}
\label{eq:finite-disintegration-main}
    \eW_1(P^\ast,P_g)
    =
    \sum_{j=1}^M
    q_j W_1(\mu_j^\ast,\nu_{g,j})
    .
\end{equation}
Thus the causal distributional risk decomposes into $M$ local unconditional generative estimation problems, aggregated under the target design $Q_\rho$.

For $s>0$, $\mathcal{H}_{K_0}^{s}$ is the periodic H\"older ball on the unit torus $\mathbb{T}^{d_U}=\mathbb{R}^{d_U}/\mathbb{Z}^{d_U}$: maps into $B^p(0,K_0)$ have derivatives through order $\lfloor s\rfloor$ bounded by $K_0$ and, for noninteger $s$, top-order H\"older seminorm of order $s-\lfloor s\rfloor$ at most $K_0$.
Throughout the rate results, $\beta\ge1$ and $1\le d_U\le p$.
For fixed $(M,q_{\min},\kappa,\beta,d_U,p,K_0)$, let $\cM_{\disc}$ be the class of finite-state causal models satisfying Assumptions~\ref{ass:consistency}--\ref{ass:overlap-general}, $q_j\ge q_{\min}$ with $0<q_{\min}\le1/M$, $\mu_j^\ast=(g_j^\ast)_\#U$, and $g_j^\ast\in\mathcal{H}_{K_0}^{\beta+1}$ for all $j=1,\ldots,M$.
The generator and critic classes are obtained by applying the local WGAN construction of \citep{stephanovitch2024wasserstein} independently to each state.

\begin{theorem}[Finite-state upper bound]
\label{thm:discrete-upper-bound}
    Fix $M<\infty$ and define $a\coloneq\min\{(\beta+1)/(2\beta+d_U),1/2\}$.
    There exist constants $C,c>0$, depending only on $(M,q_{\min},\kappa,\beta,d_U,p,K_0)$, such that
    \begin{equation}
        \label{eq:discrete-upper-bound}
        \sup_{(P_{\mathrm{obs}},P^\ast)\in\cM_{\disc}}
        \E\bigl[
            \eW_{1}(P^\ast,P_{\hat{g}_n^{\disc}})
        \bigr]
        \le
        C(\log n)^c n^{-a}
        .
    \end{equation}
\end{theorem}
The exponent $a$ is the local one-state WGAN exponent:
$(\beta+1)/(2\beta+d_U)$ is the nonparametric latent-pushforward rate determined by smoothness and latent dimension, capped by the parametric barrier $1/2$.
Because $M$ is fixed, finite conditioning does not change the exponent;
overlap only ensures enough observations in each target-relevant state.

Define the minimax risk as
$\mathcal{R}_n^{\disc}=\inf_{\hat{P}_n}\sup_{(P_{\obs},P^\ast)\in\cM_{\disc}}\E[\eW_1(P^\ast,\hat{P}_n)]$,
where the infimum is taken over all estimators based on the observational sample that output laws with conditioning marginal $Q_\rho$.

\begin{corollary}[Finite-state minimax optimality]
\label{cor:discrete-minimax-optimality}
    There exist constants $0<c_L<C_U<\infty$ and $c>0$ such that, for all $n$,
    $c_L n^{-a}\le\mathcal{R}_n^{\disc}\le C_U(\log n)^c n^{-a}$.
\end{corollary}
The rate is obtained in the causal loss $\eW_1$, hence it controls the target-weighted statewise error $\sum_j q_jW_1(\mu_j^\ast,\nu_{\hat{g}_n^{\disc},j})$ and, by \eqref{eq:ew-dominates-joint-w1}, ordinary joint Wasserstein distance.
This is a minimax statement for the identifiable interventional distribution:
no estimator can uniformly recover the causal distribution faster under the same identification and overlap conditions.
By Kantorovich--Rubinstein duality, it also controls target-weighted errors of all statewise 1-Lipschitz distributional functionals; quantile and tail summaries follow under the usual functional-specific regularity conditions.

\begin{remark}[Statewise and joint transport]
\label{rem:state-swap-main}
    Consider scalar outcomes with $K_0\ge1$, two states at distance $\varepsilon\in(0,1)$, and $Q_\rho=(\delta_{w_1}+\delta_{w_2})/2$.
    Let $\mu_{w_1}^\ast=\delta_0$ and $\mu_{w_2}^\ast=\delta_1$, and exchange these conditional laws in $R$.
    Then $\eW_1(P^\ast,R)=1$ and $W_1^\otimes(P^\ast,R)=\varepsilon$.
    Thus a small joint transport cost can conceal a large statewise error as state separation decreases.
    For a fixed finite state metric, a reverse comparison holds with a constant depending on its minimum separation (Appendix~\ref{app:extended-wasserstein-details}); our statewise bound is independent of that metric.
\end{remark}

\subsection{Proof Sketch}
\label{subsec:discrete-proof-sketch}

The proof reduces the risk to statewise WGAN risks via \eqref{eq:finite-disintegration-main}.
The first key point is that finite-state stratification preserves the target: conditioning on $W=j$ yields samples from $\mu_j^\ast$, while $q_j$ weights the resulting risks.
Thus, overlap lower-bounds local sample sizes ($\pi_j\ge\kappa q_j$) rather than building inverse-probability-weighted empirical processes.

The second key point is combining random state counts with \citep{stephanovitch2024wasserstein}.
Given these counts, each state is an unconditional WGAN; binomial negative-moment bounds convert them into deterministic rates.
The restricted critic construction yields the latent-pushforward rate $n^{-a}$ by controlling the empirical process through the latent dimension.
The lower bound restricts to a submodel with $Q_{\obs}=Q_\rho$ in which only one state's conditional law varies, directly transferring the unconditional lower bound of \citep{tang2023minimax,stephanovitch2024wasserstein} to $\eW_1$.
Full details are in Appendix~\ref{app:discrete-proofs}.

\section{Minimax Optimality: Continuous Conditioning}
\label{sec:continuous-conditioning}

We next analyze the continuous-conditioning estimator \eqref{eq:continuous-estimator-main}.
It remains stratified and cell-normalized: $w_\rho$ is never estimated, and overlap only controls the number of observations in target-relevant cells.

\begin{assumption}[Design regularity]
\label{ass:design-regularity-main}
    In the continuous-conditioning setting, $Q_\rho$ has a density $q_\rho$ with respect to Lebesgue measure satisfying $0<\underline{q}\le q_\rho(w)\le\bar{q}<\infty$ for Lebesgue-almost every $w$.
\end{assumption}
Assumption~\ref{ass:design-regularity-main} is a standard covariate-design condition in local nonparametric conditional estimation \citep{fan1996estimation,hall1999methods,cattaneo2024boundary}.
It is imposed on the target design, not on the outcome law:
the conditional outcome distributions may still be implicit or singular.
Together with Assumption~\ref{ass:overlap-general}, it gives expected observational cell counts of order at least $\kappa n2^{-|\boldsymbol{m}|_1}$ in target-relevant cells.

\begin{assumption}[Pointwise latent pushforward model]
\label{ass:pointwise-pushforward-main}
    For every $w\in\cW$, $\mu_w^\ast=(g_w^\ast)_\#U$ for some $g_w^\ast\in\mathcal{H}_{K_0}^{\beta+1}$.
\end{assumption}
This is a density-free local model for the outcome distribution and allows $\mu_w^\ast$ to be singular with respect to Lebesgue measure.

\begin{assumption}[Anisotropic conditional regularity]
\label{ass:anisotropic-holder-main}
    There exist $\alpha_1,\ldots,\alpha_{d_W}\in(0,1]$ and $0<L_\ast<\infty$ such that $W_1(\mu_w^\ast,\mu_{w'}^\ast)\le L_\ast\sum_{j=1}^{d_W}|w_j-w_j'|^{\alpha_j}$ for all $w,w'\in\cW$.
\end{assumption}
Assumption~\ref{ass:anisotropic-holder-main} is the Wasserstein analogue of covariate-smoothness conditions used in conditional density and distribution estimation \citep{li2022minimax,tang2025conditional,hu2025statistical}.
It is stated directly in $W_1$, so it controls finite-resolution localization bias without requiring an outcome density or score.

For $\boldsymbol{\alpha}=(\alpha_1,\ldots,\alpha_{d_W})$, the localization bias scale, effective anisotropic dimension, and rate exponent are $b_{\boldsymbol{m}}\coloneq\sum_{j=1}^{d_W}2^{-\alpha_jm_j}$, $\bar{d}\coloneq\sum_{j=1}^{d_W}\alpha_j^{-1}$, $r\coloneq(a^{-1}+\bar{d})^{-1}$.
The optimized resolution $\boldsymbol{m}_n=(m_{n,j})_{j=1}^{d_W}$ has $m_{n,j}\coloneq \lfloor r\log_2(\kappa n)/\alpha_j\rfloor\vee0$.
The exponent $r$ captures a trade-off between the local difficulty of outcome generation and the cost of conditioning.
The local difficulty is quantified by $a^{-1}$, inherited from the one-state WGAN estimation setting, while the conditioning cost is given by $\bar{d}=\sum_j\alpha_j^{-1}$.
Consequently, smaller $\alpha_j$ or an increase in the number of conditioning coordinates lead to a slower convergence rate.
The effect of weak overlap appears as the effective sample size factor $\kappa n$.

For fixed $(\kappa,\beta,d_U,p,d_W,K_0,\boldsymbol{\alpha},L_\ast,\underline{q},\bar{q})$, let $\cM_{\cont}$ denote the model class satisfying Assumptions~\ref{ass:consistency}--\ref{ass:overlap-general} and Assumptions~\ref{ass:design-regularity-main}--\ref{ass:anisotropic-holder-main}, with the specified known $Q_\rho$ and $d_W\ge1$.

For the continuous-conditioning risk bounds, write $\hat{P}_{n,\boldsymbol{m}}^{\cont}\coloneq P_{\hat{g}_{n,\boldsymbol{m}}^{\cont}}^{(\boldsymbol{m})}$, where $P_g^{(\boldsymbol{m})}$ replaces $\nu_{g,w}$ on each cell $C\in\Pi_{\boldsymbol{m}}$ by its $Q_\rho$-average (Appendix~\ref{app:continuous-geometry-proof}).

\begin{theorem}[Continuous-conditioning upper bound]
\label{thm:continuous-upper-bound}
    Assume the implemented estimator satisfies the finite-resolution transfer condition \eqref{eq:implementation-transfer-condition} uniformly over $\cM_{\cont}$ at $\boldsymbol{m}_n$, with error at most $C_0(\log n)^{c_0}(\kappa n)^{-r}$ for fixed $C_0>0,c_0\ge0$.
    Then there exist constants $C,c>0$, depending only on fixed model parameters and $C_0,c_0$, such that
    \begin{equation}
    \label{eq:continuous-optimized-upper-main}
        \sup_{(P_{\obs},P^\ast)\in\cM_{\cont}}
        \E
        \bigl[
            \eW_1(P^\ast,\hat{P}_{n,\boldsymbol{m}_n}^{\cont})
        \bigr]
        \le
        C(\log n)^c(\kappa n)^{-r}
        .
    \end{equation}
\end{theorem}

The finite-resolution inequality underlying Theorem~\ref{thm:continuous-upper-bound} is proved in Appendix~\ref{app:continuous-upper-proof}.
It decomposes the risk into localization bias and a local WGAN stochastic term with effective cell sample size $\kappa n2^{-|\boldsymbol{m}|_1}$.
Thus weak overlap reduces the effective sample size, but does not require estimating the density ratio.
The bound compares the cell-averaged generated law with the exact target $P^\ast$.
The transfer condition bounds the gap between the cell laws of $\hat{g}_{n,\boldsymbol{m}}^{\cont}$ and those of the routed estimator, which fits one generator per cell;
for the latter it is zero and the bound holds before cell averaging (Corollary~\ref{cor:raw-conditional-law-app}).

Define the minimax risk
$\mathcal{R}_n^{\cont}\coloneq\inf_{\hat{P}_n}\sup_{(P_{\obs},P^\ast)\in\cM_{\cont}}\E_{P_{\obs}^{\otimes n}}[\eW_1(P^\ast,\hat{P}_n)]$,
where the infimum is over all estimators of the form
$\hat{P}_n(\dd w,\dd y)=Q_\rho(\dd w)\hat{\mu}_{n,w}(\dd y)$.

\begin{theorem}[Continuous-conditioning lower bound]
\label{thm:continuous-lower-bound}
    There exists a constant $c_L>0$, depending only on fixed model parameters, such that for all sufficiently large $n$, $\mathcal{R}_n^{\cont}\ge c_L(\kappa n)^{-r}$.
\end{theorem}

Theorems~\ref{thm:continuous-upper-bound}--\ref{thm:continuous-lower-bound} give minimax optimality up to logarithmic factors for the routed estimator, with an additive implementation-transfer error under the statewise causal risk $\eW_1$.

\begin{remark}[Comparison with conditional diffusion rates]
\label{rem:diffusion-comparison}
    For covariate and response dimensions $D_X,D_Y$ with H\"older smoothness $\alpha_X,\alpha_Y$, \citep{tang2025conditional} obtain the TV rate $\tilde{O}(n^{-1/(2+D_X/\alpha_X+D_Y/\alpha_Y)})$ and a Wasserstein-1 rate involving the response density term $D_Y/\alpha_Y$;
    their manifold result replaces $(D_X,D_Y)$ by intrinsic dimensions and is stated in $W_1$ because the relevant laws may be mutually singular.
    \citep{hu2025statistical} obtain TV-type rates for conditional diffusion transformers under H\"older density/score assumptions.
    Our rate replaces response density complexity by the latent-pushforward WGAN exponent $a$, separates the conditioning cost as $\bar{d}$, and measures the causal statewise risk $\eW_1$.
    Hence no outcome density or conditional score is required.
\end{remark}

\subsection{Proof Sketch}
\label{subsec:continuous-proof-sketch}

The upper bound relies on a finite-resolution reduction rather than inverse-probability-weighted empirical processes.
Within a cell $C$, observations follow $\mu_C^{\obs}=\pi_C^{-1}\int_C\mu_w^\ast Q_{\obs}(\dd w)$, while the target averages under $Q_\rho$.
The first key point is the proxy-to-exact inequality:
anisotropic $W_1$ regularity ensures replacing the $Q_\rho$-cell mixture with $\mu_C^{\obs}$ incurs only localization bias.
This avoids estimating $w_\rho$ while controlling the exact causal $\eW_1$ risk.

The second key point is statistical localization combined with Besov control of the critic geometry.
Conditional on $N_C^{\obs}$ and $M_C^\rho$, each routed cell estimator is an unconditional WGAN targeting $\mu_C^{\obs}$, which lies within $L_\ast b_{\boldsymbol{m}}$ of a cell-representative law $\mu_{w_C}^\ast$ and is controlled by the perturbation oracle.
Overlap and regularity yield $n\pi_C\gtrsim\kappa n2^{-|\boldsymbol{m}|_1}$;
summing one-state oracle inequalities over cells gives the stochastic error.
The vector-valued Besov--Nikolskii embedding characterizes the jumps of the localized critics at partition boundaries \citep{triebel1983theory,suzuki2019adaptivity,suzuki2021deep}.
A single shared network adds the transfer error (Appendix~\ref{app:single-network-transfer-proof}), which vanishes when the cellwise fits are used directly.
Optimizing anisotropic resolution balances bias and stochastic error, yielding \eqref{eq:continuous-optimized-upper-main}.

The lower bound uses an anisotropic Assouad construction.
Boxes have side lengths proportional to target separation raised to $1/\alpha_j$, each containing a one-state WGAN hard family.
Smooth transport perturbations preserve the latent-generator class and anisotropic $W_1$ regularity, with Hellinger bounds uniform over the smooth cutoff transitions and local sample sizes scaling as $\kappa n$ times box volume.
Balancing local WGAN difficulty with anisotropic conditioning yields $r=(a^{-1}+\bar{d})^{-1}$.
Full details are in Appendix~\ref{app:continuous-proofs}.

\begin{remark}[Causal role of anisotropic Besov critics]
\label{rem:besov-causal-interpretation}
    Finite-resolution critics are piecewise constant in $w$ and Lipschitz in $y$.
    Appendix~\ref{app:continuous-geometry-proof} embeds the maps $w\mapsto h(w,\cdot)$ into anisotropic Besov--Nikolskii balls $B_{p_B,\infty}^{\boldsymbol{s}}$ with smoothness $s_j<1/p_B$, the range in which hyperplane jumps remain $L^{p_B}$-Besov admissible \citep{triebel1983theory,neumann1997wavelet,kerkyacharian2001nonlinear,hoffmann2002random,suzuki2019adaptivity,suzuki2021deep}.
    The embedding bounds the critic norm by $C\max_j2^{m_js_j}$ and accommodates jumps at partition boundaries while preserving the outcome Lipschitz constraint in the conditional Wasserstein dual.
\end{remark}

\section{Extension to IPM GANs}
\label{sec:ipm-extension}

The finite-resolution argument also applies to other adversarial losses.
Let $\mathcal{V}$ be a symmetric, uniformly bounded class of measurable $f:\cY\to\R$ containing zero, with a countable subclass attaining its suprema, and local integral probability metric (IPM)
$d_{\mathcal{V}}(\mu,\nu)\coloneq\sup_{f\in\mathcal{V}}\int f\dd(\mu-\nu)$.
For $P,R$ sharing $Q_\rho$, let $\mu_C^\rho=q_C^{-1}\int_C\mu_wQ_\rho(\dd w)$ and $\nu_C^\rho$ denote their target-design cell mixtures.
The finite-resolution conditional IPM is
\begin{equation}
    \label{eq:finite-resolution-ipm-main}
    d_{\mathcal{V}}^{(\boldsymbol{m})}(P,R)
    \coloneq
    \sup_{
        h(w,y)=\sum_{C\in\Pi_{\boldsymbol{m}}}\1\{w\in C\}f_C(y),\ f_C\in\mathcal{V}
    }
    \{
        \E_P
        [
            h
        ]
        -
        \E_R
        [
            h
        ]
    \}
    .
\end{equation}
This equals $\sum_{C\in\Pi_{\boldsymbol{m}}}q_Cd_{\mathcal{V}}(\mu_C^\rho,\nu_C^\rho)$.
For the routed IPM estimator $\hat{P}_{n,\boldsymbol{m}}^{\mathcal{V}}$ defined in Appendix~\ref{app:ipm-extension-proof}, a one-state perturbation oracle with exponent $a_{\mathcal{V}}>0$ and anisotropic H\"older regularity of $w\mapsto\mu_w^\ast$ under $d_{\mathcal{V}}$ with exponents $\boldsymbol{\alpha}$ give
\begin{equation}
    \label{eq:generic-ipm-main-rate}
    \E
    [
        d_{\mathcal{V}}^{(\boldsymbol{m})}
        (
            P^\ast,\hat{P}_{n,\boldsymbol{m}}^{\mathcal{V}}
        )
    ]
    \lesssim
    b_{\boldsymbol{m}}
    +
    (\log n)^c(\kappa n2^{-|\boldsymbol{m}|_1})^{-a_{\mathcal{V}}}
    .
\end{equation}
Optimizing yields exponent $r_{\mathcal{V}}=(a_{\mathcal{V}}^{-1}+\bar{d})^{-1}$.
Thus the construction covers Wasserstein critics, H\"older IPMs, MMD critics, and outcome-Besov IPMs whenever the corresponding one-state oracle inequality holds.
Appendix~\ref{app:ipm-extension-proof} gives the cell-mixture definitions and formal assumptions.

\section{Experiments}
\label{sec:experiments}

We evaluate whether \method\ estimates the full conditional interventional law, rather than only conditional means.
We use two semi-synthetic benchmarks with known conditional outcome laws and one real benchmark with randomized-trial validation:
the Infant Health and Development Program (IHDP) for binary treatment, The Cancer Genome Atlas (TCGA) for continuous dosage, and Jobs for real-world validation.
IHDP is built from the covariates and treatment assignments used in counterfactual representation learning \citep{shalit2017estimating,hill2011bayesian};
TCGA follows the high-dimensional gene-expression benchmark used for continuous-dose causal learning \citep{bica2020estimating,schwab2020learning};
Jobs combines the National Supported Work (NSW) randomized job-training experiment with Panel Study of Income Dynamics (PSID) controls \citep{lalonde1986evaluating,dehejia1999causal}.
Detailed data construction, preprocessing, target designs, and hyperparameters are given in Appendix~\ref{app:detailed-experiments}.

\paragraph{Baselines and metrics.}
For IHDP and Jobs, we compare with GANITE \citep{yoon2018ganite}, PO-Flow \citep{wu2025po}, DiffPO \citep{ma2024diffpo}, individualized normalizing flows (INFs) \citep{melnychuk2023normalizing}, DR-Learner \citep{kennedy2023towards}, CoDiTE \citep{park2021conditional}, DNet \citep{wu2023dnet}, Causal-DRF \citep{naf2026causal}, and SKCD (adapted) \citep{jain2026conditional}.
Appendix~\ref{app:baselines} describes the predictive adaptation used for SKCD.
For TCGA, we compare with SCIGAN \citep{bica2020estimating}, DRNet \citep{schwab2020learning}, and VCNet \citep{nie2021vcnet}.
For the two semi-synthetic benchmarks, the primary metric is empirical extended Wasserstein error against the known interventional law.
This metric is the empirical analogue of the causal risk analyzed in our theory:
it averages statewise Wasserstein distances between the true and generated outcome laws at the same treatment--covariate state under the target design.
We therefore use it as the headline criterion instead of point-estimation errors, which only evaluate scalar functionals and can miss discrepancies in spread, tails, quantiles, or multimodality.
For Jobs, individual counterfactual distributions are not observed;
we instead compare model-implied interventional cumulative distribution functions (CDFs) with arm-level CDFs from the held-out NSW randomized controlled trial (RCT) sample.

\begin{table}[tb]
    \centering
    \setlength{\tabcolsep}{4pt}
    \renewcommand{\arraystretch}{1.05}
    \caption{
    Main distributional results over $100$ repetitions, with standard errors in parentheses (lower is better).
    IHDP and TCGA use empirical extended Wasserstein error, and Jobs uses RCT-assisted arm-level Wasserstein error on the asinh-transformed earnings scale.
    }
    \label{tab:main-exp}
    \begin{tabular}{lccc}
        \toprule
        Method
        & IHDP
        & TCGA
        & Jobs \\
        \midrule
        GANITE
        & $0.762\;(0.039)$
        & -- 
        & $0.973\;(0.044)$
        \\
        PO-Flow
        & $0.433\;(0.007)$
        & -- 
        & $0.299\;(0.011)$
        \\
        DiffPO
        & $0.323\;(0.005)$
        & -- 
        & $0.704\;(0.015)$
        \\
        INFs
        & $0.355\;(0.005)$
        & -- 
        & $1.309\;(0.011)$ 
        \\
        DR-Learner
        & $1.056\;(0.011)$
        & -- 
        & $0.365\;(0.007)$
        \\
        CoDiTE
        & $\underline{0.221}\;(0.003)$
        & --
        & $\underline{0.293}\;(0.009)$
        \\
        DNet
        & $0.637\;(0.009)$
        & --
        & $0.356\;(0.008)$
        \\
        Causal-DRF
        & $\boldsymbol{0.203}\;(0.002)$
        & --
        & $0.553\;(0.010)$
        \\
        SKCD (adapted)
        & $0.287\;(0.003)$
        & --
        & $0.376\;(0.011)$
        \\
        SCIGAN
        & -- 
        & $\underline{0.401}\;(0.002)$
        & -- 
        \\
        DRNet
        & -- 
        & $0.426\;(0.001)$
        & -- 
        \\
        VCNet
        & -- 
        & $0.574\;(0.005)$
        & -- 
        \\
        \method\ (ours)
        & $0.286\;(0.004)$
        & $\boldsymbol{0.378}\;(0.005)$
        & $\boldsymbol{0.209}\;(0.007)$ \\
        \bottomrule
    \end{tabular}
\end{table}

\begin{figure}[tb]
    \centering
    \begin{subfigure}{0.325\linewidth}
      \centering
      \includegraphics[width=\linewidth]{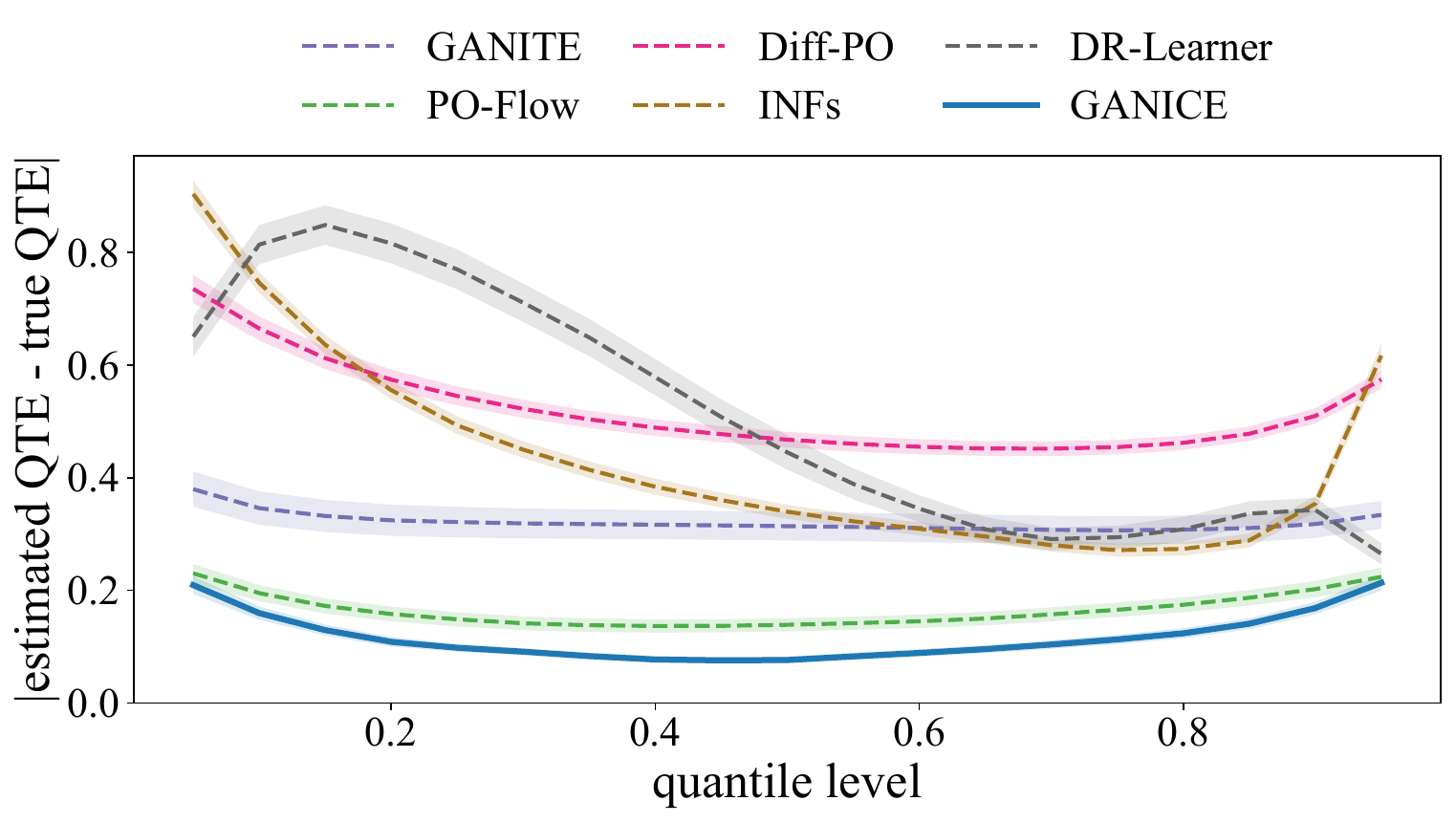}
      \caption{IHDP quantile-effect error}
    \end{subfigure}
    \hfill
    \begin{subfigure}{0.325\linewidth}
      \centering
      \includegraphics[width=\linewidth]{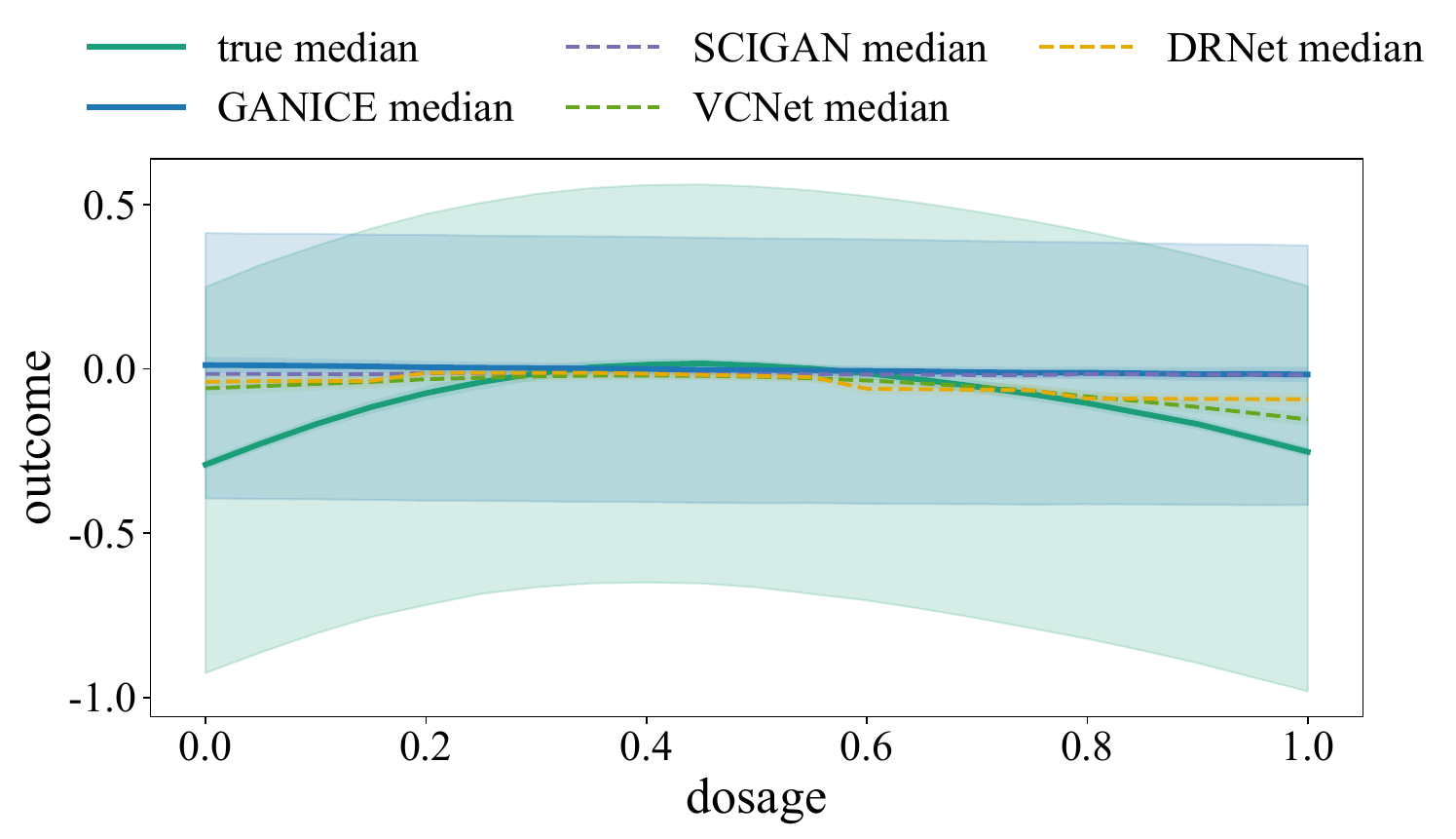}
      \caption{TCGA dose-wise quantiles}
    \end{subfigure}
    \hfill
    \begin{subfigure}{0.325\linewidth}
      \centering
      \includegraphics[width=\linewidth]{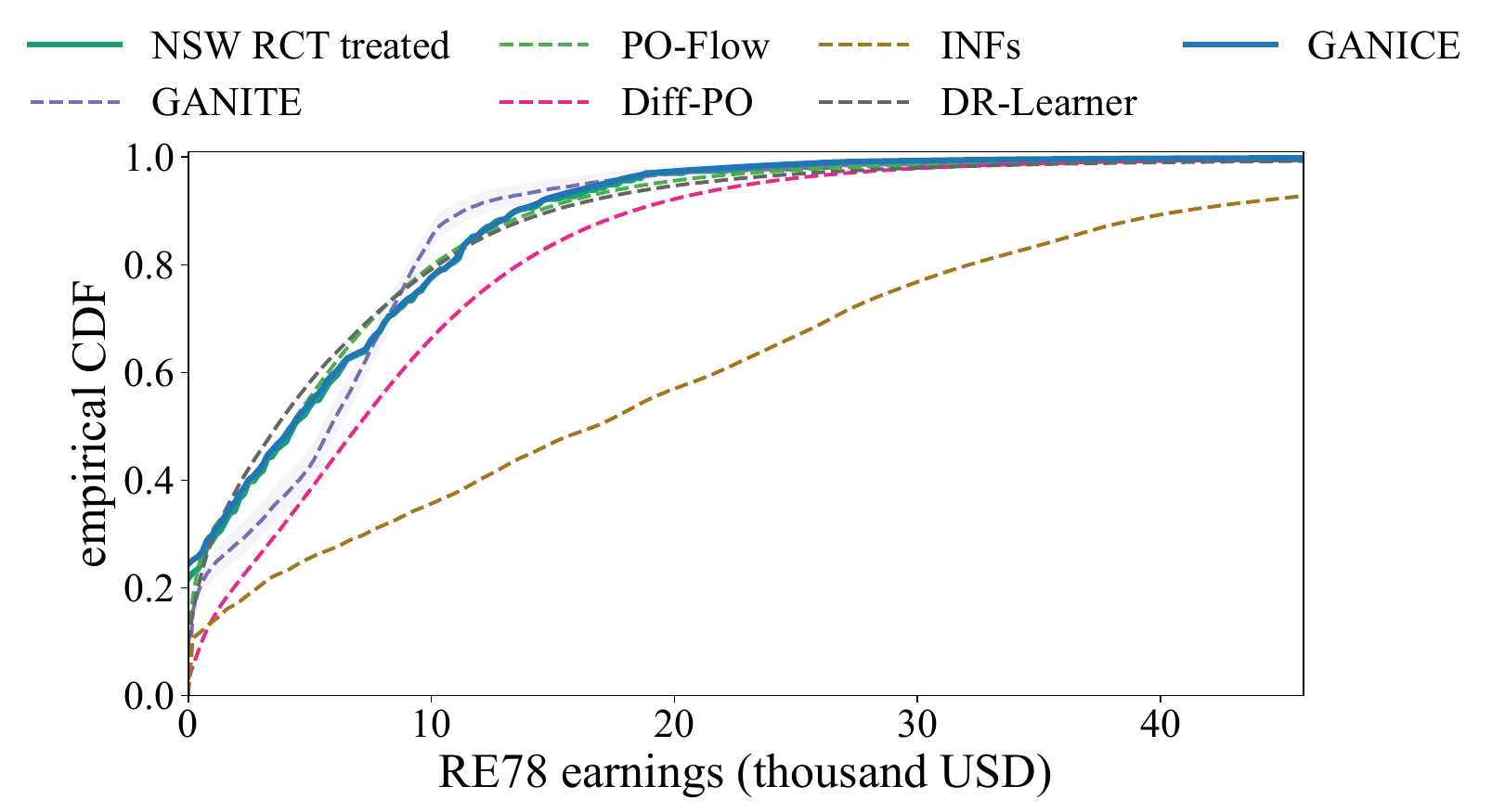}
      \caption{Jobs randomized CDF}
    \end{subfigure}
    \caption{
    Distributional diagnostics.
    (a) Absolute quantile treatment-effect error as a function of quantile level on IHDP.
    (b) Dose-indexed medians and predictive quantile bands on TCGA.
    (c) Model-implied treated-arm CDF against the RCT treated-arm CDF on Jobs.
    }
    \label{fig:main-exp}
\end{figure}

\paragraph{Results.}
Table~\ref{tab:main-exp} shows \method\ achieves the lowest distributional error on TCGA and Jobs, while Causal-DRF and CoDiTE achieve lower errors on IHDP.
This pattern is consistent with the strengths of local kernel and forest averaging \citep{park2021conditional,naf2026causal}:
our IHDP construction combines binary treatment, $25$ covariates, and smoothly varying conditional laws.
Pooling information across similar covariate profiles can stabilize conditional distribution estimates in this setting.
These benchmarks illustrate settings where scalar summaries fall short:
IHDP features heterogeneous non-Gaussian outcomes, TCGA requires a treatment--dosage-indexed law, and Jobs evaluates agreement with randomized arm-level distributions.
Figure~\ref{fig:main-exp} illustrates these effects.
On IHDP, \method\ yields the smallest absolute average quantile-treatment-effect error among displayed methods across nearly all quantiles, including tails, capturing distributional heterogeneity beyond mean contrasts.
On TCGA, baselines track the central trajectory but miss the true law's broad uncertainty band; \method\ instead generates a quantile band covering the central mass across dosages (despite residual bias at the dosage boundaries), giving the best integrated error (Table~\ref{tab:main-exp}).
On Jobs, the treated-arm CDF generated by \method\ closely tracks the randomized NSW CDF over the earnings distribution's main body and upper tail, whereas the displayed baselines are shifted or more widely dispersed.
Additional metrics, calibration diagnostics, and objective ablations in Appendix~\ref{app:additional-results} indicate that these improvements are not specific to the extended Wasserstein metric.
Implementation diagnostics and sensitivity analyses are reported in Appendices~\ref{app:implementation-diagnostics} and~\ref{app:sensitivity-analysis}.

\section{Conclusion}
\label{sec:conclusion}

We proposed \method, an extended Wasserstein GAN for causal distribution learning that targets the identifiable family of conditional interventional distributions.
We showed that factual-coordinate adversarial games identify treatment-specific conditional marginals, and introduced a stratified extended-Wasserstein objective whose population value is the causal distributional risk under overlap reweighting.
Its finite-resolution implementation learns from cell-normalized conditional samples and aggregates cellwise discrepancies under the target design, avoiding density-ratio estimation.
We established minimax-optimal rates, up to logarithmic factors and an additive implementation-transfer error, for finite and anisotropic continuous conditioning, using extended Wasserstein duality, WGAN perturbation inequalities, stratified finite-resolution reductions, and matching lower bounds;
Besov theory characterizes the discontinuous critics.

A limitation of this work is that the theoretical guarantees for continuous conditioning rely on strong regularity assumptions, including target-design regularity, anisotropic $W_1$ smoothness, and a pointwise latent pushforward structure.
Moreover, as in many minimax analyses of adversarial estimators, guarantees for single-network implementations without cell averaging require both the transfer condition and control of how the generated law varies within cells.
The resulting rates therefore characterize approximate population or empirical minimizers, rather than the dynamics or convergence of practical nonconvex optimization.
Practical issues (e.g., cell resolution, weak overlap, target-design shift, and output truncation) are discussed in Appendices~\ref{app:truncation-bias} and \ref{app:sensitivity-analysis}.

\begin{ack}
We are grateful to the reviewers for their insightful comments.
Shu Tamano was supported by JSPS KAKENHI Grant Number 25K24203.
Masaaki Imaizumi was supported by JSPS KAKENHI (24K02904), JST CREST (JPMJCR21D2),  JST FOREST (JPMJFR216I), and JST BOOST (JPMJBY24A9).
\end{ack}

\bibliographystyle{alpha}
\bibliography{bibliography}


\newpage
\appendix
\section{Detailed Related Work}
\label{app:detailed-related-work}

\paragraph{Adversarial and implicit generative models.}
The adversarial formulation of generative modeling was introduced by \citep{goodfellow2014generative}.
Conditional GANs \citep{mirza2014conditional} and conditional moment-matching networks \citep{ren2016conditional} extend implicit generation to settings with covariates, labels, or side information, while WGANs replace the Jensen--Shannon objective with a Wasserstein objective with a metric geometry that remains informative for singular distributions \citep{arjovsky2017wasserstein}.
This feature is important for causal generative modeling: outcome distributions may be low-dimensional, non-Gaussian, or supported on structured sets, and density-based divergences may be poorly aligned with the target distributional discrepancy.
Our work uses Wasserstein geometry as the statistical risk for the identifiable conditional interventional distribution, and not merely as a training heuristic.

\paragraph{Adversarial causal inference.}
GANITE \citep{yoon2018ganite} introduced adversarial counterfactual imputation for individualized treatment-effect estimation with discrete treatments, building on ideas from adversarial missing-data imputation \citep{yoon2018gain}.
SCIGAN \citep{bica2020estimating} extended this architecture to continuous-valued interventions through hierarchical discriminators over treatments and dosages.
Subsequent adversarial causal methods include conditional GANs for treatment selection \citep{ge2020conditional}, generative adversarial deconfounding for continuous treatments \citep{li2020continuous}, time-varying and sequential treatment generators \citep{wu2024counterfactual,norimatsu2025edts}, compromised adversarial nets for individualized effects \citep{abdisa2026individualized}, and mediation-oriented adversarial generators \citep{huan2024conditional,zhang2026gahmn}.
Recent flow-based methods such as PO-Flow and DoFlow further illustrate the shift toward generative causal targets \citep{wu2025po,wu2026doflow}, and \citep{luedtke2025doublegen} study debiased generative counterfactual modeling.
A complementary survey perspective on complex treatments is given by \citep{wang2026causal}.
Instead of proposing another counterfactual imputation architecture, we analyze the population objective of factual-coordinate adversarial games, show that they identify conditional interventional marginals rather than a joint potential-outcome distribution, and replace the factual-coordinate discriminator with an extended Wasserstein objective for the identifiable causal target.

\paragraph{Counterfactual distributions and distributional treatment effects.}
Classical work on distributional causal parameters includes counterfactual distribution inference \citep{chernozhukov2013inference}, quantile treatment-effect estimation \citep{firpo2007efficient}, bootstrap tests for distributional effects \citep{abadie2002bootstrap}, and distributional policy effects \citep{rothe2010nonparametric,rothe2012partial}.
Recent work studies semiparametric counterfactual density estimation \citep{kennedy2023semiparametric}, conditional distributional treatment effects through kernel embeddings and robust procedures \citep{park2021conditional,kallus2023robust}, distributional individualized treatment effects \citep{wu2023dnet}, distributional random forests for causal targets \citep{naf2026causal}, randomized-experiment distributional effects with machine-learning variance reduction \citep{byambadalai2024estimating,byambadalai2025on,oka2026regression}, and doubly robust conditional distributional effect estimation \citep{jain2026conditional}.
These works typically target distribution functions, densities, quantiles, distributional functionals, or kernel mean embeddings.
We instead target an implicit generative estimate of the full conditional interventional distribution under extended Wasserstein risk.

\paragraph{Conditional distribution and density estimation.}
Conditional CDF and quantile estimation have been studied using kernel, local-polynomial, transformation, and random-forest methods \citep{koenker1978regression,chaudhuri1991nonparametric,hall1999methods,meinshausen2006quantile,li2008nonparametric,brunel2010minimax,hothorn2014conditional,athey2019generalized,elie2022random,naf2023confidence}.
Conditional density estimation has a large nonparametric literature, including local likelihood and kernel estimators \citep{fan1996estimation,fan1998local,hyndman2002nonparametric,hall2004cross}, oracle and minimax theory \citep{efromovich2007conditional,efromovich2010oracle,bertin2016adaptive,sart2017estimating,li2022minimax,bilodeau2023minimax,cattaneo2024boundary}, least-squares and high-dimensional reductions \citep{sugiyama2010least,izbicki2017converting}, Bayesian conditional density models \citep{pati2013posterior,norets2014posterior,norets2017adaptive,shen2016adaptive}, and likelihood-based conditional deep generative models \citep{kumar2024likelihood}.
These approaches are foundational for conditional distribution learning, but many of them rely on distributions with densities or distribution-function representations.
The implicit generative formulation studied here permits conditional outcome distributions that may be singular with respect to Lebesgue measure, while the causal target is handled through overlap-based reweighting.

\paragraph{Conditional Wasserstein distances and conditional optimal transport.}
Several Wasserstein-type constructions have been proposed for conditional distributions.
Data-driven conditional optimal transport estimates transport maps between conditional distributions indexed by covariates \citep{tabak2020conditional,tabak2021data}.
Conditional Wasserstein barycenters use Wasserstein geometry for interpolation, extrapolation, and distribution-on-predictor regression \citep{fan2025conditional}.
Conditional optimal transport on function spaces studies constrained triangular transport maps and Kantorovich relaxations for conditional measures, with applications to amortized Bayesian inference \citep{hosseini2025conditional}.
Dynamic conditional optimal transport uses conditional Wasserstein geometry to construct simulation-free conditional flows \citep{kerrigan2024dynamic}.
In conditional generative modeling, conditional Wasserstein generators and Wasserstein geodesic generators use optimal-transport geometry to learn conditional distributions and geodesic interpolation between observed domains \citep{kim2023conditionalgenerator,kim2023wasserstein}; \citep{martin2021exchanging} study the exchange between expectation and supremum in conditional WGAN objectives.

Our use of extended Wasserstein distance follows \citep{chemseddine2025conditional} for a specific reason.
Their distance is defined through couplings that preserve the conditioning coordinate, equals the target-design average of pointwise Wasserstein distances between conditional distributions, and admits a conditional Kantorovich--Rubinstein dual whose critic is Lipschitz in the outcome coordinate conditional on the state.
This dual is precisely the form needed for a conditional WGAN objective.
This choice is essential in causal inference: a joint Wasserstein distance on the treatment--covariate--outcome space can reduce discrepancy by moving treatment--covariate states, whereas the causal estimand compares outcome distributions at the same treatment--covariate state.
The extended Wasserstein distance therefore matches the identifiable causal risk, while still yielding an adversarial training objective.

\paragraph{Deep generative distribution estimation theory.}
For unconditional generative distribution estimation, \citep{tang2023minimax} derive minimax rates under adversarial losses on unknown submanifolds, and \citep{stephanovitch2024wasserstein} prove minimax optimality of WGAN estimators using a neural generator sieve and an optimal-dual-on-a-net critic construction.
Conditional diffusion models have recently been analyzed by \citep{tang2025conditional} and \citep{hu2025statistical};
their rates quantify the statistical cost of conditioning under smooth density or score assumptions.
Wasserstein generative regression \citep{song2026wasserstein}, Wasserstein geodesic generators \citep{kim2023wasserstein}, and vicinal analyses of conditional GANs \citep{jang2026ve} provide related conditional generative perspectives.
Our theory differs in four aspects:
the target is causal, the loss is extended Wasserstein risk, observational sampling is handled by stratified cell normalization under overlap, and the local generative difficulty is inherited from minimax-optimal WGAN estimation under a latent pushforward model.

\paragraph{Anisotropic Besov structures and non-smooth critics.}
Anisotropic Besov spaces are classical tools for describing functions whose regularity differs across coordinates.
They have been used in wavelet thresholding and adaptive nonparametric estimation over anisotropic smoothness classes \citep{neumann1997wavelet,kerkyacharian2001nonlinear,hoffmann2002random}, and their approximation-theoretic foundations are developed in the function-space literature \citep{triebel1983theory}.
In deep learning theory, \citep{suzuki2019adaptivity} shows adaptivity of deep ReLU networks over Besov and mixed-smooth Besov spaces, while \citep{suzuki2021deep} develop approximation, estimation, and minimax analyses for anisotropic Besov spaces and intrinsic smoothness structures.
These works use Besov regularity primarily to characterize target regression functions or neural approximation classes.

Our use of Besov structure is different.
The Besov object is not the conditional mean or density, but the critic map from treatment--covariate states to outcome-Lipschitz Kantorovich potentials.
This distinction matters because extended Wasserstein critics must remain Lipschitz in the outcome coordinate while being allowed to change sharply across treatment--covariate regimes.
Embedding finite-resolution step critics into anisotropic Besov balls below the continuity threshold gives a norm bound of $C\max_j2^{m_js_j}$ for their partition-boundary jumps.
The stochastic risk bounds follow from the restricted one-state critic classes and stability under cell averaging.

\section{Notation}
\label{app:notations}

This appendix summarizes the notation used throughout the paper in Tables~\ref{tab:notation-summary} and~\ref{tab:notation-summary-rates}.
All probability measures are defined on standard Borel spaces, and conditional distributions are fixed regular versions.
Equalities involving conditional distributions are interpreted up to the relevant conditioning marginal.
Write $\Gamma(\eta,\zeta)$ for the set of couplings with marginals $\eta,\zeta$.
We use $\|f\|_\infty=\sup_x\|f(x)\|$ for the uniform norm, $e_j$ for the $j$th coordinate vector in the relevant Euclidean space.
We also write $\1_A$ for the indicator of a measurable set $A$.

\begin{table}[tb]
    \centering
    \caption{Summary of basic causal and generative notation}
    \label{tab:notation-summary}
    \begin{tabular}{p{0.25\linewidth}p{0.68\linewidth}}
        \toprule
        Symbol & Description \\
        \midrule
        $\cP(S)$
        & Borel probability measures on a standard Borel space $S$.
        \\
        $g_\#\xi$
        & Pushforward of a probability law $\xi$ by a measurable map $g$;
        $g_\#U$ abbreviates $g_\#\Unif([0,1]^{d_U})$.
        \\
        $W_1$
        & Ordinary Wasserstein-1 distance on the outcome space $\cY$;
        $W_1^\otimes$ denotes the product-space version where applicable.
        \\
        $\Lip(f)$
        & Optimal Lipschitz constant of $f$ with respect to the relevant metric.
        \\
        $\cX,\cT,\cY$
        & Covariate, treatment, and outcome spaces;
        in the theoretical analysis, $\cY=B^p(0,2K_0)$ is a closed Euclidean ball and true generator images lie in $B^p(0,K_0)$.
        \\
        $X,T,Y(t)$
        & Covariates, treatment, and potential outcome under treatment $t$.
        \\
        $Y^{\obs}$
        & Observed outcome, equal to $Y(T)$ under consistency.
        \\
        $W=(X,T)$
        & Treatment--covariate state used as the conditioning variable.
        \\
        $\cW$
        & Conditioning space;
        either finite $\{1,\ldots,M\}$ or continuous $[0,1]^{d_W}$.
        \\
        $n$
        & Total observational sample size; $(W_i,Y_i^{\obs})$, $i=1,\ldots,n$, are i.i.d. observations from $P_{\obs}$.
        \\
        $Q_{\obs}$, $Q_\rho$
        & Observational and target design marginals of $W$.
        \\
        $w_\rho$
        & Density ratio $\dd Q_\rho/\dd Q_{\obs}$;
        overlap assumes $0\le w_\rho\le\kappa^{-1}$.
        \\
        $\kappa$
        & Overlap constant in $(0,1]$;
        smaller $\kappa$ means weaker overlap.
        \\
        $\mu_w^\ast$
        & Conditional interventional outcome law at state $w=(x,t)$:
        $\mu_{(x,t)}^\ast(A)=\Prob(Y(t)\in A\mid X=x)$ for Borel $A\subset\cY$.
        \\
        $P_{\obs}$
        & Observational law $P_{\obs}(\dd w,\dd y)=Q_{\obs}(\dd w)\mu_w^\ast(\dd y)$ under consistency and ignorability.
        \\
        $P^\ast$
        & Target interventional law $P^\ast(\dd w,\dd y)=Q_\rho(\dd w)\mu_w^\ast(\dd y)$.
        \\
        $U,d_U$
        & Exogenous noise $U\sim\Unif([0,1]^{d_U})$ and its latent dimension $d_U$.
        \\
        $g$
        & Conditional generator $g:\cW\times[0,1]^{d_U}\to\cY$.
        \\
        $\nu_{g,w}$
        & Generated conditional law $g(w,\cdot)_\#U$;
        generic $\nu_w$ or $\nu$ denotes a comparison or generated probability law.
        \\
        $P_g$
        & Raw generated law $P_g(\dd w,\dd y)=Q_\rho(\dd w)\nu_{g,w}(\dd y)$.
        \\
        \bottomrule
    \end{tabular}
\end{table}

\begin{table}[tb]
    \centering
    \caption{Notation for distances, partitions, sieves, and rates}
    \label{tab:notation-summary-rates}
    \begin{tabular}{p{0.25\linewidth}p{0.68\linewidth}}
        \toprule
        Symbol & Description \\
        \midrule
        $\eW_1$
        & Extended Wasserstein-1 distance on $\cW\times\cY$ for laws sharing the same $\cW$-marginal; equivalently, the average of statewise $W_1$ distances.
        \\
        $\eW_1^{(\boldsymbol{m})}$
        & Finite-resolution discrepancy induced by $\cF^{(\boldsymbol{m})}$; it equals the target-mass-weighted sum of cellwise $W_1$ distances.
        \\
        $\Gamma(\eta,\zeta)$, $\Gamma_\Delta(P,R)$
        & Couplings with marginals $\eta,\zeta$, and the subset of couplings between $P,R$ that preserve the conditioning state, respectively.
        \\
        $\mathcal{L}_1$
        & Anchored outcome critic class $\{f:f(y_0)=0,\ \Lip(f)\le1\}$.
        \\
        $\cF$
        & Exact extended-Wasserstein dual class; each section $h(w,\cdot)$ belongs to $\mathcal{L}_1$.
        \\
        $\boldsymbol{m},\Pi_{\boldsymbol{m}}$
        & Resolution vector $\boldsymbol{m}=(m_1,\ldots,m_{d_W})\in\mathbb{N}_0^{d_W}$ and its anisotropic dyadic partition, with side length $2^{-m_j}$ in coordinate $j$.
        \\
        $\boldsymbol{m}_n$
        & Optimized resolution vector with $m_{n,j}=\lfloor r\log_2(\kappa n)/\alpha_j\rfloor\vee0$; the generic IPM result uses $r_{\mathcal{V}}$ in place of $r$.
        \\
        $|\boldsymbol{m}|_1$
        & Total resolution $\sum_{j=1}^{d_W}m_j$.
        \\
        $C,q_C,\pi_C$
        & A cell $C$, its target mass $q_C=Q_\rho(C)$, and observational mass $\pi_C=Q_{\obs}(C)$.
        \\
        $C_{\boldsymbol{m}}(w)$
        & Cell of $\Pi_{\boldsymbol{m}}$ containing $w$.
        \\
        $N_C^{\obs}$, $M_C^\rho$
        & Realized observational count and target-design Monte Carlo count in cell $C$.
        \\
        $\mu_C^{\obs},\mu_C^\rho,\nu_{g,C}^\rho$
        & Outcome cell mixtures under $Q_{\obs}$ and $Q_\rho$, and generated cell mixture under $Q_\rho$, respectively, normalized by cell masses; $\mu_C^\rho$ uses the kernel of $P$, with $\mu_w=\mu_w^\ast$ when $P=P^\ast$.
        \\
        $P_g^{(\boldsymbol{m})}$
        & Cell-resolution generated law $Q_\rho(\dd w)\nu_{g,C_{\boldsymbol{m}}(w)}^\rho(\dd y)$.
        \\
        $\cF^{(\boldsymbol{m})}$
        & Finite-resolution critic class, piecewise constant in $w$ and outcome-Lipschitz within each cell.
        \\
        $k,\ell$
        & Numbers of observed outcomes and independent latent draws in a one-state estimation problem.
        \\
        $N,\cG_N,\mathcal{D}_N$
        & Deterministic sieve index $N\ge2$ and finite one-state generator and critic classes; the oracle bound applies whenever $k,\ell\ge c_\ast N$ for a fixed $c_\ast>0$.
        \\
        $N_{\boldsymbol{m}}$
        & Deterministic effective cell sample size $\lceil\kappa n2^{-|\boldsymbol{m}|_1}\rceil\vee2$, used as the one-state sieve index in every cell at resolution $\boldsymbol{m}$, giving critic class $\mathcal{D}_{N_{\boldsymbol{m}}}$.
        \\
        $\cG_n^{\cont},\cH_{n,\boldsymbol{m}}^{\cont}$
        & Implemented continuous-conditioning generator class and finite-resolution critic class.
        \\
        $\hat{P}_{n,\boldsymbol{m}}^{\cont}$
        & Cell-resolution estimate $P_{\hat{g}_{n,\boldsymbol{m}}^{\cont}}^{(\boldsymbol{m})}$.
        \\
        $\mathcal{H}_{K_0}^{s}$
        & Periodic H\"older ball of maps $\mathbb{T}^{d_U}\to\R^p$ with image in $B^p(0,K_0)$ and smoothness $s$.
        \\
        $\beta$
        & Smoothness parameter of the one-state latent generator class $\mathcal{H}_{K_0}^{\beta+1}$.
        \\
        $a$
        & Local one-state WGAN rate exponent $a=\min\{(\beta+1)/(2\beta+d_U),1/2\}$.
        \\
        $\boldsymbol{\alpha}$
        & Anisotropic conditional regularity vector $(\alpha_1,\ldots,\alpha_{d_W})$.
        \\
        $b_{\boldsymbol{m}}$
        & Localization bias scale $\sum_j2^{-\alpha_jm_j}$.
        \\
        $\bar{d}$
        & Effective anisotropic conditioning dimension $\sum_j\alpha_j^{-1}$.
        \\
        $r$
        & Continuous-conditioning rate exponent $(a^{-1}+\bar{d})^{-1}$.
        \\
        $p_B,\boldsymbol{s}$
        & Integrability and smoothness indices in the anisotropic Besov--Nikolskii critic embedding.
        \\
        $\cM_{\disc},\cM_{\cont}$
        & Finite-state and continuous-conditioning causal model classes.
        \\
        $\mathcal{R}_n^{\disc},\mathcal{R}_n^{\cont}$
        & Minimax risks in finite-state and continuous-conditioning settings.
        \\
        \bottomrule
    \end{tabular}
\end{table}

\subsection{Dyadic Partitions}
\label{app:dyadic-partition-details}

For $\boldsymbol{m}=(m_1,\ldots,m_{d_W})\in\mathbb{N}_0^{d_W}$, the anisotropic dyadic partition $\Pi_{\boldsymbol{m}}$ consists of rectangles
\begin{equation*}
    \prod_{j=1}^{d_W} I_{k_j,m_j}
    ,
    \quad
    k_j\in\{0,\ldots,2^{m_j}-1\}
    ,
\end{equation*}
where $I_{k,m}=[k2^{-m},(k+1)2^{-m})$ for $k<2^m-1$ and $I_{2^m-1,m}=[(2^m-1)2^{-m},1]$.
Thus, each cell has side length $2^{-m_j}$ in coordinate $j$ and Lebesgue volume $2^{-|\boldsymbol{m}|_1}$.
Under Assumption~\ref{ass:design-regularity-main}, $q_C=Q_\rho(C)\asymp2^{-|\boldsymbol{m}|_1}$ uniformly over $C\in\Pi_{\boldsymbol{m}}$.

\section{Proofs of Section~\ref{sec:problem-setup}}
\label{app:problem-setup-proofs}

Let $K_x$ be a regular conditional law of $V$ given $X=x$.
For Borel $A\subset\cY$, write
\begin{equation*}
    \mu_{(x,t)}^\ast(A)
    =
    \int
    \1\{F(x,t,v)\in A\}
    K_x(\dd v)
    .
\end{equation*}

\begin{lemma}[Identification of the conditional outcome kernel]
\label{lem:causal-kernel-identification-app}
    Under Assumptions~\ref{ass:consistency}--\ref{ass:ignorability}, $w\mapsto\mu_w^\ast$ is a Borel probability kernel, and for every bounded measurable $\varphi:\cW\times\cY\to\R$,
    \begin{equation*}
        \E[\varphi(W,Y^{\obs})]
        =
        \int\int
        \varphi(w,y)
        \mu_w^\ast(\dd y)Q_{\obs}(\dd w)
        .
    \end{equation*}
\end{lemma}

\begin{proof}[Proof of Lemma~\ref{lem:causal-kernel-identification-app}]
    Conditional independence gives $\mathcal{L}(V\mid X,T)=K_X$ almost surely.
    Joint measurability of $F$ and consistency yield
    \begin{equation*}
        \E[\varphi(X,T,F(X,T,V))\mid X,T]
        =
        \int\varphi(X,T,F(X,T,v))K_X(\dd v)
        =
        \int\varphi(W,y)\mu_W^\ast(\dd y)
        ,
    \end{equation*}
    where the second equality is the definition of $\mu_{(x,t)}^\ast$ as the pushforward of $K_x$ by $F(x,t,\cdot)$.
    Taking expectations proves the identity.
    Measurability of parameterized kernel integrals gives the Borel kernel property.
    For each fixed $t$, the same definition is a version of $\mathcal{L}(Y(t)\mid X=x)$.
\end{proof}

This section formalizes the claim that factual-coordinate adversarial imputation identifies conditional marginals but not a joint potential-outcome law.
We state the result for finite treatments $\cT=\{1,\ldots,A\}$ and standard Borel spaces $\cX$ and $\cY$.
Let $\bar{\boldsymbol{Y}}=(\bar{Y}(1),\ldots,\bar{Y}(A))\in\cY^A$ be the completed outcome vector produced by an imputation generator after inserting the factual outcome.
All conditional probabilities are fixed regular versions, and $0\log0=0$.

Write $G$ for the imputation generator that determines the completed vector $\bar{\boldsymbol{Y}}$, and $d=(d_1,\ldots,d_A)$ for a discriminator with coordinate functions $d_a:\cX\times\cY^A\to[0,1]$.
For a fixed $G$, define
\begin{equation*}
    r_a(x,\bar{y})
    =
    \Prob(T=a\mid X=x,\bar{\boldsymbol{Y}}=\bar{y})
    ,
    \quad
    \ell(u)
    =
    u\log u
    +
    (1-u)\log(1-u)
    ,
    \quad
    u\in[0,1]
    .
\end{equation*}
The discriminator objective is
\begin{equation}
\label{eq:population-factual-index-objective-app}
    V(G,d)
    \coloneq
    \E\Biggl[
        \sum_{a=1}^A
        \{
            \1\{T=a\}\log d_a(X,\bar{\boldsymbol{Y}})
            +
            \1\{T\ne a\}\log(1-d_a(X,\bar{\boldsymbol{Y}}))
        \}
    \Biggr]
    .
\end{equation}

\begin{proposition}[Population factual-coordinate discriminator]
\label{prop:population-factual-index-app}
    For every fixed $G$, maximization over measurable discriminators $d$ has the following properties:
    \begin{enumerate}[label=(\roman*),leftmargin=*]
        \item the pointwise maximizer is $d_a^\ast(x,\bar{y})=r_a(x,\bar{y})$;
        \item the maximized value is $\sup_d V(G,d)=\E[\sum_{a=1}^A\ell(r_a(X,\bar{\boldsymbol{Y}}))]$;
        \item
        \begin{equation}
        \label{eq:invariance-jensen-lower-bound-app}
            \sup_d V(G,d)
            \ge
            \E\Biggl[
                \sum_{a=1}^A
                \ell(\Prob(T=a\mid X))
            \Biggr]
            ,
        \end{equation}
        with equality if and only if $T\indep\bar{\boldsymbol{Y}}\mid X$, up to null sets.
    \end{enumerate}
\end{proposition}

\begin{proof}[Proof of Proposition~\ref{prop:population-factual-index-app}]
    Fix $a$ and condition on $(X,\bar{\boldsymbol{Y}})=(x,\bar{y})$.
    The conditional contribution is
    $r_a(x,\bar{y})\log q+(1-r_a(x,\bar{y}))\log(1-q)$ with $q=d_a(x,\bar{y})$.
    This Bernoulli log-likelihood is concave in $q$ and has derivative $r_a/q-(1-r_a)/(1-q)$.
    The unique interior maximizer is $q=r_a$ when $r_a\in(0,1)$; when $r_a\in\{0,1\}$, the same value is attained in the closure under the convention $0\log0=0$.
    This proves (i) and (ii).

    The function $\ell$ is convex on $[0,1]$ and strictly convex on $(0,1)$.
    Jensen's inequality conditional on $X$ gives
    \begin{equation*}
        \E[\ell(r_a(X,\bar{\boldsymbol{Y}}))\mid X]
        \ge
        \ell(\E[r_a(X,\bar{\boldsymbol{Y}})\mid X])
        .
    \end{equation*}
    By the tower property,
    $\E[r_a(X,\bar{\boldsymbol{Y}})\mid X]=\Prob(T=a\mid X)$.
    Summing over $a$ and taking expectations proves \eqref{eq:invariance-jensen-lower-bound-app}.
    Equality holds exactly when $r_a(X,\bar{\boldsymbol{Y}})=\Prob(T=a\mid X)$ almost surely for every $a$ on nondegenerate conditional supports, which is equivalent to $T\indep\bar{\boldsymbol{Y}}\mid X$.
\end{proof}

Let $p_a(x)=\Prob(T=a\mid X=x)$ be a fixed regular version of the propensity for treatment $a$.

\begin{theorem}[Identified object of factual-coordinate adversarial imputation]
\label{thm:existing-gan-marginal-identification-app}
    Assume Assumptions~\ref{ass:consistency} and~\ref{ass:ignorability}.
    Suppose that a completed vector $\bar{\boldsymbol{Y}}$ satisfies exact factual reconstruction, $\bar{Y}(T)=Y^{\obs}$ almost surely, and population invariance, $T\indep\bar{\boldsymbol{Y}}\mid X$.
    Then, for every treatment $a$,
    \begin{equation}
    \label{eq:marginal-identification-existing-gan-app}
        \mathcal{L}(\bar{Y}(a)\mid X=x)
        =
        \mathcal{L}(Y(a)\mid X=x)
    \end{equation}
    for $P_X$-almost every $x$ such that $p_a(x)>0$.
    Conversely, let $x\mapsto\kappa_x$ be any measurable family of couplings of the marginals
    $\{\mathcal{L}(Y(a)\mid X=x):a=1,\ldots,A\}$.
    There exists a randomized imputation generator satisfying exact factual reconstruction and population invariance such that
    $\mathcal{L}(\bar{\boldsymbol{Y}}\mid X=x)=\kappa_x$ for $P_X$-almost every $x$.
    Therefore, factual-coordinate adversarial objectives identify treatment-specific conditional marginals but not the conditional joint law of all potential outcomes.
\end{theorem}

\begin{proof}[Proof of Theorem~\ref{thm:existing-gan-marginal-identification-app}]
    We first prove marginal identification.
    Fix $a$ and let $x$ be a point at which the relevant regular conditional distributions are defined and $p_a(x)>0$.
    By $T\indep\bar{\boldsymbol{Y}}\mid X$, the conditional law of $\bar{\boldsymbol{Y}}$ given $(X=x,T=t)$ is independent of $t$ over the support of $T\mid X=x$, and in particular coincides with $\mathcal{L}(\bar{\boldsymbol{Y}}\mid X=x)$ at $t=a$.
    Projecting onto coordinate $a$, using exact factual reconstruction $\bar{Y}(a)=Y^{\obs}$ on $\{T=a\}$, and then applying consistency ($Y^{\obs}=Y(a)$ on $\{T=a\}$) and ignorability ($Y(a)\indep T\mid X$) yields
    \begin{equation*}
        \begin{split}
            \mathcal{L}(\bar{Y}(a)\mid X=x)
            &=
            \mathcal{L}(\bar{Y}(a)\mid X=x,T=a)
            \\
            &=
            \mathcal{L}(Y^{\obs}\mid X=x,T=a)
            \\
            &=
            \mathcal{L}(Y(a)\mid X=x,T=a)
            \\
            &=
            \mathcal{L}(Y(a)\mid X=x)
            ,
        \end{split}
    \end{equation*}
    which proves \eqref{eq:marginal-identification-existing-gan-app}.

    For the converse, write $y_{-a}$ for the outcome vector with coordinate $a$ omitted, and disintegrate $P_X(\dd x)\kappa_x(\dd\boldsymbol{y})$ with respect to $(x,y_a)$, separately for each of the finitely many treatments $a$.
    Since the $a$th marginal of $\kappa_x$ is $\mathcal{L}(Y(a)\mid X=x)$, standard Borel disintegration supplies a jointly measurable kernel $K_{x,a}(\dd y_{-a}\mid y_a)$ such that
    \begin{equation*}
        \kappa_x(\dd y_a,\dd y_{-a})
        =
        \mathcal{L}(Y(a)\mid X=x)(\dd y_a)\,K_{x,a}(\dd y_{-a}\mid y_a)
    \end{equation*}
    for $P_X$-almost every $x$.
    By the randomization lemma for kernels on standard Borel spaces, there exist an auxiliary $Z\sim\Unif([0,1])$, independent of $(X,T,Y^{\obs})$, and a measurable map $R$ such that $R(x,a,y,Z)$ has law $K_{x,a}(\cdot\mid y)$.
    Given $(X,T,Y^{\obs},Z)=(x,a,y,z)$, define $\bar{Y}(a)=y$ and $\bar{\boldsymbol{Y}}_{-a}=R(x,a,y,z)$.
    Exact factual reconstruction is immediate.
    Conditional on $(X=x,T=a)$, the factual outcome has law $\mathcal{L}(Y(a)\mid X=x)$ by consistency and ignorability; sampling the remaining coordinates from $K_{x,a}$ reconstructs $\kappa_x$.
    Therefore $\mathcal{L}(\bar{\boldsymbol{Y}}\mid X=x,T=a)=\kappa_x$ for each treatment in the support of $T\mid X=x$, which is equivalent to $T\indep\bar{\boldsymbol{Y}}\mid X$.
\end{proof}

\section{Proofs of Section~\ref{sec:proposed-method}}
\label{app:proposed-method-proofs}

\subsection{Technical Conventions and Extended Wasserstein Duality}
\label{app:extended-wasserstein-details}

Throughout this section, $\cY$ is compact and metrized by the Euclidean distance inherited from $\R^p$, and $\cW$ is standard Borel.
Every probability measure on $\cW\times\cY$ admits a regular conditional distribution given $W$.
For $P(\dd w,\dd y)=Q(\dd w)\mu_w(\dd y)$ and $R(\dd w,\dd y)=Q(\dd w)\nu_w(\dd y)$, the kernels $w\mapsto\mu_w$ and $w\mapsto\nu_w$ are fixed versions, and $\Delta(w)=(w,w)$.

\begin{lemma}[Disintegration of diagonal couplings]
\label{lem:diagonal-coupling-disintegration}
    Let $P$ and $R$ share the conditioning marginal $Q$.
    A coupling $\alpha$ of $P$ and $R$ belongs to $\Gamma_\Delta(P,R)$ if and only if there exists a $Q$-measurable kernel $w\mapsto\gamma_w$ from $\cW$ to $\cY\times\cY$ such that $\gamma_w\in\Gamma(\mu_w,\nu_w)$ for $Q$-almost every $w$ and
    \begin{equation}
    \label{eq:diagonal-coupling-disintegration}
        \alpha(\dd w,\dd y,\dd w',\dd y')
        =
        Q(\dd w)\delta_w(\dd w')\gamma_w(\dd y,\dd y')
        .
    \end{equation}
    Consequently,
    \begin{equation}
    \label{eq:cw-disintegration-app}
        \inf_{\alpha\in\Gamma_\Delta(P,R)}
        \int\|y-y'\|\dd\alpha
        =
        \int_\cW W_1(\mu_w,\nu_w)Q(\dd w)
        .
    \end{equation}
\end{lemma}

\begin{proof}[Proof of Lemma~\ref{lem:diagonal-coupling-disintegration}]
    Let $\alpha\in\Gamma_\Delta(P,R)$.
    Since its $(W,W')$-marginal is $\Delta_\#Q$, $\alpha$ is concentrated on $\{w=w'\}$.
    Disintegrating $\alpha$ with respect to the first $w$ coordinate gives
    $\alpha(\dd w,\dd y,\dd w',\dd y')=Q(\dd w)K_w(\dd y,\dd w',\dd y')$.
    The diagonal constraint implies that $K_w$ is concentrated on $\cY\times\{w\}\times\cY$ for $Q$-almost every $w$, so there is a kernel $\gamma_w$ satisfying \eqref{eq:diagonal-coupling-disintegration}.
    The marginal constraints give $\gamma_w\in\Gamma(\mu_w,\nu_w)$ almost surely.

    Conversely, any representation \eqref{eq:diagonal-coupling-disintegration} with $\gamma_w\in\Gamma(\mu_w,\nu_w)$ has marginals $P$ and $R$ and diagonal $(W,W')$-marginal.
    The value in \eqref{eq:cw-disintegration-app} therefore reduces to an infimum of $\int\!\int\|y-y'\|\gamma_w(\dd y,\dd y')Q(\dd w)$ over measurable selections $w\mapsto\gamma_w\in\Gamma(\mu_w,\nu_w)$.
    Every selection has cost at least $\int W_1(\mu_w,\nu_w)Q(\dd w)$, because the inner integral is at least $W_1(\mu_w,\nu_w)$ for each $w$.
    For the reverse inequality, compactness of $\cY$ and continuity of the cost imply that the set of optimal couplings of $(\mu_w,\nu_w)$ is nonempty and compact for each $w$, and $w\mapsto(\mu_w,\nu_w)$ is measurable; the measurable selection theorem for compact-valued measurable correspondences then yields a measurable kernel $w\mapsto\gamma_w^\ast$ of optimal couplings, whose cost equals $\int W_1(\mu_w,\nu_w)Q(\dd w)$.
\end{proof}

\begin{lemma}[Extended Kantorovich--Rubinstein duality]
\label{lem:extended-kr-duality-app}
    Let $P$ and $R$ share the marginal $Q$ on $\cW$.
    Then
    \begin{equation}
    \label{eq:extended-kr-duality-app}
        \eW_1(P,R)
        =
        \sup_{h\in\cF}
        \{\E_P[h(W,Y)]-\E_R[h(W,Y)]\}
        .
    \end{equation}
    Moreover, for fixed Borel kernels $w\mapsto\mu_w$ and $w\mapsto\nu_w$, the supremum is attained.
\end{lemma}

\begin{proof}[Proof of Lemma~\ref{lem:extended-kr-duality-app}]
    For every $h\in\cF$, ordinary Kantorovich--Rubinstein duality on compact $\cY$ gives
    \begin{equation*}
        \int h(w,y)\mu_w(\dd y)-\int h(w,y)\nu_w(\dd y)
        \le
        W_1(\mu_w,\nu_w)
    \end{equation*}
    for $Q$-almost every $w$.
    Integrating and using Lemma~\ref{lem:diagonal-coupling-disintegration} gives the upper bound.

    Let $C(\cY)$ denote the space of continuous real-valued functions on $\cY$, equipped with the uniform norm.
    For the reverse direction, $\mathcal{L}_1$ is compact in $C(\cY)$ by Arzel\`a--Ascoli.
    Let $\Phi(w,f)=\int f\dd\mu_w-\int f\dd\nu_w$.
    The map $w\mapsto\Phi(w,f)$ is measurable for each $f$, and $f\mapsto\Phi(w,f)$ is continuous for each $w$.
    The measurable maximum theorem yields a measurable selector $f_w^\ast\in\mathcal{L}_1$ maximizing $\Phi(w,\cdot)$.
    Ordinary duality gives $\Phi(w,f_w^\ast)=W_1(\mu_w,\nu_w)$.
    Since evaluation is continuous on $C(\cY)\times\cY$, $h^\ast(w,y)=f_w^\ast(y)$ is jointly measurable and belongs to $\cF$.
    Integrating $\Phi(w,f_w^\ast)=W_1(\mu_w,\nu_w)$ over $Q$ gives $\E_P[h^\ast]-\E_R[h^\ast]=\int W_1(\mu_w,\nu_w)Q(\dd w)=\eW_1(P,R)$ by Lemma~\ref{lem:diagonal-coupling-disintegration}, which proves equality and attainment.
\end{proof}

For a finite metric conditioning space with at least two points, denote its minimum separation by $\delta_\cW=\min_{w\ne w'}d_\cW(w,w')$.

\begin{proposition}[Comparison on a finite conditioning space]
    \label{prop:finite-metric-comparison-app}
    Suppose $(\cW,d_\cW)$ is a finite metric space with at least two points.
    For any $P,R$ with the same conditioning marginal,
    \begin{equation*}
        W_1^\otimes(P,R)
        \le
        \eW_1(P,R)
        \le
        \max
        \{
            1,\diam(\cY)/\delta_\cW
        \}
        \,W_1^\otimes(P,R)
        .
    \end{equation*}
\end{proposition}

\begin{proof}[Proof of Proposition~\ref{prop:finite-metric-comparison-app}]
    The first inequality follows by restricting the admissible couplings to diagonal couplings.
For the second, fix any coupling $\gamma$ of $P$ and $R$ and let $m=\gamma\{W\ne W'\}$.
    Retain the restriction of $\gamma$ to $\{W=W'\}$.
    For each state $w\in\cW$, since $P$ and $R$ share the state marginal $Q$, we have
    \begin{equation*}
        \begin{split}
            \gamma(\{W=w,W'\ne w\})
            &=
            Q(\{w\})-\gamma(\{W=w,W'=w\})
            \\
            &=
            \gamma(\{W\ne w,W'=w\})
            \eqcolon m_w
            ,
        \end{split}
    \end{equation*}
    with $\sum_{w\in\cW}m_w=\gamma\{W\ne W'\}=m$.
    Thus, at each state $w$, the remaining first and second outcome marginals have equal total mass $m_w$; when $m_w>0$, we couple these remaining measures within state $w$ (for instance, via their normalized product measure).
    The resulting diagonal coupling has outcome cost at most
    \begin{equation*}
        \int_{\{W=W'\}}\|y-y'\|\dd\gamma+\diam(\cY)\,m.
    \end{equation*}
    The original product transport cost is at least
    \begin{equation*}
        \int_{\{W=W'\}}\|y-y'\|\dd\gamma+\delta_\cW m.
    \end{equation*}
    Hence
    \begin{equation*}
        \begin{split}
            \eW_1(P,R)
            &\le
            \int_{\{W=W'\}}\|y-y'\|\dd\gamma+\diam(\cY)\,m
            \\
            &\le
            \max\{1,\diam(\cY)/\delta_\cW\}
            \Bigl(
                \int_{\{W=W'\}}\|y-y'\|\dd\gamma+\delta_\cW m
            \Bigr)
            ,
        \end{split}
    \end{equation*}
    and taking the infimum over $\gamma$ proves the claim.
\end{proof}

\subsection{Population Identity}
\label{app:population-identity-proof}

\begin{proof}[Proof of \eqref{eq:population-identity-short}]
    Let $g$ be measurable and let $P_g(\dd w,\dd y)=Q_\rho(\dd w)\nu_{g,w}(\dd y)$.
    For every bounded measurable $h\in\cF$, the identity in the main text
    \begin{equation*}
        \E_{P_\obs}[w_\rho(W)\varphi(W,Y^\obs)]
        =
        \E_{P^\ast}[\varphi(W,Y)]
    \end{equation*}
    gives
    \begin{equation*}
        \E_{P_{\obs}}[w_\rho(W)h(W,Y^{\obs})]
        =
        \E_{P^\ast}[h(W,Y)]
        .
    \end{equation*}
    Since $\tilde{W}\sim Q_\rho$ independently of $U$,
    \begin{equation*}
        \E[h(\tilde{W},g(\tilde{W},U))]
        =
        \E_{P_g}[h(W,Y)]
        .
    \end{equation*}
    Hence $L(g,h)=\E_{P^\ast}[h]-\E_{P_g}[h]$.
    The laws $P^\ast$ and $P_g$ share the marginal $Q_\rho$, so Lemma~\ref{lem:extended-kr-duality-app} proves the identity.
\end{proof}

\section{Proofs of Section~\ref{sec:discrete-minimax}}
\label{app:discrete-proofs}

\subsection{Imported One-State WGAN Primitive}
\label{app:one-state-sieve}

Let $U\sim\Unif([0,1]^{d_U})$ and set
\begin{equation}
\label{eq:def-a-app}
    a
    \coloneq
    \min\Biggl\{
        \frac{\beta+1}{2\beta+d_U},\frac{1}{2}
    \Biggr\}
    .
\end{equation}
The one-state construction uses the generator approximation and covering bounds of \citep{stephanovitch2024wasserstein} for their general latent-pushforward model.
For each integer $N\ge2$, take a finite $N^{-1}$-net of their bounded neural generator sieve in the uniform norm, and include the constant generator $g(u)=y_0$.
Denote the resulting class by $\cG_N$, and fix an enumeration with this constant generator first.
The sieve and its wavelet approximation precision can be chosen so that every generator takes values in $\cY$ and
\begin{align}
\label{eq:onestate-generator-properties-app}
    &\sup_{g\in\cG_N}\Lip(g)
    \le
    B_0
    ,
    \quad\log(2|\cG_N|)
    \le
    B_0N^{d_U/(2\beta+d_U)}(\log N)^{c_1}
    ,
    \\
\label{eq:onestate-generator-approximation-app}
    &A_N
    \coloneq
    \sup_{g^\ast\in\mathcal{H}_{K_0}^{\beta+1}}
    \min_{g\in\cG_N}
    W_1((g^\ast)_\#U,g_\#U)
    \le
    B_0(\log N)^{c_1}N^{-(\beta+1)/(2\beta+d_U)}
    .
\end{align}
The constants $B_0,c_1>0$, like all constants below, may depend on the fixed dimensions and model parameters.
To obtain the bounded sieve, restrict the wavelet generator class to maps with uniform norm at most $2K_0$.
The truncated wavelet expansion approximates each target uniformly at order $N^{-(\beta+1)/(2\beta+d_U)}$, and its neural basis functions can be chosen with arbitrarily small $C^1$ approximation error.
Consequently, for all sufficiently large $N$, these approximants remain in the bounded sieve.
The wavelet coefficient bounds give a uniform first-derivative bound, since $\beta\ge1$;
choose the neural approximation precision so that the additional first-derivative error is at most one.
Passage to a finite net preserves the displayed bounds after changing their constants.
For the finitely many smaller $N$, take $\cG_N=\{g:g(u)=y_0\}$ and enlarge the constants.
These approximation and covering arguments apply coordinatewise for every $p\ge1$.

For $p=1$, the generator construction can equivalently be applied in $\mathbb{R}^2$ to $(g^\ast,0)$ and projected onto its first coordinate;
the approximation, covering, and Lipschitz bounds are preserved.

For each pair $g,g'\in\cG_N$, choose an anchored Kantorovich potential $D_{g,g'}\in\mathcal{L}_1$ such that
\begin{equation*}
    \int D_{g,g'}\dd(g_\#U-g'_\#U)
    =
    W_1(g_\#U,g'_\#U)
    ,
\end{equation*}
and set
\begin{equation}
\label{eq:onestate-net-critics-app}
    \mathcal{D}_N
    =
    \{0\}\cup\{\pm D_{g,g'}:g,g'\in\cG_N\}
    .
\end{equation}
These finitely many potentials exist by compactness of $\mathcal{L}_1$ in the uniform norm.
Thus $\mathcal{D}_N$ is symmetric, is contained in $\mathcal{L}_1$, and satisfies
\begin{equation}
\label{eq:onestate-exact-sieve-dual-app}
    d_{\mathcal{D}_N}(g_\#U,g'_\#U)
    =
    W_1(g_\#U,g'_\#U)
    ,
    \quad
    g,g'\in\cG_N
    ,
\end{equation}
where $d_{\mathcal{D}_N}(\eta,\nu)=\max_{D\in\mathcal{D}_N} \int D\dd(\eta-\nu)$.
This is the finite optimal-potential critic construction for the general pushforward model in \citep{stephanovitch2024wasserstein}.

For $\eta\in\cP(\cY)$, let $Z_1,\ldots,Z_k$ be i.i.d. with law $\eta$, and let $U_1,\ldots,U_\ell$ be independent copies of $U$, independent of the observations.
Risk expectations average over both samples and any independent estimator randomization; the subscript in $\E_{\eta}$ specifies the observational outcome law.
For $k,\ell\ge1$, define
\begin{equation}
\label{eq:onestate-two-sample-objective-app}
    J_{k,\ell,N}(g)
    =
    \max_{D\in\mathcal{D}_N}\Biggl\{
        \frac{1}{k}\sum_{i=1}^kD(Z_i)
        -
        \frac{1}{\ell}\sum_{i=1}^\ell D(g(U_i))
    \Biggr\}
    .
\end{equation}
The first minimizer in the fixed enumeration of $\cG_N$ is a measurable estimator.
For a deterministic tolerance $\tau\ge0$, a $\tau$-approximate minimizer has criterion value at most $\inf_{g\in\cG_N}J_{k,\ell,N}(g)+\tau$.
For a measurable map $g_0:[0,1]^{d_U}\to\cY$, write $\eta_0=(g_0)_\#U$.

\begin{lemma}[One-state WGAN oracle primitive]
\label{lem:onestate-oracle-primitive-app}
    Suppose $g_0\in\mathcal{H}_{K_0}^{\beta+1}$ and $\hat{g}\in\cG_N$ is a measurable $\tau$-approximate minimizer of \eqref{eq:onestate-two-sample-objective-app} based on observations from $\eta_0$.
    For every fixed $c_\ast>0$, if $\min\{k,\ell\}\ge c_\ast N$, then
    \begin{equation}
    \label{eq:onestate-oracle-rate-app}
        \sup_{g_0\in\mathcal{H}_{K_0}^{\beta+1}}
        \E[
            W_1(\eta_0,\hat{g}_\#U)
        ]
        \le
        C(\log N)^cN^{-a}
        +
        \tau
        .
    \end{equation}
    The constants may also depend on $c_\ast$.
\end{lemma}

\begin{proof}[Proof of Lemma~\ref{lem:onestate-oracle-primitive-app}]
    Choose $\bar{g}\in\cG_N$ with $W_1(\eta_0,\bar{g}_\#U)\le A_N$.
    Since $\mathcal{D}_N\subset\mathcal{L}_1$ is symmetric, $d_{\mathcal{D}_N}$ is a symmetric pseudometric with $d_{\mathcal{D}_N}\le W_1$ by Kantorovich--Rubinstein duality.
    For $g\in\cG_N$, the triangle inequality for $W_1$, \eqref{eq:onestate-exact-sieve-dual-app}, and the triangle inequality for $d_{\mathcal{D}_N}$ therefore give
    \begin{equation*}
        \begin{split}
            W_1(\eta_0,g_\#U)
            &\le
            A_N+W_1(\bar{g}_\#U,g_\#U)
            =
            A_N+d_{\mathcal{D}_N}(\bar{g}_\#U,g_\#U)
            \\
            &\le
            A_N+d_{\mathcal{D}_N}(\bar{g}_\#U,\eta_0)+d_{\mathcal{D}_N}(\eta_0,g_\#U)
            \\
            &\le
            2A_N
            +
            d_{\mathcal{D}_N}(\eta_0,g_\#U)
            .
        \end{split}
    \end{equation*}
    Write
    \begin{equation*}
        \begin{split}
            S_k
            &=
            \max_{D\in\mathcal{D}_N}\Biggl|
                \frac{1}{k}\sum_{i=1}^kD(Z_i)-\int D\dd\eta_0
            \Biggr|
            ,\\
            T_\ell
            &=
            \max_{g\in\cG_N,\,D\in\mathcal{D}_N}\Biggl|
                \frac{1}{\ell}\sum_{i=1}^\ell D(g(U_i))-\E[D(g(U))]
            \Biggr|
            .
        \end{split}
    \end{equation*}
    Since $\mathcal{D}_N$ is finite, $|J_{k,\ell,N}(g)-d_{\mathcal{D}_N}(\eta_0,g_\#U)|\le S_k+T_\ell$ for every $g\in\cG_N$.
    The approximate minimizing property of $\hat{g}$ therefore gives
    \begin{equation*}
        \begin{split}
            d_{\mathcal{D}_N}(\eta_0,\hat{g}_\#U)
            &\le
            J_{k,\ell,N}(\hat{g})+S_k+T_\ell
            \\
            &\le
            J_{k,\ell,N}(\bar{g})+\tau+S_k+T_\ell
            \\
            &\le
            d_{\mathcal{D}_N}(\eta_0,\bar{g}_\#U)+2S_k+2T_\ell+\tau
            \\
            &\le
            A_N+2S_k+2T_\ell+\tau
            .
        \end{split}
    \end{equation*}
    Combined with the previous display, this is the deterministic basic inequality
    \begin{equation}
    \label{eq:onestate-basic-inequality-app}
        W_1(\eta_0,\hat{g}_\#U)
        \le
        3A_N+2S_k+2T_\ell+\tau
        .
    \end{equation}

    Since every $D\in\mathcal{D}_N\subset\mathcal{L}_1$ satisfies $\|D\|_\infty\le\diam(\cY)$ and $\Lip(D)\le1$, while $\sup_{g\in\cG_N}\Lip(g)\le B_0$ by \eqref{eq:onestate-generator-properties-app} and $\Lip(g_0)\le C$, each of the finite classes $\{D\circ g_0:D\in\mathcal{D}_N\}$ and $\{D\circ g:D\in\mathcal{D}_N,g\in\cG_N\}$ consists of functions on $[0,1]^{d_U}$ uniformly bounded by $\diam(\cY)$ and uniformly Lipschitz with constant at most $\max\{B_0,C\}$.
    Let $\mathcal{A}$ denote either class, and write $Pf=\E[f(U)]$ and $P_sf=s^{-1}\sum_{i=1}^s f(V_i)$ for independent copies $V_1,\ldots,V_s$ of $U$.
    Let $\mathcal{N}(\epsilon,\mathcal{A},\|\cdot\|_\infty)$ be the minimum number of uniform-norm balls of radius $\epsilon$ needed to cover $\mathcal{A}$.
    By \eqref{eq:onestate-generator-properties-app} and $|\mathcal{D}_N|\le2|\cG_N|^2+1$, counting the finite class gives $|\mathcal{A}|\le|\cG_N|(2|\cG_N|^2+1)$ and hence the first bound below, while the classical Kolmogorov--Tikhomirov entropy bound for uniformly bounded Lipschitz functions on $[0,1]^{d_U}$ gives the second:
    \begin{equation*}
        \log\mathcal{N}(\epsilon,\mathcal{A},\|\cdot\|_\infty)
        \le
        C\min\bigl\{
            N^{d_U/(2\beta+d_U)}(\log N)^{c_1},\epsilon^{-d_U}
        \bigr\}
        ,\quad
        0<\epsilon<1
        .
    \end{equation*}
    Apply Dudley's entropy-integral bound
    \begin{equation*}
        \E[\sup_{f\in\mathcal{A}}|(P_s-P)f|]
        \le
        Cs^{-1/2}\int_0^{\diam(\cY)}\sqrt{\log\mathcal{N}(\epsilon,\mathcal{A},\|\cdot\|_\infty)}
        \dd\epsilon
        ,
    \end{equation*}
    splitting the integral over $(0,t_N]$, $(t_N,1]$, and $(1,\diam(\cY)]$ with $t_N=N^{-1/(2\beta+d_U)}$, and using the cardinality bound on $(0,t_N]$ and the Lipschitz entropy bound on $(t_N,1]$.
    For a sample size $s\ge c_\ast N$, this yields
    \begin{equation*}
        \E\bigl[
            \sup_{f\in\mathcal{A}}|(P_s-P)f|
        \bigr]
        \le
        \frac{C}{\sqrt{s}}
        \Biggl\{
            t_NN^{d_U/(2(2\beta+d_U))}(\log N)^{c_1/2}
            +\int_{t_N}^{1}\epsilon^{-d_U/2}\dd\epsilon+1
        \Biggr\}
        .
    \end{equation*}
    Here $t_NN^{d_U/(2(2\beta+d_U))}=N^{(d_U-2)/(2(2\beta+d_U))}$, and $\int_{t_N}^1\epsilon^{-d_U/2}\dd\epsilon$ is bounded by a constant for $d_U<2$, by $\log(1/t_N)\lesssim\log N$ for $d_U=2$, and by $Ct_N^{1-d_U/2}=CN^{(d_U-2)/(2(2\beta+d_U))}$ for $d_U>2$.
    Since $s\ge c_\ast N$, the right-hand side is at most $C(\log N)^cN^{-1/2}$ when $d_U\le2$ and at most $C(\log N)^cN^{-1/2+(d_U-2)/(2(2\beta+d_U))}=C(\log N)^cN^{-(\beta+1)/(2\beta+d_U)}$ when $d_U>2$, that is, at most $C(\log N)^cN^{-a}$ in both cases.
    For $S_k$, represent the observations as $g_0(V_i)$ with independent uniform $V_i$;
    for $T_\ell$, apply the same bound directly to the noise sample.
    Substituting these bounds and \eqref{eq:onestate-generator-approximation-app} into \eqref{eq:onestate-basic-inequality-app} proves the result.
\end{proof}

In the following minimax risk, $\hat{\eta}_k$ ranges over all measurable, possibly randomized estimators based on $k$ i.i.d. observations and taking values in $\cP(\cY)$.

\begin{lemma}[One-state minimax lower bound]
\label{lem:onestate-lower-bound-app}
    Under the stated dimension and smoothness conditions, for all integers $k\ge1$,
    \begin{equation}
    \label{eq:onestate-lower-rate-app}
        \inf_{\hat{\eta}_k}
        \sup_{g^\ast\in\mathcal{H}_{K_0}^{\beta+1}}
        \E[
            W_1((g^\ast)_\#U,\hat{\eta}_k)
        ]
        \ge
        c(k\vee1)^{-a}
        .
    \end{equation}
\end{lemma}

\begin{proof}[Proof of Lemma~\ref{lem:onestate-lower-bound-app}]
    For sufficiently large $k$, apply Lemma~\ref{lem:onestate-assouad-primitive} with $N=k$ and $\zeta\le1/16$, and write $\eta_\theta=\eta_\theta^k$.
    With the Hellinger normalization specified in Appendix~\ref{app:continuous-minimax-lower-bound}, the tensorization bound $H^2(P^{\otimes k},Q^{\otimes k})\le kH^2(P,Q)$ and Lemma~\ref{lem:onestate-assouad-primitive}(ii) give $H^2(\eta_\theta^{\otimes k},\eta_{\theta'}^{\otimes k})\le\zeta\le1/16$ for every pair $(\theta,\theta')$ differing in one bit, hence $\TV(\eta_\theta^{\otimes k},\eta_{\theta'}^{\otimes k})\le\sqrt{H^2(\eta_\theta^{\otimes k},\eta_{\theta'}^{\otimes k})}\le1/4$.
    Given any estimate $\hat{\eta}_k$, let $\hat{\theta}$ index its closest vertex in $W_1$, breaking ties lexicographically.
    The triangle inequality and the choice of $\hat{\theta}$ give $W_1(\eta_\theta,\eta_{\hat{\theta}})\le W_1(\eta_\theta,\hat{\eta}_k)+W_1(\hat{\eta}_k,\eta_{\hat{\theta}})\le2W_1(\eta_\theta,\hat{\eta}_k)$, so Lemma~\ref{lem:onestate-assouad-primitive}(iii) yields
    \begin{equation*}
        W_1(\eta_\theta,\hat{\eta}_k)
        \ge
        \frac{c\Delta_k}{2}d_H(\theta,\hat{\theta})
        .
    \end{equation*}
    Under the uniform prior on $\theta\in\{0,1\}^{K_k}$, Assouad's lemma bounds the error probability of each bit below by $(1-1/4)/2=3/8$, so the prior-averaged risk satisfies $\E[d_H(\theta,\hat{\theta})]\ge3K_k/8$.
    Hence the maximal risk is at least $cK_k\Delta_k\asymp k^{-a}$.
    The construction in Appendix~\ref{app:continuous-minimax-lower-bound} proves this family directly using smooth transport perturbations, following the testing approach of \citep{tang2023minimax}.
    For the finitely many smaller $k$, monotonicity of minimax risk in sample size extends the bound after decreasing $c$.
    Independent estimator randomization preserves the testing bound.
\end{proof}

\begin{lemma}[Binomial negative moment]
\label{lem:binomial-negative-moment-discrete-app}
    Let $K\sim\mathrm{Binomial}(n,p)$, $a>0$, and $q\ge0$.
    There exists $C<\infty$, depending only on $(a,q)$, such that
    \begin{equation*}
        \E[(\log(K\vee2))^q(K\vee1)^{-a}]
        \le
        C(\log(n\vee2))^q(np\vee1)^{-a}
        .
    \end{equation*}
\end{lemma}

\begin{proof}[Proof of Lemma~\ref{lem:binomial-negative-moment-discrete-app}]
    If $np<2$, the left side is bounded by a constant and $(np\vee1)^{-a}\ge2^{-a}$.
    If $np\ge2$, split on $A=\{K\ge np/2\}$.
    On $A$,
    $(\log(K\vee2))^q(K\vee1)^{-a}\le C(\log(n\vee2))^q(np)^{-a}$.
    On $A^c$, the expression is at most $(\log(n\vee2))^q$, while Chernoff's bound gives $\Prob(A^c)\le\exp(-np/8)$.
    Since $\exp(-x/8)(\log(n\vee2))^q\le C x^{-a}(\log(n\vee2))^q$ for $x=np\ge2$, the claim follows.
\end{proof}

\subsection{Finite-State Upper Bound}
\label{app:finite-state-rate-proof}

For conditional outcome laws $\mu_j,\nu_j\in\cP(\cY)$, $j=1,\ldots,M$, write $P(\dd w,\dd y)=\sum_{j=1}^Mq_j\delta_j(\dd w)\mu_j(\dd y)$ and $R(\dd w,\dd y)=\sum_{j=1}^Mq_j\delta_j(\dd w)\nu_j(\dd y)$.

\begin{lemma}[Finite-state extended-Wasserstein identity]
\label{lem:finite-state-cw-identity-app}
    For $P$ and $R$ with the same finite conditioning marginal,
    \begin{equation*}
        \eW_1(P,R)
        =
        \sum_{j=1}^M q_jW_1(\mu_j,\nu_j)
        .
    \end{equation*}
\end{lemma}

\begin{proof}[Proof of Lemma~\ref{lem:finite-state-cw-identity-app}]
    This is Lemma~\ref{lem:diagonal-coupling-disintegration} specialized to finite $\cW$:
    the diagonal constraint forces the coupling to decompose independently over states.
\end{proof}

\begin{proof}[Proof of Theorem~\ref{thm:discrete-upper-bound}]
    Since critic coordinates are independent,
    $\sup_{h\in\cH_n^{\disc}}\hat{L}_n^{\disc}(g,h)$ is the sum of statewise suprema.
    The minimization over $g=(g_1,\ldots,g_M)$ therefore separates across states; in each state, choose the first minimizer in the fixed enumeration of $\cG_n$.
    Conditional on the state sequence, the outcomes with $W_i=j$ are i.i.d. from $\mu_j^\ast$.
    By overlap and $q_j\ge q_{\min}$, $\pi_j\ge\kappa q_j\ge\kappa q_{\min}$.
    On the event $A_j=\{N_j^{\obs}\ge n\pi_j/2\}$, $\min\{N_j^{\obs},n\}\ge(\kappa q_{\min}/2)n$.
    Conditional on $N_j^{\obs}$, the state-$j$ estimator is the one-state estimator of Appendix~\ref{app:one-state-sieve} with $N=n$, $k=N_j^{\obs}$, and $\ell=n$, so on $A_j$ Lemma~\ref{lem:onestate-oracle-primitive-app} with $c_\ast=\kappa q_{\min}/2$ gives
    \begin{equation*}
        \E[W_1(\mu_j^\ast,\hat{\mu}_j)\mid N_j^{\obs}]
        \le
        C(\log n)^cn^{-a}
        ,
    \end{equation*}
    where $\hat{\mu}_j$ is the generated law in state $j$.
    On $A_j^c$, the error is at most $\diam(\cY)$, and Chernoff's bound gives $\Prob(A_j^c)\le\exp(-n\pi_j/8)\le\exp(-\kappa q_{\min}n/8)$.
    Hence
    \begin{equation*}
        \E[W_1(\mu_j^\ast,\hat{\mu}_j)]
        \le
        C(\log n)^c n^{-a}
        +
        \diam(\cY)\exp(-\kappa q_{\min}n/8)
        \le
        C(\log n)^c n^{-a}
        .
    \end{equation*}
    By Lemma~\ref{lem:finite-state-cw-identity-app},
    \begin{equation*}
        \E[\eW_1(P^\ast,P_{\hat{g}_n^{\disc}})]
        =
        \sum_{j=1}^M q_j\E[W_1(\mu_j^\ast,\hat{\mu}_j)]
        \le
        C(\log n)^c n^{-a}
        ,
    \end{equation*}
    because $M$ is fixed and $\sum_jq_j=1$.
\end{proof}

\begin{proof}[Proof of Corollary~\ref{cor:discrete-minimax-optimality}]
    The upper bound is Theorem~\ref{thm:discrete-upper-bound}.
    For the lower bound, restrict to the submodel with $Q_{\obs}=Q_\rho$, conditional laws $\delta_0$ in states $j\ge2$, and only $\mu_1^\ast=(g^\ast)_\#U$ varying over $\mathcal{H}_{K_0}^{\beta+1}$.
    For any estimator $\bar{P}_n(\dd j,\dd y)=\sum_jq_j\delta_j(\dd j)\bar{\mu}_{n,j}(\dd y)$,
    \begin{equation*}
        \eW_1(P^\ast,\bar{P}_n)
        \ge
        q_1W_1(\mu_1^\ast,\bar{\mu}_{n,1})
        \ge
        q_{\min}W_1(\mu_1^\ast,\bar{\mu}_{n,1})
        .
    \end{equation*}
    This is a causal submodel:
    draw an independent uniform variable and use the corresponding statewise generators for the potential outcomes.
    From $n$ i.i.d. observations $Z_1,\ldots,Z_n\sim\mu_1^\ast$, independently draw $J_1,\ldots,J_n\sim Q_\rho$.
    Set $Y_i=Z_i$ when $J_i=1$ and $Y_i=0$ otherwise.
    The pairs $(J_i,Y_i)$ are exactly an observational sample from this finite-state submodel.
    Applying any finite-state estimator and retaining its first conditional law therefore gives a randomized one-state estimator based on $n$ observations.
    Lemma~\ref{lem:onestate-lower-bound-app} and the first display then give
    \begin{equation*}
        \sup_{\cM_{\disc}}
        \E[\eW_1(P^\ast,\bar{P}_n)]
        \ge
        cq_1n^{-a}
        \ge
        cq_{\min}n^{-a}
        .
    \end{equation*}
    Taking the infimum over estimators proves the claim.
\end{proof}

\section{Proofs of Section~\ref{sec:continuous-conditioning}}
\label{app:continuous-proofs}

\subsection{Continuous Finite-Resolution Geometry and Besov Embedding}
\label{app:continuous-geometry-proof}

Let $\mathcal{L}=\{f:\cY\to\R:f(y_0)=0,\ \Lip(f)<\infty\}$ with norm $\|f\|_{\mathcal{L}}=\Lip(f)$.
For $p_B\in[1,\infty)$ and $\boldsymbol{s}=(s_1,\ldots,s_{d_W})\in(0,1)^{d_W}$, define the vector-valued anisotropic Besov--Nikolskii norm of a Bochner-measurable map $H:\cW\to\mathcal{L}$ by
\begin{equation}
    \label{eq:besov-norm-app}
    \|H\|_{B_{p_B,\infty}^{\boldsymbol{s}}(\cW;\mathcal{L})}
    \coloneq
    \|H\|_{L^{p_B}(\cW;\mathcal{L})}
    +
    \max_{1\le j\le d_W}
    \sup_{0<\tau\le1}
    \tau^{-s_j}
    \|H(\cdot+\tau e_j)-H(\cdot)\|_{L^{p_B}(\cW_{j,\tau};\mathcal{L})}
    ,
\end{equation}
where $\cW_{j,\tau}=\{w\in\cW:w+\tau e_j\in\cW\}$.
For measurable $S\subset\cW$, $\|H\|_{L^{p_B}(S;\mathcal{L})}=(\int_S\|H(w)\|_{\mathcal{L}}^{p_B}\,\dd w)^{1/p_B}$; all such norms in this subsection use Lebesgue measure.
We use Borel, strongly measurable representatives of $H$, setting them equal to zero on exceptional null sets. Under Assumption~\ref{ass:design-regularity-main}, these changes also occur on $Q_\rho$-null sets.
The anchored Lipschitz space $\mathcal{L}$ is a Banach space, and $|f(y)-g(y)|\le\diam(\cY)\|f-g\|_{\mathcal{L}}$, so evaluation is continuous.
In particular, finite-range step maps are strongly measurable and yield jointly measurable critics.
This is the anisotropic Besov--Nikolskii structure used in classical nonparametric theory and modern deep approximation theory \citep{triebel1983theory,neumann1997wavelet,kerkyacharian2001nonlinear,hoffmann2002random,suzuki2019adaptivity,suzuki2021deep}.
For a radius $L>0$, define
\begin{equation}
    \label{eq:besov-critic-class-app}
    \cH_{\boldsymbol{s},p_B}(L)
    \coloneq
    \{
        h(w,y)
        =
        H(w)(y):
        \|H\|_{B_{p_B,\infty}^{\boldsymbol{s}}(\cW;\mathcal{L})}
        \le
        L,
        \ H(w)\in\mathcal{L}_1\text{ for a.e. }w
    \}
    .
\end{equation}

For laws $P(\dd w,\dd y)=Q_\rho(\dd w)\mu_w(\dd y)$ and $R(\dd w,\dd y)=Q_\rho(\dd w)\nu_w(\dd y)$, define, for $q_C>0$,
\begin{equation}
\label{eq:cell-average-distributions-app}
    \mu_C^\rho=q_C^{-1}\int_C\mu_wQ_\rho(\dd w)
    ,
    \quad
    \nu_C^\rho=q_C^{-1}\int_C\nu_wQ_\rho(\dd w)
    .
\end{equation}
Under Assumption~\ref{ass:design-regularity-main}, every dyadic cell has positive $q_C$.

Define the finite-resolution discrepancy by
\begin{equation*}
    \eW_1^{(\boldsymbol{m})}(P,R)
    \coloneq
    \sup_{h\in\cF^{(\boldsymbol{m})}}
    \{\E_P[h]-\E_R[h]\}
    .
\end{equation*}

\begin{proposition}[Coarsened conditional Wasserstein identity]
\label{prop:coarsened-cw-identity-app}
    For $P$ and $R$ sharing $Q_\rho$,
    \begin{equation}
    \label{eq:coarsened-cw-identity-app}
        \eW_1^{(\boldsymbol{m})}(P,R)
        =
        \sum_{C\in\Pi_{\boldsymbol{m}}} q_CW_1(\mu_C^\rho,\nu_C^\rho)
        .
    \end{equation}
\end{proposition}

\begin{proof}[Proof of Proposition~\ref{prop:coarsened-cw-identity-app}]
    Each $h\in\cF^{(\boldsymbol{m})}$ has form $h(w,y)=\sum_C\1\{w\in C\}f_C(y)$ with $f_C\in\mathcal{L}_1$.
    Expanding expectations gives
    \begin{equation*}
        \E_P[h]-\E_R[h]
        =
        \sum_Cq_C\Biggl\{
            \int f_C\dd\mu_C^\rho-\int f_C\dd\nu_C^\rho
        \Biggr\}
        .
    \end{equation*}
    The choices of $f_C$ are independent across cells, so the supremum separates.
    Ordinary Kantorovich--Rubinstein duality on $\cY$ proves the claim.
\end{proof}

\begin{proposition}[Stepwise critics are anisotropic Besov--Nikolskii critics]
\label{prop:stepwise-besov-app}
    Fix $p_B\in[1,\infty)$ and $\boldsymbol{s}$ satisfying $0<s_j<1/p_B$ for all $j$.
    There exists $C<\infty$, depending only on $(p_B,\boldsymbol{s},d_W)$, such that, for every $\boldsymbol{m}$,
    \begin{equation}
    \label{eq:stepwise-besov-app}
        \cF^{(\boldsymbol{m})}
        \subset
        \cH_{\boldsymbol{s},p_B}\bigl(
            C\max_j2^{m_js_j}
        \bigr)
        .
    \end{equation}
\end{proposition}

\begin{proof}[Proof of Proposition~\ref{prop:stepwise-besov-app}]
    Let $h(w,y)=\sum_C\1\{w\in C\}f_C(y)$ and set $H(w)=h(w,\cdot)$.
    Since $\Lip(f_C)\le1$, $\|H\|_{L^{p_B}(\cW;\mathcal{L})}\le1$.
    Fix $j$ and $\tau\in(0,1]$.
    Since $w$ and $w+\tau e_j$ lie in different cells only if one of the $2^{m_j}-1$ internal grid hyperplanes orthogonal to $e_j$ separates them, the set $A_{j,\tau}=\{w\in\cW_{j,\tau}:H(w+\tau e_j)\ne H(w)\}$ is contained in the union of the slabs of width $\tau$ below these hyperplanes, and its Lebesgue measure is bounded by $\min\{1,(2^{m_j}-1)\tau\}\le C\min\{1,2^{m_j}\tau\}$.
    On $A_{j,\tau}$, $\|H(w+\tau e_j)-H(w)\|_{\mathcal{L}}\le2$.
    Therefore
    \begin{equation*}
        \|H(\cdot+\tau e_j)-H(\cdot)\|_{L^{p_B}(\cW_{j,\tau};\mathcal{L})}
        \le
        C\min\{1,(2^{m_j}\tau)^{1/p_B}\}
        .
    \end{equation*}
    If $\tau\le2^{-m_j}$, since $1/p_B-s_j>0$, multiplying by $\tau^{-s_j}$ gives $C2^{m_j/p_B}\tau^{1/p_B-s_j}\le C2^{m_j/p_B}(2^{-m_j})^{1/p_B-s_j}=C2^{m_js_j}$.
    If $\tau>2^{-m_j}$, the minimum is $1$, so multiplying by $\tau^{-s_j}$ gives $C\tau^{-s_j}\le C(2^{-m_j})^{-s_j}=C2^{m_js_j}$.
    Taking suprema over $\tau\in(0,1]$ and $1\le j\le d_W$ proves \eqref{eq:stepwise-besov-app}.
\end{proof}

For a generator $g$ and $q_C > 0$, define the target-design generated cell law
\begin{equation*}
    \nu_{g,C}^\rho
    =
    q_C^{-1}\int_C\nu_{g,w}Q_\rho(\dd w)
    ,
\end{equation*}
with arbitrary values on cells of zero target mass, and define the cell-resolution generated law $P_g^{(\boldsymbol{m})}(\dd w,\dd y)=Q_\rho(\dd w)\nu_{g,C_{\boldsymbol{m}}(w)}^\rho(\dd y)$.

\begin{proposition}[Finite-resolution sandwich and raw-law oscillation]
\label{prop:besov-sandwich-app}
    Under Assumption~\ref{ass:design-regularity-main}, with $p_B,\boldsymbol{s},C$ as in Proposition~\ref{prop:stepwise-besov-app}, for any $P,R$ sharing $Q_\rho$,
    \begin{equation}
    \label{eq:finite-resolution-sandwich-app}
        \eW_1^{(\boldsymbol{m})}(P,R)
        \le
        \sup_{h\in\cH_{\boldsymbol{s},p_B}(C\max_j2^{m_js_j})}\{\E_P[h]-\E_R[h]\}
        \le
        \eW_1(P,R)
        .
    \end{equation}
    If $R(\dd w,\dd y)=Q_\rho(\dd w)\nu_C(\dd y)$ for $w\in C$ and $P=P^\ast$ satisfies Assumption~\ref{ass:anisotropic-holder-main}, then
    \begin{equation}
    \label{eq:exact-to-coarse-bias-app}
        \eW_1(P^\ast,R)
        \le
        \eW_1^{(\boldsymbol{m})}(P^\ast,R)+L_\ast b_{\boldsymbol{m}}
        .
    \end{equation}
    Finally, for any generator $g$,
    \begin{equation}
    \label{eq:raw-oscillation-triangle-app}
        \eW_1(P^\ast,P_g)
        \le
        \eW_1(P^\ast,P_g^{(\boldsymbol{m})})
        +
        \int W_1(\nu_{g,w},\nu_{g,C_{\boldsymbol{m}}(w)}^\rho)Q_\rho(\dd w)
        .
    \end{equation}
\end{proposition}

\begin{proof}[Proof of Proposition~\ref{prop:besov-sandwich-app}]
    The first inequality in \eqref{eq:finite-resolution-sandwich-app} follows from Proposition~\ref{prop:stepwise-besov-app}.
    The second follows from extended Kantorovich--Rubinstein duality, because every admissible Besov critic is sectionwise 1-Lipschitz in $y$.

    For \eqref{eq:exact-to-coarse-bias-app}, use the target cell mixture $\mu_C^\rho$ from \eqref{eq:cell-average-distributions-app} with $P=P^\ast$.
    Since $R$ is cellwise constant,
    \begin{equation*}
        \eW_1(P^\ast,R)
        =
        \sum_C\int_C W_1(\mu_w^\ast,\nu_C)Q_\rho(\dd w)
        .
    \end{equation*}
    The triangle inequality gives
    \begin{equation*}
        W_1(\mu_w^\ast,\nu_C)
        \le
        W_1(\mu_w^\ast,\mu_C^\rho)+W_1(\mu_C^\rho,\nu_C)
        .
    \end{equation*}
    For $w,w'\in C$, Assumption~\ref{ass:anisotropic-holder-main} yields
    $W_1(\mu_w^\ast,\mu_{w'}^\ast)\le L_\ast b_{\boldsymbol{m}}$.
    Convexity of $W_1$ in its second argument gives
    $W_1(\mu_w^\ast,\mu_C^\rho)\le L_\ast b_{\boldsymbol{m}}$.
    Summing over cells and using Proposition~\ref{prop:coarsened-cw-identity-app} proves \eqref{eq:exact-to-coarse-bias-app}.

    Equation~\eqref{eq:raw-oscillation-triangle-app} is the triangle inequality for $\eW_1$ and the disintegration formula:
    \begin{equation*}
        \eW_1(P_g,P_g^{(\boldsymbol{m})})
        =
        \int W_1(\nu_{g,w},\nu_{g,C_{\boldsymbol{m}}(w)}^\rho)Q_\rho(\dd w)
        .
    \end{equation*}
\end{proof}

\subsection{Stratified Finite-Resolution Upper Bound}
\label{app:continuous-upper-proof}

For theoretical calibration, use $\mathcal{D}_{N_{\boldsymbol{m}}}$, where $\mathcal{D}_N$ is the one-state WGAN critic class from Appendix~\ref{app:one-state-sieve}.
Throughout this subsection, $\mu_C^\rho$ denotes the target cell mixture in \eqref{eq:cell-average-distributions-app} with $\mu_w=\mu_w^\ast$.
For $C\in\Pi_{\boldsymbol{m}}$, define
\begin{equation}
\label{eq:obs-target-cell-mixtures-app}
    \mu_C^{\obs}
    =
    \pi_C^{-1}
    \int_C\mu_w^\ast Q_{\obs}(\dd w)
    .
\end{equation}
Overlap gives $\pi_C\ge\kappa q_C$.

\begin{lemma}[Cell-normalized sampling]
\label{lem:cell-normalized-sampling-app}
    Conditional on $N_C^{\obs}$, the outcomes $\{Y_i^{\obs}:W_i\in C\}$ are i.i.d. from $\mu_C^{\obs}$.
    Moreover, $N_C^{\obs}\sim\mathrm{Binomial}(n,\pi_C)$ and $\pi_C\ge\kappa q_C$.
\end{lemma}

\begin{proof}[Proof of Lemma~\ref{lem:cell-normalized-sampling-app}]
    By consistency and ignorability, the conditional law of $Y^{\obs}$ given $W=w$ is $\mu_w^\ast$.
    Conditioning on $W\in C$ averages this kernel with respect to $Q_{\obs}(\dd w\mid C)$, giving $\mu_C^{\obs}$.
    The count is binomial by i.i.d. sampling.
    Finally, $q_C=\int_Cw_\rho(w)Q_{\obs}(\dd w)\le\kappa^{-1}\pi_C$.
\end{proof}

\begin{lemma}[Observed-target cell bias]
\label{lem:observed-target-cell-bias-app}
    Under Assumption~\ref{ass:anisotropic-holder-main},
    \begin{equation}
    \label{eq:observed-target-cell-bias-app}
        W_1(\mu_C^{\obs},\mu_C^\rho)
        \le
        L_\ast b_{\boldsymbol{m}}
    \end{equation}
    for every $C\in\Pi_{\boldsymbol{m}}$.
\end{lemma}

\begin{proof}[Proof of Lemma~\ref{lem:observed-target-cell-bias-app}]
    By convexity of $W_1$ in both arguments,
    \begin{equation*}
        W_1(\mu_C^{\obs},\mu_C^\rho)
        \le
        \int_C\int_C W_1(\mu_w^\ast,\mu_{w'}^\ast)
        Q_{\obs}(\dd w\mid C)Q_\rho(\dd w'\mid C)
        .
    \end{equation*}
    For $w,w'\in C$, $|w_j-w'_j|\le2^{-m_j}$.
    Assumption~\ref{ass:anisotropic-holder-main} proves the claim.
\end{proof}

\begin{lemma}[Proxy-to-exact inequality for the finite-resolution generator]
\label{lem:proxy-to-exact-app}
    Under Assumptions~\ref{ass:overlap-general}, \ref{ass:design-regularity-main}, and \ref{ass:anisotropic-holder-main}, for any measurable conditional generator $g$,
    \begin{equation}
    \label{eq:proxy-to-exact-app}
        \eW_1(P^\ast,P_g^{(\boldsymbol{m})})
        \le
        \sum_{C\in\Pi_{\boldsymbol{m}}} q_CW_1(\mu_C^{\obs},\nu_{g,C}^\rho)
        +
        L_\ast b_{\boldsymbol{m}}
        .
    \end{equation}
\end{lemma}

\begin{proof}[Proof of Lemma~\ref{lem:proxy-to-exact-app}]
    For $w\in C$, convexity and Assumption~\ref{ass:anisotropic-holder-main} give
    \begin{equation*}
        W_1(\mu_w^\ast,\mu_C^{\obs})
        \le
        \int_C W_1(\mu_w^\ast,\mu_{w'}^\ast)Q_{\obs}(\dd w'\mid C)
        \le
        L_\ast b_{\boldsymbol{m}}
        .
    \end{equation*}
    Thus
    \begin{equation*}
        W_1(\mu_w^\ast,\nu_{g,C}^\rho)
        \le
        L_\ast b_{\boldsymbol{m}}
        +
        W_1(\mu_C^{\obs},\nu_{g,C}^\rho)
        .
    \end{equation*}
    Integrating over $C$ under $Q_\rho$ and summing over cells proves the result.
\end{proof}

Use $\eta_0=(g_0)_\#U$ as in Appendix~\ref{app:one-state-sieve}.

\begin{lemma}[Structural one-state stability]
\label{lem:structural-onestate-stability-app}
    Suppose $\eta\in\cP(\cY)$, $g_0\in\mathcal{H}_{K_0}^{\beta+1}$, and $W_1(\eta,\eta_0)\le\delta$ for some $\delta\ge0$.
    Suppose $\hat{g}\in\cG_N$ is a measurable $\tau$-approximate minimizer of \eqref{eq:onestate-two-sample-objective-app}, based on $k$ i.i.d. observations from $\eta$ and $\ell$ independent copies of $U$, with the two samples independent.
    For each fixed $c_\ast>0$, if $\min\{k,\ell\}\ge c_\ast N$, then
    \begin{equation}
    \label{eq:perturbed-wgan-oracle-app}
        \E_{\eta}[
            W_1(\eta,\hat{g}_\#U)
        ]
        \le
        3\delta
        +
        C(\log N)^cN^{-a}
        +
        \tau
        .
    \end{equation}
\end{lemma}

\begin{proof}[Proof of Lemma~\ref{lem:structural-onestate-stability-app}]
    On an auxiliary probability space, draw i.i.d. pairs $(Z_i,Z_i^0)$ from an optimal coupling of $\eta$ and $\eta_0$, independently of the generated-noise sample.
    Set $\xi_k=k^{-1}\sum_{i=1}^k\|Z_i-Z_i^0\|$.
    Since every $D\in\mathcal{D}_N$ is $1$-Lipschitz, the empirical criteria computed with $(Z_i)$ and $(Z_i^0)$, denoted by $J$ and $J^0$, satisfy
    \begin{equation*}
        \max_{g\in\cG_N}
        |J(g)-J^0(g)|
        \le
        \xi_k
        .
    \end{equation*}
    Consequently, for every $g\in\cG_N$,
    \begin{equation*}
        J^0(\hat{g})
        \le
        J(\hat{g})+\xi_k
        \le
        J(g)+\tau+\xi_k
        \le
        J^0(g)+\tau+2\xi_k
        ,
    \end{equation*}
    so $J^0(\hat{g})\le\min_{g\in\cG_N}J^0(g)+\tau+2\xi_k$.
    Applying the samplewise inequality \eqref{eq:onestate-basic-inequality-app} to the clean sample $(Z_i^0)$ with random tolerance $\tau+2\xi_k$ yields $W_1(\eta_0,\hat{g}_\#U)\le 3A_N+2S_k+2T_\ell+\tau+2\xi_k$.
    The empirical-process terms are bounded as in Lemma~\ref{lem:onestate-oracle-primitive-app}, and $\E[\xi_k]=W_1(\eta,\eta_0)\le\delta$, so
    \begin{equation*}
        \E[W_1(\eta_0,\hat{g}_\#U)]
        \le
        3A_N+2\E[S_k]+2\E[T_\ell]+\tau+2\E[\xi_k]
        \le
        C(\log N)^cN^{-a}+\tau+2\delta
        .
    \end{equation*}
    Finally, the triangle inequality
    \begin{equation*}
        W_1(\eta,\hat{g}_\#U)
        \le
        W_1(\eta,\eta_0)+W_1(\eta_0,\hat{g}_\#U)
        \le
        \delta+W_1(\eta_0,\hat{g}_\#U)
    \end{equation*}
    proves the result.
\end{proof}

At resolution $\boldsymbol{m}$, consider the finite product class of generators satisfying $g(w,u)=g_C(u)$ for $w\in C$, where $g_C\in\cG_{N_{\boldsymbol{m}}}$ for every $C\in\Pi_{\boldsymbol{m}}$.
Using the observational sample and the same target-design and noise samples as in \eqref{eq:continuous-estimator-main}, choose in each cell the first minimizer of its empirical criterion over $\cG_{N_{\boldsymbol{m}}}$ and denote its law by $\hat{\nu}_C$.
This defines a measurable minimizer over the product class, because both the generator and critic choices separate across cells.
The zero-count conventions in \eqref{eq:def-cell-obs-average}--\eqref{eq:def-cell-generated-average} and the constant first generator specify the estimator also in empty cells.

\begin{lemma}[Uniform cell oracle]
\label{lem:uniform-cell-oracle-app}
    Uniformly over cells $C\in\Pi_{\boldsymbol{m}}$ and models in $\cM_{\cont}$,
    \begin{equation}
    \label{eq:uniform-cell-oracle-app}
        \E[
            W_1(\mu_C^{\obs},\hat{\nu}_C)
        ]
        \le
        Cb_{\boldsymbol{m}}
        +
        C(\log n)^c(\kappa n2^{-|\boldsymbol{m}|_1})^{-a}
        .
    \end{equation}
\end{lemma}

\begin{proof}[Proof of Lemma~\ref{lem:uniform-cell-oracle-app}]
    Choose any representative $w_C\in C$.
    Assumption~\ref{ass:pointwise-pushforward-main} gives $\mu_{w_C}^\ast=(g_{w_C}^\ast)_\#U$ with $g_{w_C}^\ast\in\mathcal{H}_{K_0}^{\beta+1}$.
    By Assumption~\ref{ass:anisotropic-holder-main} and convexity,
    \begin{equation*}
        W_1(\mu_C^{\obs},\mu_{w_C}^\ast)
        \le
        L_\ast b_{\boldsymbol{m}}
        .
    \end{equation*}
    Put $t=\kappa n2^{-|\boldsymbol{m}|_1}$.
    If $t<2$, the stated bound follows from the diameter bound after increasing its constant.
    Suppose $t\ge2$.
    The two cell counts satisfy
    \begin{equation*}
        N_C^{\obs}
        \sim
        \operatorname{Binomial}(n,\pi_C)
        ,
        \quad
        M_C^\rho
        \sim
        \operatorname{Binomial}(n,q_C)
        ,
    \end{equation*}
    and the two counts are independent.
    Moreover, Assumption~\ref{ass:design-regularity-main} gives $q_C\ge\underline{q}\,2^{-|\boldsymbol{m}|_1}$, so $nq_C\ge\underline{q}\,n2^{-|\boldsymbol{m}|_1}\ge\underline{q}\,t$ because $\kappa\le1$, and overlap gives $n\pi_C\ge\kappa nq_C\ge\underline{q}\,t$.
    Define
    \begin{equation*}
        A_C
        =
        \{N_C^{\obs}\ge n\pi_C/2,\ M_C^\rho\ge nq_C/2\}
        .
    \end{equation*}
    On $A_C$, since $N_{\boldsymbol{m}}\le2t$,
    \begin{equation*}
        \min\{N_C^{\obs},M_C^\rho\}
        \ge
        (\underline{q}/4)N_{\boldsymbol{m}}
        .
    \end{equation*}
    Conditional on both counts, the observed outcomes in the cell are i.i.d. from $\mu_C^{\obs}$ and the selected noise draws are independent copies of $U$.
    The routed generator is constant in the conditioning coordinate within the cell.
    Therefore, Lemma~\ref{lem:structural-onestate-stability-app}, with $k=N_C^{\obs}$, $\ell=M_C^\rho$, and $\delta=L_\ast b_{\boldsymbol{m}}$, gives on $A_C$
    \begin{equation*}
        \E[W_1(\mu_C^{\obs},\hat{\nu}_C)\mid N_C^{\obs},M_C^\rho]
        \le
        3L_\ast b_{\boldsymbol{m}}
        +
        C(\log N_{\boldsymbol{m}})^cN_{\boldsymbol{m}}^{-a}
        .
    \end{equation*}
    Chernoff's inequality gives $\Prob(A_C^c)\le2\exp(-\underline{q}\,t/8)$, and on $A_C^c$ we use $W_1\le\diam(\cY)$.
    Taking expectations over the counts,
    \begin{equation*}
        \E[W_1(\mu_C^{\obs},\hat{\nu}_C)]
        \le
        3L_\ast b_{\boldsymbol{m}}
        +
        C(\log N_{\boldsymbol{m}})^cN_{\boldsymbol{m}}^{-a}
        +
        2\diam(\cY)\exp(-\underline{q}\,t/8)
        .
    \end{equation*}
    Since $t\le N_{\boldsymbol{m}}\le2t$, $\log N_{\boldsymbol{m}}\lesssim\log n$, and $\exp(-\underline{q}\,t/8)\le Ct^{-a}$, the desired bound follows, uniformly over cells and over the stated model class.
\end{proof}

\begin{proposition}[Finite-resolution continuous upper bound]
\label{prop:continuous-finite-resolution-upper-app}
    Assume that the implemented estimator satisfies, uniformly over $\cM_{\cont}$, the finite-resolution transfer condition \eqref{eq:implementation-transfer-condition} with error $\varepsilon_{n,\boldsymbol{m}}$.
    There exist constants $C,c>0$, depending only on fixed model parameters, such that for every resolution $\boldsymbol{m}$,
    \begin{equation}
    \label{eq:continuous-upper-bound-app}
        \sup_{(P_{\obs},P^\ast)\in\cM_{\cont}}
        \E
        \bigl[
            \eW_1(P^\ast,\hat{P}_{n,\boldsymbol{m}}^{\cont})
        \bigr]
        \le
        C\{b_{\boldsymbol{m}}+(\log n)^c(\kappa n2^{-|\boldsymbol{m}|_1})^{-a}\}
        +
        \varepsilon_{n,\boldsymbol{m}}
        .
    \end{equation}
\end{proposition}

\begin{proof}[Proof of Proposition~\ref{prop:continuous-finite-resolution-upper-app}]
    First consider the ideal routed finite-resolution estimator with cell laws $\hat{\nu}_C$.
    By Lemma~\ref{lem:uniform-cell-oracle-app},
    \begin{equation*}
        \E\Biggl[
            \sum_Cq_CW_1(\mu_C^{\obs},\hat{\nu}_C)
        \Biggr]
        \le
        Cb_{\boldsymbol{m}}
        +
        C(\log n)^c(\kappa n2^{-|\boldsymbol{m}|_1})^{-a}
        ,
    \end{equation*}
    because $\sum_Cq_C=1$.
    Lemma~\ref{lem:proxy-to-exact-app} transfers this proxy risk to the extended-Wasserstein risk of the routed estimator, adding only another constant multiple of $b_{\boldsymbol{m}}$.

    For the implemented estimator, the transfer condition \eqref{eq:implementation-transfer-condition} gives
    \begin{equation*}
        \E\Biggl[
            \sum_Cq_CW_1(\nu_{\hat{g}_{n,\boldsymbol{m}}^{\cont},C}^\rho,
            \hat{\nu}_C)
        \Biggr]
        \le
        \varepsilon_{n,\boldsymbol{m}}
        .
    \end{equation*}
    Applying Lemma~\ref{lem:proxy-to-exact-app} to $\hat{g}_{n,\boldsymbol{m}}^{\cont}$ and the triangle inequality $W_1(\mu_C^{\obs},\nu_{\hat{g}_{n,\boldsymbol{m}}^{\cont},C}^\rho)\le W_1(\mu_C^{\obs},\hat{\nu}_C)+W_1(\hat{\nu}_C,\nu_{\hat{g}_{n,\boldsymbol{m}}^{\cont},C}^\rho)$ in each cell,
    \begin{equation*}
        \E\bigl[\eW_1(P^\ast,\hat{P}_{n,\boldsymbol{m}}^{\cont})\bigr]
        \le
        \E\Biggl[\sum_Cq_CW_1(\mu_C^{\obs},\hat{\nu}_C)\Biggr]
        +
        \E\Biggl[\sum_Cq_CW_1(\hat{\nu}_C,\nu_{\hat{g}_{n,\boldsymbol{m}}^{\cont},C}^\rho)\Biggr]
        +
        L_\ast b_{\boldsymbol{m}}
        ,
    \end{equation*}
    and the two displays above bound the first two terms, which gives \eqref{eq:continuous-upper-bound-app}.
\end{proof}

\begin{proof}[Proof of Theorem~\ref{thm:continuous-upper-bound}]
    The finite-resolution inequality is Proposition~\ref{prop:continuous-finite-resolution-upper-app}.
    If $\kappa n<1$, the claim follows from $\eW_1(P^\ast,\hat{P}_{n,\boldsymbol{m}_n}^{\cont})\le\diam(\cY)$ after increasing $C$.
    Assume henceforth that $\kappa n\ge1$.
    For $m_{n,j}=\lfloor r\log_2(\kappa n)/\alpha_j\rfloor\vee0$, we have $2^{-\alpha_jm_{n,j}}\le2^{\alpha_j}(\kappa n)^{-r}\le2(\kappa n)^{-r}$, hence $b_{\boldsymbol{m}_n}\le2d_W(\kappa n)^{-r}$, and $2^{-|\boldsymbol{m}_n|_1}$ lies between $(\kappa n)^{-r\bar{d}}$ and $2^{d_W}(\kappa n)^{-r\bar{d}}$, hence $\kappa n2^{-|\boldsymbol{m}_n|_1}\asymp(\kappa n)^{1-r\bar{d}}=(\kappa n)^{1/(1+a\bar{d})}$ because $r=a/(1+a\bar{d})$.
    Hence
    \begin{equation*}
        (\kappa n2^{-|\boldsymbol{m}_n|_1})^{-a}
        \asymp
        (\kappa n)^{-a/(1+a\bar{d})}
        =
        (\kappa n)^{-r}
        .
    \end{equation*}
    The assumed transfer bound at $\boldsymbol{m}_n$ yields \eqref{eq:continuous-optimized-upper-main}.
    The raw-law statement, when needed, follows from \eqref{eq:raw-oscillation-triangle-app}.
\end{proof}

Write $\hat{g}=\hat{g}_{n,\boldsymbol{m}_n}^{\cont}$ for the following corollary.

\begin{corollary}[Raw conditional law]
\label{cor:raw-conditional-law-app}
    Under the hypotheses of Theorem~\ref{thm:continuous-upper-bound}, suppose that, uniformly over $\cM_{\cont}$ and almost surely, either $\nu_{\hat{g},w}$ is constant within each cell of $\Pi_{\boldsymbol{m}_n}$ or
    \begin{equation*}
        W_1(\nu_{\hat{g},w},\nu_{\hat{g},w'})
        \le
        L_G\sum_{j=1}^{d_W}|w_j-w_j'|
        ,
        \quad
        w,w'\in\cW
        ,
    \end{equation*}
    for a fixed $0\le L_G<\infty$ independent of $n$.
    Then \eqref{eq:continuous-optimized-upper-main} also holds with $\hat{P}_{n,\boldsymbol{m}_n}^{\cont}$ replaced by $P_{\hat{g}}$, with constants possibly depending on $L_G$.
\end{corollary}

\begin{proof}[Proof of Corollary~\ref{cor:raw-conditional-law-app}]
    For a cellwise constant law, the within-cell term in \eqref{eq:raw-oscillation-triangle-app} is zero.
    Otherwise, for $\kappa n\ge1$, convexity of $W_1$ in its second argument, the cell side lengths $2^{-m_{n,j}}$, and the bound $2^{-m_{n,j}}\le2^{-\alpha_jm_{n,j}}\le2(\kappa n)^{-r}$ from the proof of Theorem~\ref{thm:continuous-upper-bound}, which uses $\alpha_j\le1$, give
    \begin{equation*}
        \int
        W_1(\nu_{\hat{g},w},\nu_{\hat{g},C_{\boldsymbol{m}_n}(w)}^\rho)
        Q_\rho(\dd w)
        \le
        L_G\sum_{j=1}^{d_W}2^{-m_{n,j}}
        \le
        2L_Gd_W(\kappa n)^{-r}
        .
    \end{equation*}
    Combine this bound with \eqref{eq:raw-oscillation-triangle-app} and Theorem~\ref{thm:continuous-upper-bound}; the case $\kappa n<1$ follows from compactness as above.
\end{proof}

For $t>0$, Markov's inequality also gives
\begin{equation*}
    \E
    \Bigl[
        Q_\rho
        \bigl(
            \{w\in\cW:W_1(\mu_w^\ast,\nu_{\hat{g},w})>t\}
        \bigr)
    \Bigr]
    \le
    \frac{C(\log n)^c(\kappa n)^{-r}}{t}
    .
\end{equation*}
This bounds the expected target mass of states with conditional Wasserstein error above $t$.

\subsection{Continuous-Conditioning Minimax Lower Bound}
\label{app:continuous-minimax-lower-bound}

The lower bound uses an anisotropic Assouad construction.
We state the one-state lower-bound primitive in the form needed below and write $d=d_U$.
For probability measures $P,Q$, use the squared Hellinger distance $H^2(P,Q)=\int(\sqrt{\dd P/\dd\lambda}-\sqrt{\dd Q/\dd\lambda})^2\dd\lambda$, where $\lambda$ dominates both measures, and the total variation distance $\operatorname{TV}(P,Q)=\sup_A|P(A)-Q(A)|$, where the supremum is over measurable sets.
Let $d_H$ denote the Hamming distance, the number of coordinates on which two binary vectors of the same length differ.
For an indexed generator family, write $\eta_{\theta,t}^{N}=(g_{\theta,t}^{N})_\#U$ and $\eta_\theta^{N}=\eta_{\theta,1}^{N}$.
For an open set $S$, $C_c^\infty(S)$ denotes the smooth functions with compact support in $S$; these functions are extended by zero outside $S$.

\begin{lemma}[One-state Assouad primitive]
\label{lem:onestate-assouad-primitive}
    Suppose that $1\le d\le p$, $\beta\ge1$, $K_0>0$, and $\cY=B^p(0,2K_0)$.
    For any fixed $\zeta>0$ and every sufficiently large integer $N$, there exist an integer $K_N$, a number $\Delta_N>0$, a fixed generator $g_0$, and generators $\{g_{\theta,t}^{N}:\theta\in\{0,1\}^{K_N},\ t\in[0,1]\}$ in $\mathcal{H}_{K_0}^{\beta+1}$ with the following properties.
    \begin{enumerate}[label=(\roman*),leftmargin=*]
        \item
        $g_{\theta,0}^{N}=g_0$ and $g_{\theta,t}^{N}=g_0+t(g_{\theta,1}^{N}-g_0)$, with $\|g_{\theta,1}^{N}-g_0\|_\infty\le N^{-a}$.
        \item
        For every pair differing in exactly one coordinate,
        \begin{equation*}
            \sup_{t\in[0,1]}
            H^2(\eta_{\theta,t}^{N},\eta_{\theta',t}^{N})
            \le
            \frac{\zeta}{N}
            .
        \end{equation*}
        \item
        For all $\theta,\theta'$,
        \begin{equation*}
            W_1(\eta_\theta^{N},\eta_{\theta'}^{N})
            \ge
            c\Delta_N d_H(\theta,\theta')
            ,
            \quad
            K_N\Delta_N
            \asymp
            N^{-a}
            .
        \end{equation*}
    \end{enumerate}
    Constants may depend on $\zeta$ and the fixed model parameters.
\end{lemma}

\begin{proof}[Proof of Lemma~\ref{lem:onestate-assouad-primitive}]
    Choose $\sigma\in(0,1)$ sufficiently small and define
    \begin{equation*}
        g_0(u)
        =
        \sigma\bigl(
            \sin(2\pi u_1),\ldots,
            \sin(2\pi u_d),0,\ldots,0
        \bigr)
        .
    \end{equation*}
    This periodic map has image in $B^p(0,K_0/2)$ and belongs to $\mathcal{H}_{K_0/2}^{\beta+1}$ after decreasing $\sigma$ if necessary.
    Identifying the first $d$ coordinates with $\mathbb{R}^d$, the law of $g_0(U)$ has density
    \begin{equation*}
        f_0(z)=
        \begin{cases}
            \displaystyle\prod_{i=1}^d \frac{1}{\pi\sqrt{\sigma^2-z_i^2}}, &z\in(-\sigma,\sigma)^d
            ,\\
            0,&\text{otherwise}
            .
        \end{cases}
    \end{equation*}
    This density is smooth and bounded above and below by positive constants on $[-\sigma/2,\sigma/2]^d$.
    Fix a nonzero, nonnegative $b\in C_c^\infty((-1/8,1/8)^d)$ and $\varphi\in C_c^\infty((-1/3,1/3)^d)$ such that $\varphi(z)=z_1$ on $[-1/4,1/4]^d$.
    If $d>2$, let
    \begin{equation*}
        J_N
        =
        \bigl\lceil
            N^{1/(2\beta+d)}
        \bigr\rceil
        ,\quad
        \ell_N
        =
        \frac{\sigma}{2J_N}
        ,\quad
        K_N=J_N^d
        ,\quad
        A_N
        =
        c_0\ell_N^{\beta+1}
        .
    \end{equation*}
    If $d\le2$, let
    \begin{equation*}
        J_N=K_N=1
        ,\quad
        \ell_N
        =
        \sigma/2
        ,
        \quad
        A_N
        =
        c_0N^{-1/2}
        .
    \end{equation*}
    Here $c_0>0$ is a fixed constant chosen below.
    Divide $[-\sigma/4,\sigma/4]^d$ into $K_N$ cubes $Q_k$ of side length $\ell_N$, and denote their centers by $z_k$.
    For $z\in\mathbb{R}^d$, define
    \begin{equation*}
        S_{\theta,t}^{N}(z)
        =
        z
        +
        tA_N\sum_{k=1}^{K_N}\theta_k b\bigl(
            (z-z_k)/\ell_N
        \bigr)e_1
        .
    \end{equation*}
    For sufficiently small $c_0$, $\|DS_{\theta,t}^{N}-I\|_\infty\le1/4$ uniformly in $(N,\theta,t)$.
    Here $DS_{\theta,t}^{N}$ is the Jacobian matrix, $I$ is the $d\times d$ identity matrix, and the uniform matrix norm uses the Euclidean operator norm.
    Each map is a diffeomorphism, equals the identity near the boundary of every $Q_k$, and maps $Q_k$ onto itself.
    Indeed, all coordinates except the first are unchanged, and its first coordinate is strictly increasing on every line parallel to $e_1$.
    Define $g_{\theta,t}^{N}$ by applying $S_{\theta,t}^{N}$ to the first $d$ coordinates of $g_0$ and leaving the remaining coordinates zero.

    The scaled bumps have disjoint supports separated by distances proportional to $\ell_N$, and consequently
    \begin{equation*}
        \Biggl\|
            A_N\sum_k\theta_k
            b\bigl(
                (\cdot-z_k)/\ell_N
            \bigr)e_1
        \Biggr\|_{C^{\beta+1}(\mathbb{R}^d;\mathbb{R}^d)}
        \le
        C A_N\ell_N^{-(\beta+1)}
        .
    \end{equation*}
    This is the usual H\"older norm; at integer smoothness the bound applies to all derivatives through that order.
    Composition with the fixed smooth map $g_0$ preserves this bound up to a fixed constant.
    The right-hand side is at most $Cc_0$ when $d>2$ and at most $Cc_0\ell_N^{-(\beta+1)}$ when $d\le2$.
    Decreasing $c_0$ therefore ensures $g_{\theta,t}^{N}\in\mathcal{H}_{K_0}^{\beta+1}$ uniformly.
    It also ensures $\|g_{\theta,1}^{N}-g_0\|_\infty\le N^{-a}$.
    The image of every generator is contained in the image of $g_0$, since $S_{\theta,t}^{N}$ maps $[-\sigma,\sigma]^d$ onto itself.
    This proves (i) and the required support condition.

    To prove (ii), consider an edge changing bit $k$.
    The two pushforward densities coincide outside $Q_k$.
    On $Q_k$, one density is $f_0$ and the other is the pushforward by $S(z)=z+tA_Nb((z-z_k)/\ell_N)e_1$.
    Its density is
    \begin{equation*}
        f_t(x)
        =
        \frac{f_0(S^{-1}(x))}
        {J_t(x)}
        ,
        \quad
        J_t(x)
        \coloneq
        1+(tA_N/\ell_N)\partial_1b((S^{-1}(x)-z_k)/\ell_N)
        .
    \end{equation*}
    Since $\|S^{-1}(x)-x\|\le A_N\|b\|_\infty$ and $|J_t(x)-1|\le(A_N/\ell_N)\|\partial_1b\|_\infty\le1/4$, the decomposition $|f_t(x)-f_0(x)|\le J_t(x)^{-1}\{|f_0(S^{-1}(x))-f_0(x)|+f_0(x)|1-J_t(x)|\}$ together with uniform bounds on $f_0$ and $\nabla f_0$ on $Q_k$ and $\ell_N\le\sigma/2<1$ gives $f_t,f_0\ge c>0$ on $Q_k$ and $\|f_t-f_0\|_{L^\infty(Q_k)}\le C(A_N+A_N/\ell_N)\le CA_N/\ell_N$.
    Since $(\sqrt{f_t}-\sqrt{f_0})^2=(f_t-f_0)^2/(\sqrt{f_t}+\sqrt{f_0})^2\le(f_t-f_0)^2/(4c)$ on $Q_k$, integrating over the cube of volume $\ell_N^d$ gives
    \begin{equation*}
        H^2(\eta_{\theta,t}^{N},\eta_{\theta',t}^{N})
        \le
        C A_N^2\ell_N^{d-2}
        \le
        Cc_0^2N^{-1}
        ,
    \end{equation*}
    where the last inequality uses $A_N^2\ell_N^{d-2}=c_0^2\ell_N^{2\beta+d}\asymp c_0^2N^{-1}$ when $d>2$ and $A_N^2\ell_N^{d-2}=c_0^2(\sigma/2)^{d-2}N^{-1}$ when $d\le2$.
    For $d<p$, these densities are taken with respect to Lebesgue measure on the common $d$-dimensional coordinate subspace.
    Choosing $c_0$ small enough gives (ii) for the prescribed $\zeta$.

    For (iii), write $y_{1:d}=(y_1,\ldots,y_d)$ and put $\varphi_k(y)=\ell_N\varphi((y_{1:d}-z_k)/\ell_N)$ for $y\in\mathbb{R}^p$.
    The $k$th bump is supported in $z_k+(-\ell_N/8,\ell_N/8)^d$ and displaces points by at most $A_N\|b\|_\infty\le\ell_N/8$ for $c_0$ sufficiently small, so every displaced point remains in $z_k+[-\ell_N/4,\ell_N/4]^d$, where $\varphi_k(y)=y_1-z_{k,1}$.
    Since each $S_{\theta,1}^{N}$ maps $Q_k$ onto itself, depends only on $\theta_k$ on $Q_k$, and $\varphi_k$ vanishes outside $Q_k$, writing $\eta_\theta^{N}$ as the image of $f_0(z)\dd z$ under $S_{\theta,1}^{N}$ gives, for any $\theta,\theta'\in\{0,1\}^{K_N}$,
    \begin{equation*}
        \begin{split}
            \int\varphi_k\dd(\eta_\theta^{N}-\eta_{\theta'}^{N})
            &=
            \int_{Q_k}\bigl\{
                \varphi_k(S_{\theta,1}^{N}(z))-\varphi_k(S_{\theta',1}^{N}(z))
            \bigr\}f_0(z)\dd z
            \\
            &=
            A_N(\theta_k-\theta'_k)
            \int_{Q_k}b((z-z_k)/\ell_N)f_0(z)\dd z
            .
        \end{split}
    \end{equation*}
    By the change of variables $u=(z-z_k)/\ell_N$ and $f_0\ge c_f>0$ on $Q_k$, the last integral is bounded below by $c\ell_N^d$.
    Because the supports of the $\varphi_k$ are disjoint,
    \begin{equation*}
        F_{\theta,\theta'}(y)
        =
        \frac{1}{\max\{1,\|\nabla\varphi\|_\infty\}}
        \sum_k\operatorname{sgn}(\theta_k-\theta'_k)\varphi_k(y)
    \end{equation*}
    is $1$-Lipschitz on $\mathbb{R}^p$.
    Since $\operatorname{sgn}(\theta_k-\theta'_k)(\theta_k-\theta'_k)=|\theta_k-\theta'_k|$, Kantorovich--Rubinstein duality gives
    \begin{equation*}
        W_1(\eta_\theta^{N},\eta_{\theta'}^{N})
        \ge
        \int F_{\theta,\theta'}\dd(\eta_\theta^{N}-\eta_{\theta'}^{N})
        \ge
        cA_N\ell_N^d\sum_{k=1}^{K_N}|\theta_k-\theta'_k|
        =
        cA_N\ell_N^d d_H(\theta,\theta')
        .
    \end{equation*}
    Taking $\Delta_N=A_N\ell_N^d$ yields $K_N\Delta_N=(\sigma/2)^dA_N\asymp N^{-a}$.
\end{proof}

\begin{proof}[Proof of Theorem~\ref{thm:continuous-lower-bound}]
    Fix the known target design $Q_\rho$ satisfying Assumption~\ref{ass:design-regularity-main}, and put $D=d_W\ge1$ and $\bar{d}=\sum_{j=1}^D\alpha_j^{-1}$.
    Fix $\chi\in C_c^\infty((0,1))$ with $0\le\chi\le1$ and $\chi=1$ on $[1/4,3/4]$, and put $C_\chi=\max\{1,\|\chi'\|_\infty\}$ and $C_1=\max\{2,C_\chi\}$.
    Choose
    \begin{equation*}
        B
        \ge
        \max\{1,C_1/L_\ast\}
        ,\quad
        0
        <
        \zeta
        \le
        (16\bar{q} B^{\bar{d}})^{-1}
        .
    \end{equation*}
    These constants are fixed independently of $n$ and $\kappa$.
    Write
    \begin{equation*}
        N
        =
        \Bigl\lceil
            (\kappa n)^{1/(1+a\bar{d})}
        \Bigr\rceil
        ,\quad
        \delta
        =
        N^{-a}
        ,\quad
        h_j
        =
        (B\delta)^{1/\alpha_j}
        .
    \end{equation*}
    Then $\delta\asymp(\kappa n)^{-r}$ and $v\coloneq\prod_jh_j=B^{\bar{d}}N^{-a\bar{d}}$.
    All following constructions apply once $\kappa n$ is sufficiently large.

    Choose a fixed $\tau\in(0,\min\{1/2,(4\bar{q})^{-1}\}]$.
    Set $L_1=\tau$, $L_j=1$ for $j\ge2$, and $J_j=\lfloor L_j/(2h_j)\rfloor$.
    For each multi-index $k\in\prod_j\{0,\ldots,J_j-1\}$, take the rectangle
    \begin{equation*}
        R_k
        =
        \prod_{j=1}^D[2k_jh_j,(2k_j+1)h_j]
    \end{equation*}
    and its central subrectangle $B_k$, whose side lengths are $h_j/2$.
    Relabel the rectangles by $\ell=1,\ldots,M$, where $M=\prod_{j=1}^D J_j$, and set $A=\bigcup_{\ell=1}^M R_\ell$.
    Their closures are pairwise disjoint, and
    \begin{equation*}
        M
        \asymp
        v^{-1}
        ,
        \quad
        c
        \le
        \sum_{\ell=1}^M Q_\rho(B_\ell)
        ,
        \quad
        Q_\rho(A)
        \le
        1/4
        .
    \end{equation*}
    Indeed, $L_j/(4h_j)\le J_j\le L_j/(2h_j)$ for large $n$, so $M=\prod_jJ_j$ lies between $4^{-D}\tau/v$ and $2^{-D}\tau/v$;
    the density bounds then give $\sum_\ell Q_\rho(B_\ell)\ge\underline{q}M\prod_j(h_j/2)\ge\underline{q}8^{-D}\tau$ and $Q_\rho(A)\le\bar{q}M\prod_jh_j\le\bar{q}2^{-D}\tau\le1/4$ by the choice of $\tau$.
    If $a_{\ell,j}$ is the lower endpoint of $R_\ell$ in coordinate $j$, put
    \begin{equation*}
        \psi_\ell(w)
        =
        \prod_{j=1}^D\chi((w_j-a_{\ell,j})/h_j)
        .
    \end{equation*}
    Then $\psi_\ell=1$ on $B_\ell$, $\operatorname{supp}(\psi_\ell)\subset R_\ell$, and $\|\partial_j\psi_\ell\|_\infty\le C_\chi/h_j$.

    Define the observational design by
    \begin{equation*}
        \frac{\dd Q_{\obs}}{\dd Q_\rho}(w)
        =
        \kappa\1_A(w)
        +
        \frac{1-\kappa Q_\rho(A)}{1-Q_\rho(A)}\1_{A^c}(w)
        .
    \end{equation*}
    This defines a probability measure, since its total mass is $\kappa Q_\rho(A)+(1-\kappa Q_\rho(A))=1$, and its density relative to $Q_\rho$ is at least $\kappa$ because $\kappa\le1$ implies $(1-\kappa Q_\rho(A))/(1-Q_\rho(A))\ge\kappa$.
    Hence $w_\rho=\dd Q_\rho/\dd Q_{\obs}\le\kappa^{-1}$, which is Assumption~\ref{ass:overlap-general}.

    Apply Lemma~\ref{lem:onestate-assouad-primitive} at $N$ with the prescribed $\zeta$.
    For $\Theta=(\theta_1,\ldots,\theta_M)\in(\{0,1\}^{K_N})^M$, define
    \begin{equation*}
        g_\Theta(w,u)
        =
        g_0(u)
        +
        \sum_{\ell=1}^M
        \psi_\ell(w)(g_{\theta_\ell,1}^{N}(u)-g_0(u))
        ,
        \quad
        \mu_{\Theta,w}
        =
        (g_\Theta(w,\cdot))_\#U
        .
    \end{equation*}
    Let $P_\Theta^\ast(\dd w,\dd y)=Q_\rho(\dd w)\mu_{\Theta,w}(\dd y)$.
    At most one summand is nonzero for every $w$.
    Consequently each section is either $g_0$ or $g_{\theta_\ell,\psi_\ell(w)}^{N}$, and belongs to $\mathcal{H}_{K_0}^{\beta+1}$.
    The construction is jointly measurable.
    It is also a causal submodel:
    independently of the assigned $(X,T)\sim Q_{\obs}$, draw $U$ and set $Y_\Theta(t)=g_\Theta((X,t),U)$ for every $t$, with $Y_\Theta^{\obs}=Y_\Theta(T)$.
    This gives consistency and conditional ignorability and induces the specified observational kernel.

    For points differing only in coordinate $j$, the displacement bound $\|g_{\theta,1}^{N}-g_0\|_\infty\le\delta$, the derivative bound $\|\partial_j\psi_\ell\|_\infty\le C_\chi/h_j$, and the inequality $\min\{x,y\}\le x^{1-\alpha_j}y^{\alpha_j}$ for $x,y\ge0$ give
    \begin{equation*}
        \|g_\Theta(w,\cdot)-g_\Theta(w',\cdot)\|_\infty
        \le
        \min\{2\delta,C_\chi\delta h_j^{-1}|w_j-w'_j|\}
        \le
        C_1\delta h_j^{-\alpha_j}|w_j-w'_j|^{\alpha_j}
        =
        \frac{C_1}{B}|w_j-w'_j|^{\alpha_j}
        ,
    \end{equation*}
    where the last equality uses $h_j^{\alpha_j}=B\delta$.
    Coupling $\mu_{\Theta,w}$ and $\mu_{\Theta,w'}$ through the same $U$ bounds $W_1$ by the uniform distance of the generators, and moving one coordinate at a time along a path from $w$ to $w'$ gives
    \begin{equation*}
        W_1(\mu_{\Theta,w},\mu_{\Theta,w'})
        \le
        \frac{C_1}{B}\sum_{j=1}^D|w_j-w'_j|^{\alpha_j}
        .
    \end{equation*}
    The choice of $B$ therefore gives Assumption~\ref{ass:anisotropic-holder-main}.

    Consider global vertices differing in a single bit in rectangle $R_\ell$.
    Write $P_{\Theta,\obs}(\dd w,\dd y)=Q_{\obs}(\dd w)\mu_{\Theta,w}(\dd y)$.
    Let $\E_\Theta$ denote expectation under $P_{\Theta,\obs}^{\otimes n}$, including any independent estimator randomization.
    Since their design marginals agree, the conditional Hellinger identity applies; the two conditional laws coincide outside $R_\ell$, on $R_\ell$ they equal $\eta_{\theta_\ell,\psi_\ell(w)}^{N}$ and $\eta_{\theta'_\ell,\psi_\ell(w)}^{N}$, and $Q_{\obs}(R_\ell)=\kappa Q_\rho(R_\ell)\le\kappa\bar{q}v$, so Lemma~\ref{lem:onestate-assouad-primitive}(ii) gives
    \begin{equation*}
        H^2(P_{\Theta,\obs},P_{\Theta',\obs})
        =
        \int H^2(\mu_{\Theta,w},\mu_{\Theta',w})Q_{\obs}(\dd w)
        \le
        \frac{\zeta}{N}\kappa Q_\rho(R_\ell)
        \le
        \frac{\zeta}{N}\kappa\bar{q} v
        .
    \end{equation*}
    By the tensorization bound $H^2(P^{\otimes n},Q^{\otimes n})\le nH^2(P,Q)$, the identity $v=B^{\bar{d}}N^{-a\bar{d}}$, and $N\ge(\kappa n)^{1/(1+a\bar{d})}$,
    \begin{equation*}
        H^2(P_{\Theta,\obs}^{\otimes n},P_{\Theta',\obs}^{\otimes n})
        \le
        \zeta\bar{q} B^{\bar{d}}(\kappa n)N^{-(1+a\bar{d})}
        \le
        \zeta\bar{q} B^{\bar{d}}
        \le
        \frac{1}{16}
        .
    \end{equation*}
    Hence every such pair has $n$-sample total variation distance at most $\sqrt{1/16}=1/4$.

    To reduce arbitrary estimators to binary decisions, write $\hat{P}_n(\dd w,\dd y)=Q_\rho(\dd w)\hat{\mu}_{n,w}(\dd y)$ and, for each $\ell$, let $\hat{\theta}_\ell$ minimize
    \begin{equation*}
        \int_{B_\ell}W_1(\hat{\mu}_{n,w},\eta_\theta^{N})Q_\rho(\dd w)
        ,\quad
        \theta\in\{0,1\}^{K_N}
        ,
    \end{equation*}
    with lexicographic tie breaking.
    The minimization is over a finite set and its integrands are measurable, so these decisions are measurable.
    On $B_\ell$, the true conditional law is $\eta_{\theta_\ell}^{N}$, because $\psi_\ell=1$ there.
    By the triangle inequality and the minimizing property of $\hat{\theta}_\ell$,
    \begin{equation*}
        \begin{split}
            Q_\rho(B_\ell)W_1(\eta_{\theta_\ell}^{N},\eta_{\hat{\theta}_\ell}^{N})
            &\le
            \int_{B_\ell}W_1(\eta_{\theta_\ell}^{N},\hat{\mu}_{n,w})Q_\rho(\dd w)
            +
            \int_{B_\ell}W_1(\hat{\mu}_{n,w},\eta_{\hat{\theta}_\ell}^{N})Q_\rho(\dd w)
            \\
            &\le
            2\int_{B_\ell}W_1(\mu_{\Theta,w},\hat{\mu}_{n,w})Q_\rho(\dd w)
            ,
        \end{split}
    \end{equation*}
    so Lemma~\ref{lem:onestate-assouad-primitive}(iii) gives
    \begin{equation*}
        \int_{B_\ell}W_1(\mu_{\Theta,w},\hat{\mu}_{n,w})Q_\rho(\dd w)
        \ge
        \frac{Q_\rho(B_\ell)}{2}W_1(\eta_{\theta_\ell}^{N},\eta_{\hat{\theta}_\ell}^{N})
        \ge
        cQ_\rho(B_\ell)\Delta_Nd_H(\theta_\ell,\hat{\theta}_\ell)
        .
    \end{equation*}
    Under the uniform prior on the $MK_N$ bits, the testing bound implies an error probability at least $3/8$ for every bit:
    the average error in each neighboring pair is at least one half of one minus the total variation distance between its two $n$-sample observational laws.
    Restricting the integral defining $\eW_1(P_\Theta^\ast,\hat{P}_n)$ to $\bigcup_\ell B_\ell$, applying the previous display in each $B_\ell$, and averaging over the uniform prior, which bounds the maximum over $\Theta$ from below, we obtain
    \begin{equation*}
        \begin{split}
            \sup_\Theta
            \E_\Theta
            [
                \eW_1(P_\Theta^\ast,\hat{P}_n)
            ]
            &\ge
            c\sum_{\ell=1}^M Q_\rho(B_\ell)\Delta_N\E[d_H(\theta_\ell,\hat{\theta}_\ell)]
            \\
            &\ge
            \frac{3c}{8}\sum_{\ell=1}^M Q_\rho(B_\ell)K_N\Delta_N
            \\
            &\ge
            cN^{-a}
            \\
            &\ge
            c(\kappa n)^{-r}
            ,
        \end{split}
    \end{equation*}
    where the last two inequalities use $\sum_\ell Q_\rho(B_\ell)\ge c$, $K_N\Delta_N\asymp N^{-a}$, and $N^{-a}=\delta\asymp(\kappa n)^{-r}$.
    Independent randomization of an estimator leaves the testing bound unchanged.
    Taking the infimum over all estimators proves the theorem.
\end{proof}

\subsection{Transfer from Routed Sieves to Single-Network Implementations}
\label{app:single-network-transfer-proof}

The proof uses an ideal routed estimator with one expert per cell.
Write the implemented class as $\cG_n^{\cont}=\{g_\theta:\theta\in\Theta_n\}$, where $\Theta_n$ is a compact metric space, $(\theta,w,u)\mapsto g_\theta(w,u)$ is jointly measurable, and $\theta\mapsto g_\theta(w,u)$ is continuous for every $(w,u)$.
Each empirical criterion is a maximum of finitely many continuous functions of $\theta$.
The measurable maximum theorem therefore gives a measurable minimizer;
fix one such selection throughout.
All risk expectations include both the observational sample and the independent simulation draws used in fitting.
For a general implementation, the required condition is integrated finite-resolution transfer, not uniform approximation of discontinuous indicators.
Let $\hat{\nu}_C$ denote the routed cell law from the proof of Proposition~\ref{prop:continuous-finite-resolution-upper-app}, and let $\nu_{\hat{g}_{n,\boldsymbol{m}}^{\cont},C}^\rho$ be the generated cell law of the implemented estimator in \eqref{eq:continuous-estimator-main}.
Let $\varepsilon_{n,\boldsymbol{m}}\ge0$ be a deterministic bound on the integrated implementation error.
We assume, uniformly over $(P_{\obs},P^\ast)\in\cM_{\cont}$,
\begin{equation}
\label{eq:implementation-transfer-condition}
    \E
    \Biggl[
        \sum_{C\in\Pi_{\boldsymbol{m}}}q_C
        W_1(\nu_{\hat{g}_{n,\boldsymbol{m}}^{\cont},C}^\rho,\hat{\nu}_C)
    \Biggr]
    \le
    \varepsilon_{n,\boldsymbol{m}}
    .
\end{equation}
Hard routing of the same fitted cell experts $\hat{\nu}_C$, with the same samples and minimizer selection, gives $\varepsilon_{n,\boldsymbol{m}}=0$.
The within-cell term in \eqref{eq:raw-oscillation-triangle-app} is also zero, giving the raw-law guarantee in Corollary~\ref{cor:raw-conditional-law-app}.
For approximate empirical minimizers, the optimization tolerances in Lemma~\ref{lem:structural-onestate-stability-app} enter the bound through their expected target-weighted average.
The following proposition bounds the transfer error when soft stochastic gates combine the same fitted cell experts.

Write $\Pi_{\boldsymbol{m}}=\{C_1,\ldots,C_K\}$, with $K=|\Pi_{\boldsymbol{m}}|$, and set $r_k(w)=\1\{w\in C_k\}$.
For expert measures $\nu_k$ and weight functions $\gamma_k$, use the notation
\begin{equation*}
    \eta
    =
    \int\sum_{k=1}^K|\gamma_k(w)-r_k(w)|Q_\rho(\dd w)
    ,
    \quad
    \nu_w
    =
    \sum_{k=1}^K\gamma_k(w)\nu_k
    .
\end{equation*}
Let $\nu_{C_k}^\rho$ be the cell mixture of $\nu_w$ from \eqref{eq:cell-average-distributions-app}, with arbitrary probability laws on cells of zero target mass.

\begin{proposition}[Integrated soft-gate transfer]
\label{prop:integrated-gate-transfer}
    Suppose that $\nu_k\in\cP(\cY)$ and that the functions $\gamma_k:\cW\to[0,1]$ are measurable and satisfy $\sum_{k=1}^K\gamma_k(w)=1$ for every $w\in\cW$.
    Then
    \begin{equation}
    \label{eq:soft-gate-cell-transfer}
        \sum_{k=1}^Kq_{C_k}W_1(\nu_{C_k}^\rho,\nu_k)
        \le
        \diam(\cY)\eta
        .
    \end{equation}
    Moreover,
    \begin{equation}
    \label{eq:soft-gate-raw-oscillation}
        \int W_1(\nu_w,\nu_{C_{\boldsymbol{m}}(w)}^\rho)Q_\rho(\dd w)
        \le
        2\diam(\cY)\eta
        .
    \end{equation}
\end{proposition}

\begin{proof}[Proof of Proposition~\ref{prop:integrated-gate-transfer}]
    Fix $w\in C_k$.
    Since $r_k(w)=1$ and $r_\ell(w)=0$ for $\ell\ne k$, we have $\nu_k=\sum_{\ell=1}^K r_\ell(w)\nu_\ell$, and every $f\in\mathcal{L}_1$ satisfies $\|f\|_\infty\le\diam(\cY)$, so Kantorovich--Rubinstein duality gives
    \begin{equation*}
        W_1\Biggl(
            \sum_{\ell=1}^K\gamma_\ell(w)\nu_\ell,\nu_k
        \Biggr)
        =
        \sup_{f\in\mathcal{L}_1}
        \sum_{\ell=1}^K(\gamma_\ell(w)-r_\ell(w))\int f\dd\nu_\ell
        \le
        \diam(\cY)\sum_{\ell=1}^K|\gamma_\ell(w)-r_\ell(w)|
        .
    \end{equation*}
    For $q_{C_k}>0$, convexity of $W_1$ in its first argument gives
    \begin{equation*}
        W_1(\nu_{C_k}^\rho,\nu_k)
        \le
        q_{C_k}^{-1}\int_{C_k}
        W_1(\nu_w,\nu_k)Q_\rho(\dd w)
        .
    \end{equation*}
    Multiplying by $q_{C_k}$ and summing over the positive-mass cells proves \eqref{eq:soft-gate-cell-transfer}; zero-mass cells contribute nothing.

    For $w\in C_k$,
    \begin{equation*}
        W_1(\nu_w,\nu_{C_k}^\rho)
        \le
        W_1(\nu_w,\nu_k)
        +
        W_1(\nu_k,\nu_{C_k}^\rho)
        .
    \end{equation*}
    Integrating and using the first part twice proves \eqref{eq:soft-gate-raw-oscillation}.
\end{proof}

For $\delta>0$, define the boundary layer
\begin{equation*}
    \mathcal{B}_\delta
    =
    \Biggl\{
        w\in\cW:
        \operatorname{dist}\Biggl(
            w,\bigcup_{C\in\Pi_{\boldsymbol{m}}}\partial C
        \Biggr)
        \le
        \delta
    \Biggr\}
    .
\end{equation*}
Here $\operatorname{dist}$ is Euclidean point-to-set distance and $\partial C$ is the boundary of $C$.

\begin{lemma}[Dyadic boundary-layer control]
\label{lem:dyadic-boundary-layer-app}
    Suppose Assumption~\ref{ass:design-regularity-main} holds.
    For $0<\delta\le1$,
    \begin{equation}
    \label{eq:dyadic-boundary-layer-app}
        Q_\rho(\mathcal{B}_\delta)
        \le
        C\delta\sum_{j=1}^{d_W}2^{m_j}
        .
    \end{equation}
    If, for constants $A,c_0>0$ and gate temperature $\tau>0$, gates satisfying the conditions of Proposition~\ref{prop:integrated-gate-transfer} also satisfy
    \begin{equation*}
        \sum_k|\gamma_k(w)-r_k(w)|
        \le
        A\exp(-c_0\delta/\tau)
        \quad
        \text{for all }w\in\cW\setminus\mathcal{B}_\delta
        ,
    \end{equation*}
    then
    \begin{equation}
    \label{eq:soft-gate-boundary-transfer-app}
        \eta
        \le
        2Q_\rho(\mathcal{B}_\delta)
        +
        A\exp(-c_0\delta/\tau)
        .
    \end{equation}
\end{lemma}

\begin{proof}[Proof of Lemma~\ref{lem:dyadic-boundary-layer-app}]
    In coordinate $j$, the dyadic partition has $2^{m_j}+1$ boundary hyperplanes, including the two outer faces.
    The $\delta$-neighborhood of these hyperplanes has Lebesgue measure at most $C\delta2^{m_j}$.
    A union bound over coordinates and the density upper bound $q_\rho\le\bar{q}$ prove \eqref{eq:dyadic-boundary-layer-app}.
    On $\mathcal{B}_\delta$, the quantity $\sum_k|\gamma_k-r_k|$ is at most $2$ because the soft and hard gate vectors are probability vectors.
    Outside $\mathcal{B}_\delta$, use the assumed exponential bound and integrate.
\end{proof}

Choosing $\delta$ and the gate temperature $\tau$ so that the right-hand side of \eqref{eq:soft-gate-boundary-transfer-app} is $o((\kappa n)^{-r})$ makes the transfer term negligible.
This is an integrated condition; no continuous network is required to approximate discontinuous indicators uniformly.
Under the gate bound in Lemma~\ref{lem:dyadic-boundary-layer-app}, with $A,c_0$ independent of $n$, one choice at resolution $\boldsymbol{m}_n$ for $\kappa n>1$ is
\begin{equation*}
    \delta_n
    =
    \frac{(\kappa n)^{-(r+1)}}{\sum_{j=1}^{d_W}2^{m_{n,j}}}
    ,
    \quad
    \tau_n
    =
    \frac{c_0\delta_n}{(r+1)\log(\kappa n)}
    .
\end{equation*}
Then $\eta=O((\kappa n)^{-(r+1)})$, so Proposition~\ref{prop:integrated-gate-transfer} makes both the transfer error and within-cell oscillation negligible relative to $(\kappa n)^{-r}$.

\subsection{Outcome Truncation}
\label{app:truncation-bias}

Use the statewise integral representation of $\eW_1$ for outcomes in $\R^p$.
For $R>0$, let $T_R$ be projection onto $B^p(0,R)$ and define $P_R^\ast(\dd w,\dd y)=Q_\rho(\dd w)((T_R)_\#\mu_w^\ast)(\dd y)$.
Write $(x)_+=\max\{x,0\}$.

\begin{proposition}[Truncation bias]
\label{prop:truncation-bias-app}
    Suppose $\int\!\int\|y\|\mu_w^\ast(\dd y)Q_\rho(\dd w)<\infty$.
    For any estimated law $\hat{P}$ sharing $Q_\rho$ and supported on $\cW\times B^p(0,R)$,
    \begin{equation}
    \label{eq:truncation-bias-app}
        \eW_1(P^\ast,\hat{P})
        \le
        \eW_1(P_R^\ast,\hat{P})
        +
        \int\!\int
        (\|y\|-R)_+
        \mu_w^\ast(\dd y)Q_\rho(\dd w)
        .
    \end{equation}
\end{proposition}
\begin{proof}[Proof of Proposition~\ref{prop:truncation-bias-app}]
    Couple $y$ with $T_R(y)$ at cost $(\|y\|-R)_+$ and apply the statewise triangle inequality before integrating over $Q_\rho$.
\end{proof}

Conditional on the fixed covariates and coefficients, the Gaussian and Student-$t_5$ components in the semi-synthetic benchmarks have truncation biases $O(e^{-cR^2})$ and $O(R^{-4})$, respectively; the latter follows by integrating the $O(R^{-5})$ tail probability.
For Jobs, $Y=\operatorname{asinh}(re78/1000)$ compresses earnings tails but remains unbounded.
A uniform bound $\Prob(|re78|>z\mid W=w)\lesssim z^{-\gamma}$ for some $\gamma>0$ gives transformed-scale truncation bias $O(e^{-\gamma R})$.
The experiments enforce a bounded output range through a bounded generator output layer (Appendix~\ref{app:baselines}).
Equation~\eqref{eq:truncation-bias-app} separates truncation bias from estimation error; extending the fixed-range minimax results to growing $R$ additionally requires control of generator smoothness and range-dependent constants.

\section{Proofs of Section~\ref{sec:ipm-extension}}
\label{app:ipm-extension-proof}

Let $\mathcal{V}$ be a symmetric class of measurable functions containing zero and uniformly bounded by $B_{\mathcal{V}}<\infty$.
Set $d_{\mathcal{V}}(\mu,\nu)=\sup_{f\in\mathcal{V}}\int f\dd(\mu-\nu)$.
Assume that a countable subclass $\mathcal{V}_0\subset\mathcal{V}$ satisfies
\begin{equation*}
    d_{\mathcal{V}}(\mu,\nu)
    =
    \sup_{f\in\mathcal{V}_0}\int f\dd(\mu-\nu)    
\end{equation*}
for every pair of probability measures on $\cY$.
All estimators below are measurable probability kernels.
For $P(\dd w,\dd y)=Q_\rho(\dd w)\mu_w(\dd y)$ and $R(\dd w,\dd y)=Q_\rho(\dd w)\nu_w(\dd y)$, let $\mu_C^\rho$ and $\nu_C^\rho$ be the cell mixtures defined as in \eqref{eq:cell-average-distributions-app}.

\begin{lemma}[Finite-resolution IPM disintegration]
\label{lem:ipm-disintegration-app}
    For $P,R$ sharing $Q_\rho$,
    \begin{equation}
    \label{eq:ipm-disintegration-app}
        d_{\mathcal{V}}^{(\boldsymbol{m})}(P,R)
        =
        \sum_{C\in\Pi_{\boldsymbol{m}}}q_Cd_{\mathcal{V}}(\mu_C^\rho,\nu_C^\rho)
        .
    \end{equation}
\end{lemma}

\begin{proof}[Proof of Lemma~\ref{lem:ipm-disintegration-app}]
    Expanding expectations gives
    \begin{equation*}
        \sum_Cq_C\Biggl\{
            \int f_C\dd\mu_C^\rho-\int f_C\dd\nu_C^\rho
        \Biggr\}
        .
    \end{equation*}
    Since the functions $f_C$ are chosen independently across cells, the supremum separates and yields \eqref{eq:ipm-disintegration-app}.
\end{proof}

Fix $a_{\mathcal{V}}>0$ and assume the following one-state perturbation oracle for the local IPM.
There exist constants $C,c>0$, local generator classes, and estimators such that, whenever $\eta,\eta_0\in\cP(\cY)$ satisfy $d_{\mathcal{V}}(\eta,\eta_0)\le\delta$ for $\delta\ge0$ and $\eta_0$ belongs to the pointwise smooth pushforward model, the estimator $\hat{\eta}_k^{\mathcal{V}}$ based on $k$ i.i.d. samples from $\eta$ satisfies
\begin{equation}
\label{eq:local-ipm-perturbation-oracle-app}
    \E_\eta
    [
        d_{\mathcal{V}}(\eta,\hat{\eta}_k^{\mathcal{V}})
    ]
    \le
    C\{
        \delta
        +
        (\log(k\vee2))^c(k\vee1)^{-a_{\mathcal{V}}}
    \}
    .
\end{equation}
Assume also that the conditional law is anisotropic H\"older under $d_{\mathcal{V}}$: for some $0<L_{\mathcal{V}}<\infty$ and exponents $\alpha_j\in(0,1]$,
\begin{equation}
\label{eq:ipm-holder-app}
    d_{\mathcal{V}}(\mu_w^\ast,\mu_{w'}^\ast)
    \le
    L_{\mathcal{V}}\sum_{j=1}^{d_W}|w_j-w_j'|^{\alpha_j}
    .
\end{equation}
This inequality holds for all $w,w'\in\cW$.
Use $b_{\boldsymbol{m}}$ and $\bar{d}$ with these exponents, and assume that the constants in the local perturbation oracle are uniform over the pointwise model.
Conditional on $N_C^{\obs}=k\ge1$, let $\hat{\nu}_C^{\mathcal{V}}$ be the local estimator in \eqref{eq:local-ipm-perturbation-oracle-app}, with any required simulation independent of the data.
For $k=0$, set $\hat{\nu}_C^{\mathcal{V}}=\delta_{y_0}$.
Define the routed estimator by
\begin{equation*}
    \hat{P}_{n,\boldsymbol{m}}^{\mathcal{V}}(\dd w,\dd y)
    =
    Q_\rho(\dd w)\hat{\nu}_{C_{\boldsymbol{m}}(w)}^{\mathcal{V}}(\dd y)
    .
\end{equation*}

For the optimized resolution in this section, set $r_{\mathcal{V}}=(a_{\mathcal{V}}^{-1}+\bar{d})^{-1}$ and $\boldsymbol{m}_n=(m_{n,1},\ldots,m_{n,d_W})$, where
\begin{equation*}
    m_{n,j}
    =
    \Biggl\lfloor
        \frac{r_{\mathcal{V}}\log_2(\kappa n)}{\alpha_j}
    \Biggr\rfloor
    \vee
    0
    .
\end{equation*}

\begin{theorem}[Generic anisotropic IPM rate]
\label{thm:generic-ipm-rate-app}
    Under Assumptions~\ref{ass:consistency}--\ref{ass:overlap-general}, \ref{ass:design-regularity-main}, and \ref{ass:pointwise-pushforward-main}, and under \eqref{eq:local-ipm-perturbation-oracle-app}--\eqref{eq:ipm-holder-app}, the routed IPM estimator defined above satisfies
    \begin{equation}
    \label{eq:generic-ipm-rate-app}
        \E
        [
            d_{\mathcal{V}}^{(\boldsymbol{m})}(P^\ast, \hat{P}_{n,\boldsymbol{m}}^{\mathcal{V}})
        ]
        \le
        C\Biggl(
            b_{\boldsymbol{m}}
            +
            (\log n)^c(\kappa n2^{-|\boldsymbol{m}|_1})^{-a_{\mathcal{V}}}
        \Biggr)
        .
    \end{equation}
    At the optimized resolution,
    \begin{equation}
    \label{eq:generic-ipm-optimized-app}
        \E
        [
            d_{\mathcal{V}}^{(\boldsymbol{m}_n)}(P^\ast, \hat{P}_{n,\boldsymbol{m}_n}^{\mathcal{V}})
        ]
        \le
        C(\log n)^c(\kappa n)^{-r_{\mathcal{V}}}
        .
    \end{equation}
\end{theorem}

\begin{proof}[Proof of Theorem~\ref{thm:generic-ipm-rate-app}]
    The proof is the IPM analogue of Appendix~\ref{app:continuous-upper-proof}.
    For a cell $C$, conditional on $N_C^{\obs}$, observed outcomes in the cell are i.i.d. from $\mu_C^{\obs}$.
    Choose a representative $w_C\in C$.
    By \eqref{eq:ipm-holder-app},
    \begin{equation*}
        d_{\mathcal{V}}(\mu_C^{\obs},\mu_{w_C}^\ast)
        \le
        L_{\mathcal{V}}b_{\boldsymbol{m}}
        .
    \end{equation*}
    Conditional on $N_C^{\obs}=k\ge1$, the local perturbation oracle gives
    \begin{equation*}
        \E[d_{\mathcal{V}}(\mu_C^{\obs},\hat{\nu}_C^{\mathcal{V}})\mid N_C^{\obs}=k]
        \le
        C\{L_{\mathcal{V}}b_{\boldsymbol{m}}
        +
        (\log(k\vee2))^c(k\vee1)^{-a_{\mathcal{V}}}\}
        .
    \end{equation*}
    For $k=0$, uniform boundedness of $\mathcal{V}$ gives the same bound after increasing $C$.
    Apply Lemma~\ref{lem:binomial-negative-moment-discrete-app} with $a=a_{\mathcal{V}}$ and use $n\pi_C\ge\underline{q}\,\kappa n2^{-|\boldsymbol{m}|_1}$ to obtain
    \begin{equation*}
        \E[d_{\mathcal{V}}(\mu_C^{\obs},\hat{\nu}_C^{\mathcal{V}})]
        \le
        Cb_{\boldsymbol{m}}
        +
        C(\log n)^c(\kappa n2^{-|\boldsymbol{m}|_1})^{-a_{\mathcal{V}}}
        .
    \end{equation*}
    Multiplying by $q_C$ and summing over cells gives the same bound for the proxy risk because $\sum_Cq_C=1$.
    The target-versus-observational cell-mixture bias is controlled under \eqref{eq:ipm-holder-app} exactly as in Lemma~\ref{lem:observed-target-cell-bias-app}.
    Lemma~\ref{lem:ipm-disintegration-app} then yields \eqref{eq:generic-ipm-rate-app}.
    If $\kappa n<1$, \eqref{eq:generic-ipm-optimized-app} follows from the uniform boundedness of $\mathcal{V}$.
    For $\kappa n\ge1$, optimizing the anisotropic resolution gives \eqref{eq:generic-ipm-optimized-app}.
\end{proof}

\section{Detailed Experiments}
\label{app:detailed-experiments}

This appendix provides the experimental protocols for the three benchmarks presented in Section~\ref{sec:experiments}.
We begin by summarizing the datasets and data sources, followed by descriptions of the dataset-specific construction procedures, baseline methods, the primary evaluation metrics reported in Section~\ref{sec:experiments}, as well as additional metrics and supplementary results.
The benchmark comparisons are run over $100$ independent repetitions.
Within each repetition, all methods use the same train/validation/test split, simulation seed, target design, and evaluation protocol;
reported standard errors are standard errors over these $100$ repetitions.
The paired implementation diagnostics and synthetic sensitivity experiments use the seed counts specified in their respective subsections.
Throughout all experiments, \method\ is trained using the stratified, cell-normalized objective introduced in Section~\ref{subsec:conditional-estimator}.
The main benchmark experiments use one shared conditional generator.

\subsection{Dataset Information}
\label{app:dataset-information}

Table~\ref{tab:dataset-summary} summarizes the three datasets.
IHDP \citep{hill2011bayesian,shalit2017estimating} and TCGA \citep{tcga2013pancancer} are semi-synthetic benchmarks:
the covariates and treatment assignments are taken from standard causal datasets, while the potential-outcome laws are generated from known stochastic mechanisms.
This allows direct evaluation of the full conditional interventional distribution.
Jobs \citep{lalonde1986evaluating,dehejia1999causal} is a real-data benchmark with randomized-trial validation:
individual counterfactual distributions are not observed, so we evaluate model-implied arm-level interventional cumulative distribution functions against held-out randomized experimental arms.
Table~\ref{tab:data-sources-licenses} reports the data sources and license or access terms.

\begin{table}[tb]
    \centering
    \small
    \setlength{\tabcolsep}{4pt}
    \renewcommand{\arraystretch}{1.08}
    \caption{Datasets used in the experiments.}
    \label{tab:dataset-summary}
    \begin{tabular}{@{}llll@{}}
        \toprule
        Dataset
        & Size
        & Treatment
        & Evaluation target \\
        \midrule
        IHDP
        & $n=747$, $25$ covariates
        & binary
        & known $\{\mu_{x,t}^\ast\}_{t=0,1}$ \\
        TCGA
        & $n=9{,}659$, $4{,}000$ features
        & $3$ treatments + dosage
        & known $\{\mu_{x,a,d}^\ast\}$ on a dosage grid \\
        Jobs
        & NSW $297/425$, PSID $2{,}490$
        & binary
        & randomized arm-level CDFs \\
        \bottomrule
    \end{tabular}
\end{table}

\begin{table}[tb]
    \centering
    \footnotesize
    \setlength{\tabcolsep}{3pt}
    \renewcommand{\arraystretch}{1.12}
    \caption{Data sources and licenses.}
    \label{tab:data-sources-licenses}
    \begin{tabular}{@{}p{0.12\linewidth}p{0.53\linewidth}p{0.29\linewidth}@{}}
        \toprule
        Dataset & Source & License / terms \\
        \midrule
        IHDP
        &
        Supplementary materials of \citep{hill2011bayesian}
        &
        Not specified in the supplied supplementary files.
        \\
        \midrule
        \multirow{2}{*}{TCGA}
        &
        \url{https://paperdatasets.s3.amazonaws.com/tcga.db}
        &
        Not explicitly specified for the processed benchmark database.
        \\
        &
        TCGA Research Network / Genomic Data Commons:
        \url{https://portal.gdc.cancer.gov}
        &
        Governed by NIH/GDC open-data access policies.
        \\
        \midrule
        Jobs
        &
        \url{https://users.nber.org/~rdehejia/data/nsw_treated.txt}\newline
        \url{https://users.nber.org/~rdehejia/data/nsw_control.txt}\newline
        \url{https://users.nber.org/~rdehejia/data/psid_controls.txt}
        &
        CC BY-NC 2.0.
        \\
        \bottomrule
    \end{tabular}

    \vspace{2pt}
    \begin{minipage}{0.96\linewidth}
    \footnotesize
    \emph{Notes.}
    For IHDP, we found no explicit license statement in the supplied supplementary files of \citep{hill2011bayesian}.
    For TCGA, the first row reports the processed benchmark database used in continuous-dose causal learning; the second row records the underlying TCGA/GDC data-access regime \citep{tcga2013pancancer}.
    The MIT license of the Perfect Match code repository covers the code and does not extend to \texttt{tcga.db}.
    \end{minipage}
\end{table}

For semi-synthetic benchmarks, the target designs are known by construction.
For Jobs, the target design is the empirical covariate distribution of the held-out randomized NSW sample with equal treatment-arm weights.
All continuous covariates are standardized using training-set statistics, and all train/validation/test splits are shared across methods within each repetition.

\subsubsection{IHDP}
\label{app:ihdp}

\paragraph{Source and preprocessing.}
We use the IHDP covariates and treatment assignments from the benchmark popularized in counterfactual representation learning \citep{hill2011bayesian,shalit2017estimating}.
The raw \citep{hill2011bayesian} file contains $985$ randomized IHDP units.
Following the standard \citep{hill2011bayesian,shalit2017estimating} preprocessing, we remove non-white treated children and exclude race indicators, yielding $747$ units, $25$ covariates, and a binary treatment, with $139$ treated units and $608$ control units.
Continuous covariates are standardized using training-set means and standard deviations, while binary covariates are kept as $0/1$ variables.
We use stratified splits within treatment arms with training fraction $0.63$, validation fraction $0.27$, and the remaining fraction for testing.
The seed of the data-generating process (DGP) and the split seed are varied across the $100$ repetitions, and all methods share the resulting split within each repetition.

\paragraph{Stochastic potential-outcome law.}
The original IHDP benchmark is primarily a conditional-mean benchmark.
To evaluate distributional causal learning, we retain the observed covariates and treatment assignments but generate stochastic potential outcomes from a known law.
Let $\tilde{x}\in\R^d$ denote the standardized covariate vector, with $d=25$.
For scalar distributions, $\mathcal{N}(m,s^2)$ denotes the normal law with mean $m$ and variance $s^2$, and $\operatorname{logit}^{-1}(z)=(1+\exp(-z))^{-1}$.
For fixed coefficient vectors in $\R^d$, sampled once per repetition, define
\begin{equation*}
    m_0(x)
    =
    \exp(0.2 b_0^\top \tilde{x})-1
    ,\quad
    \tau(x)
    =
    1+0.5\tanh(b_\tau^\top \tilde{x})
    ,\quad
    m_1(x)
    =
    m_0(x)+\tau(x)
    .
\end{equation*}
For each treatment $t\in\{0,1\}$, define
\begin{equation*}
    \pi_t(x)
    =
    \operatorname{logit}^{-1}(b_{\pi,t}^\top \tilde{x})
    ,\quad
    \Delta_t(x)
    =
    0.5+0.5\tanh(b_{\Delta,t}^\top \tilde{x})
    ,
\end{equation*}
\begin{equation*}
    \sigma_{t,1}(x)
    =
    0.2+0.3\operatorname{logit}^{-1}(b_{\sigma,t}^\top \tilde{x})
    ,\quad
    \sigma_{t,2}(x)
    =
    0.3+0.5\operatorname{logit}^{-1}((b'_{\sigma,t})^\top\tilde{x})
    .
\end{equation*}
The potential-outcome law is
\begin{equation}
\label{eq:ihdp-dist-law}
    \begin{split}
        Y(t)\mid X=x
        &\sim
        \pi_t(x)
        \mathcal{N}\bigl(
            m_t(x)+(1-\pi_t(x))\Delta_t(x),\sigma_{t,1}^2(x)
        \bigr)
        \\
        &\quad
        +
        (1-\pi_t(x))
        t_\nu\bigl(
            m_t(x)-\pi_t(x)\Delta_t(x),\sigma_{t,2}(x)
        \bigr)
        ,
    \end{split}
\end{equation}
where $t_\nu(\ell,s)$ denotes a Student-$t$ distribution with location $\ell$, scale $s$, and degrees of freedom $\nu$.
We set $\nu=5$.
The coefficient entries of $b_0$ and $b_\tau$ are drawn independently from $\mathcal{N}(0,d^{-1})$.
The entries of $b_{\pi,t}$, $b_{\Delta,t}$, $b_{\sigma,t}$, and $b'_{\sigma,t}$ are drawn independently from $\mathcal{N}(0,2.5^2/d)$, separately for $t=0,1$.
These coefficients are redrawn for each repetition and then held fixed within that repetition.
This construction preserves the conditional mean $m_t(x)$ while inducing heteroskedasticity, skewness, heavy tails, and covariate-dependent distributional treatment heterogeneity.

\paragraph{Target design.}
Write $n_{\mathrm{test}}$ for the number of test units and $\hat{P}_{X,\mathrm{test}}=n_{\mathrm{test}}^{-1}\sum_{i=1}^{n_{\mathrm{test}}}\delta_{x_i}$ for their empirical covariate law.
The target design is the empirical covariate distribution on the test set crossed with a uniform treatment intervention:
\begin{equation*}
    Q_\rho
    =
    \hat{P}_{X,\mathrm{test}}
    \otimes
    \operatorname{Unif}\{0,1\}
    .
\end{equation*}
Thus, the evaluation averages the conditional interventional distributions across both treatment arms and over the held-out covariate distribution.

\subsubsection{TCGA}
\label{app:tcga}

\paragraph{Source and preprocessing.}
We use the processed TCGA gene-expression benchmark used in continuous-dose causal inference \citep{bica2020estimating,schwab2020learning}.
The benchmark contains $9{,}659$ units and $4{,}000$ gene-expression features.
We read RNA-seq arrays from the TCGA SQLite database, apply a $\log(1+\cdot)$ transform, min--max scale each gene, select the $4{,}000$ most variable genes, and normalize each sample by its row norm.
The selected feature matrix is then standardized using training-set means and standard deviations.
The train/validation/test split fractions are $0.64/0.16/0.20$.
For the semi-synthetic assignment and outcome mechanisms below, we compute the principal component analysis (PCA) score vector $z(x)\in\R^{d_z}$ from the leading $d_z=8$ principal components of the standardized expression matrix.
All predictive models are trained on the full $4{,}000$-dimensional feature vector unless otherwise stated.

\paragraph{Treatment and dosage assignment.}
There are three treatment classes, $A\in\{1,2,3\}$, and a continuous dosage $D\in[0,1]$; the intervention variable is $T=(A,D)$.
The treatment-specific coefficient vectors $v_a,r_a,\theta_a,q_a,s_a,u_{a,1},u_{a,2}$ belong to $\R^{d_z}$; their sampling laws are specified below.
Treatment is assigned according to
\begin{equation*}
    \Prob(A=a\mid X=x)
    =
    \frac{\exp(\gamma_A v_a^\top z(x))}
    {\sum_{b=1}^3 \exp(\gamma_A v_b^\top z(x))}
    ,
\end{equation*}
where $\gamma_A$ controls treatment-selection strength.
For each treatment $a$, define the covariate-dependent reference dosage $d_a(x)=\operatorname{logit}^{-1}(r_a^\top z(x))$.
The observed dosage is then sampled as $D\mid X=x,A=a\sim\operatorname{Beta}(1+\alpha_D d_a(x),1+\alpha_D\{1-d_a(x)\})$.
We set $\gamma_A=1.0$ and $\alpha_D=8.0$.

\paragraph{Stochastic dose-response law.}
Let $m_a(x,d)=\iota_a+\theta_a^\top z(x)-\lambda_a\{d-d_a(x)\}^2+\rho_a\sin(2\pi d)$ be the treatment- and dosage-specific conditional mean, where $\iota_a$ is a treatment-specific intercept.
We generate outcomes from the mixture law
\begin{equation}
\label{eq:tcga-dose-law}
    \begin{split}
        Y(a,d)\mid X=x
        &\sim
        \pi_a(x,d)
        \mathcal{N}\bigl(
            m_a(x,d)+(1-\pi_a(x,d))\Delta_a(x,d),\sigma_{a,1}^2(x,d)
        \bigr)
        \\
        &\quad
        +
        \{1-\pi_a(x,d)\}
        \mathcal{N}\bigl(
            m_a(x,d)-\pi_a(x,d)\Delta_a(x,d),\sigma_{a,2}^2(x,d)
        \bigr)
        ,
    \end{split}
\end{equation}
where
\begin{equation*}
    \pi_a(x,d)
    =
    \operatorname{logit}^{-1}(q_a^\top z(x)+2d-1)
    ,\quad
    \Delta_a(x,d)
    =
    0.5+0.5\tanh(s_a^\top z(x)+d)
    ,
\end{equation*}
\begin{equation*}
    \sigma_{a,1}(x,d)
    =
    0.1+0.3\operatorname{logit}^{-1}(u_{a,1}^\top z(x)+d)
    ,\quad
    \sigma_{a,2}(x,d)
    =
    0.2+0.5\operatorname{logit}^{-1}(u_{a,2}^\top z(x)-d)
    .
\end{equation*}
Let $I_{d_z}$ denote the $d_z\times d_z$ identity matrix.
For each repetition, the coefficient vectors are drawn once and then fixed within that repetition:
\begin{equation*}
    v_a
    \sim
    \mathcal{N}\bigl(
        0,(0.9^2/d_z)I_{d_z}
    \bigr)
    ,
    \quad
    r_a
    \sim
    \mathcal{N}\bigl(
        0,(1.0^2/d_z)I_{d_z}
    \bigr)
    ,
    \quad
    \theta_a
    \sim
    \mathcal{N}\bigl(
        0,(0.55^2/d_z)I_{d_z}
    \bigr)
    ,
\end{equation*}
\begin{equation*}
    q_a
    \sim
    \mathcal{N}\bigl(
        0,(0.75^2/d_z)I_{d_z}
    \bigr)
    ,
    \quad
    s_a
    \sim
    \mathcal{N}\bigl(
        0,(0.70^2/d_z)I_{d_z}
    \bigr)
    ,
    \quad
    u_{a,1},u_{a,2}
    \sim
    \mathcal{N}\bigl(
        0,(0.65^2/d_z)I_{d_z}
    \bigr)
    ,
\end{equation*}
with independent entries and independent draws across treatment arms.
We also draw $\lambda_a\sim\operatorname{Unif}(0.8,1.5)$, $\rho_a\sim\operatorname{Unif}(-0.30,0.30)$, and $\iota_a\sim \mathcal{N}(0,0.12^2)$ independently across treatment arms.
This setup requires methods to capture dose-dependent uncertainty and potential multimodality in addition to the mean dose-response surface.
Finally, the observed outcome is given by $Y_i^{\obs} = Y_i(A_i,D_i)$.

\paragraph{Target design.}
The target design is the empirical test covariate distribution crossed with a uniform intervention over treatments and a fixed dosage grid:
\begin{equation*}
    Q_\rho
    =
    \hat{P}_{X,\mathrm{test}}
    \otimes
    \operatorname{Unif}\{1,2,3\}
    \otimes
    \operatorname{Unif}(\mathcal{D})
    ,
    \quad
    \mathcal{D}=\{0,0.05,\ldots,1\}
    .
\end{equation*}
Equivalently, $\mathcal{D}$ consists of $21$ equally spaced dosage values on $[0,1]$.
Evaluation therefore averages over all treatment classes and all grid dosages.
\subsubsection{Jobs}
\label{app:jobs}

\paragraph{Source and sample construction.}
We use the National Supported Work randomized job-training experiment and Panel Study of Income Dynamics controls from the LaLonde benchmark \citep{lalonde1986evaluating,dehejia1999causal}.
The NSW randomized sample contains $297$ treated units and $425$ control units.
The PSID comparison group contains $2{,}490$ control units.
The common feature set drops the PSID-specific $re74$ column and uses $re75$ as the pre-treatment earnings covariate.

\paragraph{Leakage-free splitting.}
To prevent data leakage between the observational training set and the randomized evaluation set, we first split each data source prior to constructing the training and evaluation sets.
The NSW treated sample, NSW control sample, and PSID control sample are each split into training ($0.56$), validation ($0.24$), and test (remaining) fractions.
The held-out RCT split consists solely of NSW treated and control units and is never used for training.
The observational training split is formed by combining the corresponding training splits of NSW treated units, NSW control units, and PSID controls.
The validation split is used exclusively for hyperparameter tuning.

\paragraph{Covariates and outcome.}
We use the common \citep{lalonde1986evaluating} covariates
\begin{equation*}
    X
    =
    (\mathrm{age},\mathrm{education},\mathrm{black},\mathrm{hispanic},
    \mathrm{married},\mathrm{nodegree},re75)
    .
\end{equation*}
Continuous covariates are standardized using training-set statistics, with age, education, and $re75$ treated as continuous variables.
The outcome is post-treatment earnings $re78$.
Because earnings have a point mass near zero and a heavy right tail, models are trained on the transformed outcome
\begin{equation*}
    Y
    =
    \operatorname{asinh}(re78/1000)
    .
\end{equation*}
Figures and earnings-scale metrics use the inverse transform $re78=1000\sinh(Y)$.

\paragraph{Target design.}
Let $\mathrm{RCT}$ denote the index set of the held-out randomized NSW sample, let $n_{\mathrm{RCT}}=|\mathrm{RCT}|$, and write $\hat{P}_{X,\mathrm{RCT}}=n_{\mathrm{RCT}}^{-1}\sum_{i\in\mathrm{RCT}}\delta_{x_i}$ for its empirical covariate law.
The target design is the empirical covariate distribution of the held-out randomized NSW sample crossed with equal treatment-arm weights:
\begin{equation*}
    Q_\rho
    =
    \hat{P}_{X,\mathrm{RCT}}
    \otimes
    \operatorname{Unif}\{0,1\}
    .
\end{equation*}
Since individual counterfactual laws are unavailable, evaluation compares model-implied arm-level interventional distributions with randomized NSW arm-level distributions.

\subsection{Baselines}
\label{app:baselines}

For IHDP and Jobs, we compare with GANITE \citep{yoon2018ganite}, PO-Flow \citep{wu2025po}, DiffPO \citep{ma2024diffpo}, individualized normalizing flows (INFs) \citep{melnychuk2023normalizing}, DR-Learner \citep{kennedy2023towards}, CoDiTE \citep{park2021conditional}, DNet \citep{wu2023dnet}, Causal-DRF \citep{naf2026causal}, and SKCD (adapted) \citep{jain2026conditional}.
GANITE is an adversarial counterfactual imputation method for binary treatments.
PO-Flow and DiffPO are generative potential-outcome methods based on flow matching and diffusion models, respectively.
INFs estimate interventional outcome densities with normalizing flows.
DR-Learner is a doubly robust conditional mean method;
it is included to assess the extent to which a strong point-estimation method can be extended to a distributional plug-in estimator.
CoDiTE uses conditional kernel mean embeddings, DNet estimates noncrossing conditional quantiles, and Causal-DRF uses distributional forests to estimate conditional kernel treatment effects.
\citep{jain2026conditional} develop doubly robust estimation and testing, including the smoothed kernel conditional discrepancy (SKCD);
our SKCD (adapted) row evaluates predictive distributions obtained through a sampling adapter applied to their cross-fitted outcome conditional mean embedding estimates.

For TCGA, we compare with SCIGAN \citep{bica2020estimating}, DRNet \citep{schwab2020learning}, and VCNet \citep{nie2021vcnet}.
SCIGAN is the main adversarial baseline for continuous-valued treatments and dosages.
DRNet and VCNet are strong neural dose-response baselines designed for continuous treatment-response estimation.
All hyperparameters listed below are those used in the $100$ repetitions.

For distributional evaluation, let $B$ be the number of Monte Carlo samples drawn from each conditional law, with the benchmark-specific values given in Appendix~\ref{app:metrics-figures};
generative baselines are evaluated using $B$ such draws.
For deterministic or mean-only baselines, including DR-Learner, DRNet, and VCNet, we construct a distributional plug-in baseline by adding an empirical residual distribution to the estimated conditional mean.
For IHDP and Jobs, residuals are pooled within treatment arms.
For TCGA, residuals are pooled within treatment--dosage strata.
This construction gives mean-only methods access to a simple predictive distribution for distributional evaluation, while preserving their native mean estimates for point and policy metrics.

\method\ is trained with the stratified objective described in Section~\ref{subsec:conditional-estimator}.
For binary treatments, cells are treatment arms or low-dimensional treatment--covariate cells.
For TCGA, cells are treatment--dosage strata at the selected dosage resolution, with covariates handled by the conditional generator.
Mini-batches are sampled and normalized within cells.
Cells with too few observations are merged with nearby target-relevant cells at the resolution selected by validation.
No propensity score, inverse-propensity weight, or density-ratio estimate is used by \method.
All methods use the same train/validation/test splits within each repetition, and no method is tuned on test-set distributional metrics.

\subsubsection{Hyperparameters for \method.}
Table~\ref{tab:app-ganice-hyperparameters} reports the \method\ hyperparameters used in the main benchmark experiments.
Here G/C hidden denotes the hidden-layer widths of the generator and critic.
These runs use batch size $128$, Adam optimizers with generator learning rate $2\times10^{-4}$ and critic learning rate $10^{-4}$, Adam betas $(0,0.9)$, gradient-penalty weight $10$, and outcome critics anchored at $y=0$.

Raw generator outputs are bounded through a scaled sigmoid output layer, using intervals determined from training outcomes rather than hard clipping to the empirical training minimum and maximum.
Writing $\hat{Q}_\alpha$ for the empirical training-outcome quantile on the modeling scale, the intervals are $[\hat{Q}_{0.01}-0.75,\hat{Q}_{0.99}+0.75]$ for IHDP, expanded symmetrically to width $1$ if narrower;
$[\hat{Q}_{0.002}-0.4,\hat{Q}_{0.998}+0.4]$ for TCGA;
and $[0,\max\{1,\hat{Q}_{0.995}+0.35\}]$ for Jobs.

The implemented generator loss consists of the generator component of the stratified adversarial objective in Section~\ref{subsec:conditional-estimator}, with coefficient $1$, a separate cellwise transport regularizer multiplied by the listed generator transport weight, and factual CRPS and squared-error auxiliary terms with the listed weights.
The transport regularizer compares sorted observed and generated scalar outcomes within each minibatch cell, with generated outcomes evaluated at the sampled observational states.
It aggregates the mean absolute differences using the cell frequencies in that minibatch;
cells containing fewer than two sampled observations do not contribute.
The factual auxiliary terms compare the predictive distribution with the observed outcome and the generated sample mean with the observed outcome, respectively, using $6$ generated samples per observation.
The factual squared-error term during adversarial training is disabled for Jobs.

Pretraining performs generator updates before adversarial training using only factual CRPS with coefficient $1$ and sample-mean squared error with coefficient $2$ for IHDP and TCGA, and $1.5$ for Jobs.
These pretraining losses also use $6$ generated samples per observation.

The cell map lists the coordinates on which the dyadic partition of Section~\ref{subsec:conditional-estimator} is defined, the resolution gives the vector $\boldsymbol{m}$, and the target-mass Monte Carlo size is the number of draws from $Q_\rho$ used to estimate the cell masses $q_C$.

Residual quantile calibration on IHDP is fitted separately within each active cell using training residuals relative to generated conditional sample means, with $12$ generated samples per training observation and $256$ quantile levels.
At prediction time, calibrated and uncalibrated samples are blended with weights $0.75$ and $0.25$, respectively.

For Jobs, generated samples are first replaced by zero with an arm-specific probability equal to the empirical frequency of nonpositive earnings in the NSW training split, capped at $0.95$.
A monotone arm-level quantile map fitted on the validation RCT split using $256$ quantile levels and $128$ generated samples per validation covariate and treatment arm is then applied.
This subsequent quantile map can change the injected zero mass, so the final zero-earnings frequency is not constrained to equal the training frequency.
The Jobs output bounds and calibration operate on the transformed scale $Y=\operatorname{asinh}(re78/1000)$.
Neither the IHDP residual calibration nor the Jobs quantile calibration is followed by clipping back to the raw generator bounds.
Appendix~\ref{app:truncation-bias} discusses truncation bias associated with restricting outcome support.

\begin{table}[tb]
    \centering
    \scriptsize
    \setlength{\tabcolsep}{3pt}
    \renewcommand{\arraystretch}{1.12}
    \caption{\method\ hyperparameters used in the experiments.}
    \label{tab:app-ganice-hyperparameters}
    \begin{tabular}{@{}p{0.10\linewidth}p{0.32\linewidth}p{0.54\linewidth}@{}}
        \toprule
        Dataset & Architecture and optimization & Finite-resolution and auxiliary settings \\
        \midrule
        IHDP
        &
        latent dim. $4$;
        G/C hidden $(128,128)$;
        $520$ adversarial steps;
        $1$ critic step per generator step;
        shared conditional generator;
        $800$ pretraining steps.
        &
        Cell map:
        first two training-set PCA coordinates plus treatment;
        resolution $(1,1,1)$;
        minimum cell size $6$;
        target-mass Monte Carlo size $50{,}000$.
        Validation $\eW_1$ selects among four random restarts and the grid of generator transport weights $\{0.75,1.25\}$, factual CRPS weights $\{8,12\}$, and factual MSE weights $\{1,2\}$.
        Residual quantile calibration uses $12$ samples per observation, grid size $256$, and blend $0.75$.
        \\
        \midrule
        TCGA
        &
        latent dim. $4$;
        G/C hidden $(96,96)$;
        $620$ adversarial steps;
        $2$ critic steps per generator step;
        shared conditional generator;
        $1{,}600$ pretraining steps.
        &
        Covariate resolution $0$;
        treatment resolution $2$;
        dosage resolution $3$;
        minimum cell size $4$;
        target-mass Monte Carlo size $45{,}000$.
        Generator transport weight $0.6$;
        factual CRPS weight $10$ with $6$ samples;
        factual MSE weight $8$ with $6$ samples.
        \\
        \midrule
        Jobs
        &
        latent dim. $4$;
        G/C hidden $(96,96)$;
        $3$ critic steps per generator step;
        shared conditional generator;
        $420$ pretraining steps.
        &
        Cell map:
        standardized $re75$, Black indicator, and treatment;
        resolution $(1,1,1)$;
        minimum cell size $4$;
        target-mass Monte Carlo size $40{,}000$.
        Validation selects among $(180,220,260)$ adversarial steps,
        generator transport weights in $\{4,6\}$, and factual CRPS weights in $\{4,5\}$.
        The Jobs implementation additionally matches the zero-earnings mass using the NSW training split and fits arm-level quantile calibration on the validation RCT split.
        \\
        \bottomrule
    \end{tabular}
\end{table}

\subsubsection{Hyperparameters for IHDP and Jobs baselines.}
Table~\ref{tab:app-binary-baseline-hyperparameters} reports the baseline hyperparameters for the two binary-treatment benchmarks.
All listed neural methods use batch size $128$.
The outcome bounds for bounded generators and heads are set from the corresponding training outcomes.
The settings of CoDiTE, DNet, Causal-DRF, and SKCD (adapted) are shared across IHDP and Jobs unless stated otherwise.
CoDiTE regularization and DNet training steps are selected separately within each repetition by minimizing validation extended Wasserstein error on IHDP and validation RCT-assisted Wasserstein error on Jobs.
Candidate comparisons use the same initialization and evaluation random streams.
This validation uses $96$ Monte Carlo samples per conditional state on IHDP and $64$ samples per observation and treatment arm on Jobs.
The LightGBM propensity model of SKCD (adapted) is tuned with Optuna using the tree-structured Parzen estimator over the following ranges: learning rate $[10^{-3},0.2]$ (log scale); number of leaves $\{7,\ldots,255\}$ (log scale); maximum depth $\{-1,3,5,7,10,15\}$, where $-1$ denotes no depth limit; minimum child samples $\{2,\ldots,100\}$ (log scale); minimum split gain $[0,1]$; feature and bagging fractions $[0.5,1]$; bagging frequency $\{1,\ldots,10\}$; L1 and L2 regularization $[10^{-8},10]$ (log scale); and maximum bin count $\{63,127,255\}$.
The selected propensity model is refitted on the full outer training fold using the selected boosting-round count.

\begin{table}[tb]
    \centering
    \scriptsize
    \setlength{\tabcolsep}{3pt}
    \renewcommand{\arraystretch}{1.12}
    \caption{Baseline hyperparameters for IHDP and Jobs.}
    \label{tab:app-binary-baseline-hyperparameters}
    \begin{tabular}{@{}p{0.10\linewidth}p{0.14\linewidth}p{0.72\linewidth}@{}}
        \toprule
        Dataset & Method & Hyperparameters \\
        \midrule
        IHDP
        & GANITE
        & hidden dim. $96$;
        counterfactual generator iterations $500$;
        ITE inference iterations $500$;
        counterfactual discriminator steps $2$;
        ITE discriminator steps $1$;
        reconstruction weight $2$;
        ITE weight $5$;
        learning rate $10^{-3}$. \\
        IHDP
        & PO-Flow
        & hidden dim. $64$;
        training steps $750$;
        learning rate $10^{-3}$;
        weight decay $10^{-5}$;
        RK4 integration steps $20$. \\
        IHDP
        & DiffPO
        & hidden dim. $64$;
        time-embedding dim. $128$;
        residual blocks $4$;
        propensity steps $220$;
        diffusion steps $350$;
        diffusion time steps $45$;
        denoiser learning rate $5\times10^{-4}$;
        propensity learning rate $10^{-3}$;
        weight decay $10^{-5}$. \\
        IHDP
        & INFs
        & hidden dim. $64$;
        outcome bins $64$;
        nuisance steps $700$;
        target-flow steps $700$;
        nuisance learning rate $10^{-3}$;
        target learning rate $4\times10^{-3}$;
        weight decay $10^{-5}$;
        propensity-loss weight $1$;
        propensity clipping $0.05$;
        outcome noise std. $0.01$. \\
        IHDP
        & DR-Learner
        & hidden dim. $64$;
        $2$-fold cross-fitting;
        nuisance steps $600$;
        final CATE steps $750$;
        nuisance and final learning rates $10^{-3}$;
        weight decay $10^{-5}$;
        propensity clipping $0.05$. \\
        \midrule
        Jobs
        & GANITE
        & hidden dim. $128$;
        counterfactual generator iterations $650$;
        ITE inference iterations $650$;
        counterfactual discriminator steps $2$;
        ITE discriminator steps $1$;
        reconstruction weight $1$;
        ITE weight $5$;
        learning rate $10^{-3}$. \\
        Jobs
        & PO-Flow
        & hidden dim. $64$;
        training steps $850$;
        learning rate $10^{-3}$;
        weight decay $10^{-5}$;
        RK4 integration steps $18$. \\
        Jobs
        & DiffPO
        & hidden dim. $64$;
        time-embedding dim. $128$;
        residual blocks $3$;
        propensity steps $260$;
        diffusion steps $520$;
        diffusion time steps $55$;
        denoiser learning rate $5\times10^{-4}$;
        propensity learning rate $10^{-3}$;
        weight decay $10^{-5}$. \\
        Jobs
        & INFs
        & hidden dim. $64$;
        outcome bins $72$;
        nuisance steps $700$;
        target-flow steps $700$;
        nuisance learning rate $10^{-3}$;
        target learning rate $4\times10^{-3}$;
        weight decay $10^{-5}$;
        propensity-loss weight $1$;
        propensity clipping $0.05$;
        outcome noise std. $0.01$. \\
        Jobs
        & DR-Learner
        & hidden dim. $64$;
        $2$-fold cross-fitting;
        nuisance steps $650$;
        final CATE steps $800$;
        nuisance and final learning rates $10^{-3}$;
        weight decay $10^{-5}$;
        propensity clipping $0.05$. \\
        \midrule
        IHDP, Jobs
        & CoDiTE
        & Gaussian covariate and outcome kernels;
        arm-specific covariate bandwidths and a pooled outcome bandwidth selected by the median-distance heuristic;
        bandwidth computation uses at most $20{,}000$ pair distances;
        common arm regularization $\lambda\in\{10^{-4},10^{-3},10^{-2}\}$ selected by validation;
        arm-specific kernel systems $K_a+n_a\lambda I$;
        all training observations retained;
        predictive sampling uses outcome-RKHS projection onto the empirical probability simplex, with at most $256$ iterations and tolerance $5\times10^{-4}$. \\
        IHDP, Jobs
        & DNet
        & shared hidden layers $(200,100)$;
        treatment-specific outcome hidden layer $(100)$;
        ELU activations;
        $50$ noncrossing quantiles at levels $k/51$, $k=1,\ldots,50$;
        Adam learning rate $10^{-3}$;
        training steps selected from $\{750,1000,1250\}$;
        quantile-Huber threshold $1$;
        propensity-loss weight $1$;
        treatment-tower kernel L2 coefficient $10^{-5}$;
        initial quantile increment $0.02$;
        unbounded quantile heads;
        predictive sampling assigns equal mass to the estimated quantiles. \\
        IHDP, Jobs
        & Causal-DRF
        & $2{,}500$ trees in $50$ groups on IHDP and $4{,}000$ trees in $40$ groups on Jobs;
        independent Bernoulli$(1/2)$ group half-samples;
        tree subsampling fraction $0.8$ within each group, without replacement;
        random $50/50$ honest splits;
        minimum leaf size per treatment arm $4$;
        maximum depth $12$;
        $5$ candidate covariates per split on IHDP and $2$ on Jobs (the integer part of the square root of the number of covariates);
        split-balance parameter $0.2$;
        all admissible thresholds evaluated;
        Gaussian outcome-kernel bandwidth from the full-sample all-pairs median heuristic;
        splitting uses $32$ random Fourier frequencies with sine and cosine features. \\
        IHDP, Jobs
        & SKCD (adapted)
        & $2$-fold stratified cross-fitting;
        Gaussian covariate and outcome kernels with full-sample all-pairs median bandwidths;
        outcome-embedding systems $K+10^{-3}I$;
        LightGBM propensity estimation with $30$ Optuna trials per outer training fold;
        internal stratified $80/20$ split;
        binary log-loss selection;
        maximum $2{,}000$ boosting rounds;
        $10$-round early stopping;
        propensity clipping $10^{-6}$;
        predictive sampling equally mixes the two fold-specific outcome nuisance distributions after Euclidean simplex projection of their signed embedding coefficients. \\
        \bottomrule
    \end{tabular}
\end{table}

\subsubsection{Hyperparameters for TCGA baselines.}
Table~\ref{tab:app-tcga-baseline-hyperparameters} reports the baseline hyperparameters for the continuous-dose TCGA benchmark.
For VCNet, the discrete treatment indicator is appended to the covariates as a one-hot vector, and dosage is passed as the continuous treatment.
For deterministic mean-response methods, distributional evaluation uses the residual plug-in construction described above.

\begin{table}[tb]
    \centering
    \scriptsize
    \setlength{\tabcolsep}{3pt}
    \renewcommand{\arraystretch}{1.12}
    \caption{Baseline hyperparameters for TCGA.}
    \label{tab:app-tcga-baseline-hyperparameters}
    \begin{tabular}{@{}p{0.16\linewidth}p{0.80\linewidth}@{}}
        \toprule
        Method & Hyperparameters \\
        \midrule
        SCIGAN
        & noise dim. $8$;
        hidden dim. $64$;
        set dim. $16$;
        batch size $128$;
        GAN iterations $700$;
        inference-network iterations $1{,}000$;
        dosage samples per treatment $5$;
        factual reconstruction weight $1$;
        learning rate $10^{-3}$. \\
        VCNet
        & hidden dim. $64$;
        spline degree $2$;
        spline knots $(1/3,2/3)$;
        treatment-density grid size $10$;
        batch size $128$;
        training steps $1{,}100$;
        learning rate $10^{-3}$;
        weight decay $10^{-4}$;
        density-loss weight $0.2$;
        targeted regularization disabled;
        outcome standardization enabled. \\
        DRNet
        & hidden dim. $64$;
        dosage strata $5$;
        base layers $2$;
        treatment-specific layers $1$;
        dosage-head layers $2$;
        repeated dosage input enabled;
        batch size $128$;
        training steps $1{,}100$;
        learning rate $10^{-3}$;
        weight decay $10^{-4}$;
        outcome standardization enabled. \\
        \bottomrule
    \end{tabular}
\end{table}

\subsection{Evaluation Metrics in Section~\ref{sec:experiments}}
\label{app:metrics-figures}

This subsection gives the exact definitions of the primary distributional metrics and figure diagnostics reported in Section~\ref{sec:experiments}.
Each benchmark is evaluated over $100$ repetitions.
The Monte Carlo sample size is $B=192$ per conditional law on IHDP and $B=128$ on TCGA, with the same number of samples drawn from the true law and from the fitted model, and $B=96$ per covariate vector and treatment arm on Jobs, where the comparison is with observed randomized outcomes.
For a scalar probability law $\mu$, write $F_\mu(y)=\mu((-\infty,y])$ and $Q_\mu(\alpha)=\inf\{y\in\R:F_\mu(y)\ge\alpha\}$ for its CDF and generalized quantile function, with $\alpha\in(0,1)$.
For an evaluation state $s$, $\mu_s^\ast$ and $\hat{\mu}_s$ denote the true and fitted conditional laws; their CDFs are $F_s^\ast$ and $\hat{F}_s$, and their quantile functions are $Q_s^\ast$ and $\hat{Q}_s$.

\paragraph{Empirical extended Wasserstein error for semi-synthetic benchmarks.}
For IHDP and TCGA, the true conditional interventional laws are known by construction.
The primary metric is the empirical analogue of the extended Wasserstein risk analyzed in the theory:
a target-design average of statewise Wasserstein distances, with comparisons made only at the same treatment--covariate state.
Since outcomes are scalar, each $W_1$ distance is approximated using sorted Monte Carlo samples drawn from the true and fitted conditional laws.
When writing distributional distances, we identify a univariate law with its CDF.

For IHDP, the evaluation states are $(x_i,t)$ for $i=1,\ldots,n_{\mathrm{test}}$ and $t\in\{0,1\}$.
For each state $(x_i,t)$, we draw $B$ samples from the true law in \eqref{eq:ihdp-dist-law} and $B$ samples from the fitted model.
The empirical extended Wasserstein error is
\begin{equation}
\label{eq:ihdp-ew-metric}
    \frac{1}{2n_{\mathrm{test}}}
    \sum_{i=1}^{n_{\mathrm{test}}}
    \sum_{t=0}^1
    W_1\bigl(
        \hat{\mu}_{x_i,t},\mu_{x_i,t}^\ast
    \bigr)
    .
\end{equation}

For TCGA, the evaluation states are $(x_i,a,d)$ for $i=1,\ldots,n_{\mathrm{test}}$, $a\in\{1,2,3\}$, and $d\in\mathcal{D}$.
The empirical extended Wasserstein error is
\begin{equation}
\label{eq:tcga-ew-metric}
    \frac{1}{3n_{\mathrm{test}}|\mathcal{D}|}
    \sum_{i=1}^{n_{\mathrm{test}}}
    \sum_{a=1}^3
    \sum_{d\in\mathcal{D}}
    W_1\bigl(
        \hat{\mu}_{x_i,a,d},
        \mu_{x_i,a,d}^\ast
    \bigr)
    .
\end{equation}

\paragraph{Randomized-trial-assisted Wasserstein error for Jobs.}
For Jobs, individual counterfactual distributions are not observed.
We therefore evaluate arm-level interventional distributions using the held-out randomized NSW sample.
For $a\in\{0,1\}$, let $n_{a,\mathrm{RCT}}=\sum_{i\in\mathrm{RCT}}\1\{T_i=a\}$ denote the size of its randomized treatment arm.
For a fitted model, define the model-implied interventional CDF
\begin{equation*}
    \hat{F}_a(y)
    =
    \frac{1}{n_{\mathrm{RCT}}}
    \sum_{i\in\mathrm{RCT}}
    \hat{F}_{x_i,a}(y)
    ,
    \quad
    a\in\{0,1\}
    ,
\end{equation*}
where the average is over covariates in the held-out randomized sample and each $\hat{F}_{x_i,a}$ is the empirical CDF of the $B$ generated samples at $(x_i,a)$.
The randomized empirical arm-level CDF is
\begin{equation*}
    \hat{F}_a^{\mathrm{RCT}}(y)
    =
    \frac{1}{n_{a,\mathrm{RCT}}}
    \sum_{i\in\mathrm{RCT}:T_i=a}
    \1\{Y_i^{\obs}\le y\}
    .
\end{equation*}
The quantile functions $\hat{Q}_a$ and $\hat{Q}_a^{\mathrm{RCT}}$ are the generalized inverses of $\hat{F}_a$ and $\hat{F}_a^{\mathrm{RCT}}$, respectively.
The Jobs metric reported in Table~\ref{tab:main-exp} is
\begin{equation}
\label{eq:jobs-rct-w1}
    \mathrm{RCT}\text{-}W_1
    =
    \frac{1}{2}
    \sum_{a=0}^1
    W_1\Bigl(
        \hat{F}_a,
        \hat{F}_a^{\mathrm{RCT}}
    \Bigr)
    .
\end{equation}
This metric is not a pointwise counterfactual error;
it measures whether the generated interventional distributions agree with randomized arm-level outcome distributions.

\paragraph{Figure~\ref{fig:main-exp} (a): IHDP quantile treatment-effect error.}
For each quantile level $\alpha\in\{0.05,0.10,\ldots,0.95\}$, define the true average quantile treatment effect (QTE)
\begin{equation*}
    \Delta_\alpha^\ast
    =
    \frac{1}{n_{\mathrm{test}}}
    \sum_{i=1}^{n_{\mathrm{test}}}
    \bigl[
        Q^\ast_{x_i,1}(\alpha)
        -
        Q^\ast_{x_i,0}(\alpha)
    \bigr]
    ,
\end{equation*}
and the corresponding model estimate
\begin{equation*}
    \hat{\Delta}_\alpha
    =
    \frac{1}{n_{\mathrm{test}}}
    \sum_{i=1}^{n_{\mathrm{test}}}
    \bigl[
        \hat{Q}_{x_i,1}(\alpha)
        -
        \hat{Q}_{x_i,0}(\alpha)
    \bigr]
    .
\end{equation*}
The quantile functions are computed from $4{,}096$ samples of the true law and $512$ generated samples per state.
The panel plots $|\hat{\Delta}_\alpha-\Delta_\alpha^\ast|$ as a function of $\alpha$ for the selected methods.

\paragraph{Figure~\ref{fig:main-exp} (b): TCGA dose-wise distributional diagnostics.}
Fix treatment class $a=1$ and let $\mathcal{I}_{\mathrm{sub}}$ be the index set of the displayed held-out subgroup, selected as the test units with the largest predicted dose sensitivity;
we compute model-implied dose-indexed conditional summaries from generated samples.
The median curve is
\begin{equation*}
    d
    \mapsto
    \frac{1}{|\mathcal{I}_{\mathrm{sub}}|}
    \sum_{i\in\mathcal{I}_{\mathrm{sub}}}
    \hat{Q}_{x_i,a,d}(0.5)
    ,
\end{equation*}
with the corresponding true median computed from \eqref{eq:tcga-dose-law}.
The displayed predictive bands are constructed analogously from the generated $\alpha=0.1$ and $\alpha=0.9$ quantiles.
All displayed quantiles are computed from $96$ samples per state, drawn from the fitted model and from \eqref{eq:tcga-dose-law}, respectively.
This diagnostic distinguishes methods that track only the central dose-response trajectory from those that recover dose-dependent uncertainty.

\paragraph{Figure~\ref{fig:main-exp} (c): Jobs randomized CDF matching.}
We plot the randomized treated-arm empirical CDF $\hat{F}_1^{\mathrm{RCT}}$ and the model-implied treated-arm CDF $\hat{F}_1$ on the original earnings scale.
The inverse transformation $re78=1000\sinh(Y)$ is applied before plotting.
Analogous control-arm CDF diagnostics are reported in Appendix~\ref{app:additional-results}.

\subsection{Additional Evaluation Metrics}
\label{app:additional-evaluation-metrics}

The main text focuses on distributional errors aligned with the causal target.
Appendix~\ref{app:additional-results} reports additional diagnostics to verify that the empirical conclusions are not specific to the extended Wasserstein metric.

\paragraph{Notation for statewise and arm-level evaluation.}
Use $\mu_s^\ast$ and $\hat{\mu}_s$ for the true and fitted interventional laws associated with an evaluation index $s$.
For IHDP, let $\mathcal{I}=\{1,\ldots,n_{\mathrm{test}}\}\times\{0,1\}$, where $s=(i,t)$ refers to the state $(x_i,t)$.
For TCGA, let $\mathcal{I}=\{1,\ldots,n_{\mathrm{test}}\}\times\{1,2,3\}\times\mathcal{D}$, where $s=(i,a,d)$ refers to the state $(x_i,a,d)$.
For each index $s$, we draw $Y_{s,1}^\ast,\ldots,Y_{s,B}^\ast\sim\mu_s^\ast$ and $\hat{Y}_{s,1},\ldots,\hat{Y}_{s,B}\sim\hat{\mu}_s$, with $B$ as in Appendix~\ref{app:metrics-figures}.
The CDFs of $\mu_s^\ast$ and $\hat{\mu}_s$ are denoted by $F_s^\ast$ and $\hat{F}_s$, and their quantile functions by $Q_s^\ast(\alpha)$ and $\hat{Q}_s(\alpha)$; numerical evaluation uses the Monte Carlo samples above.

For Jobs, individual counterfactual laws are unavailable.
Distributional metrics are therefore evaluated at the randomized arm level, using the model-implied arm CDFs $\hat{F}_a$ and the held-out randomized NSW arm CDFs $\hat{F}_a^{\mathrm{RCT}}$ for $a\in\{0,1\}$, as defined in Appendix~\ref{app:metrics-figures}.
Metrics reported on the earnings scale use the inverse transformation $re78=1000\sinh(Y)$.

\paragraph{Conventional point and policy metrics.}
We compute conventional causal metrics based on conditional means to check whether distributional gains come at the expense of standard scalar summaries.
For IHDP, the precision in estimation of heterogeneous effect (PEHE) is
\begin{equation*}
    \mathrm{PEHE}
    =
    \frac{1}{n_{\mathrm{test}}}
    \sum_{i=1}^{n_{\mathrm{test}}}
    \bigl\{
        \hat{\tau}(x_i)-\tau(x_i)
    \bigr\}^2
    ,
\end{equation*}
where $\hat{m}_t(x)$ is the estimated conditional mean under treatment $t$ and $\hat{\tau}(x)=\hat{m}_1(x)-\hat{m}_0(x)$.
We also compute the absolute average treatment-effect (ATE) error
\begin{equation*}
    \mathrm{ATEErr}
    =
    \Biggl|
        \frac{1}{n_{\mathrm{test}}}
        \sum_{i=1}^{n_{\mathrm{test}}}
        \hat{\tau}(x_i)
        -
        \frac{1}{n_{\mathrm{test}}}
        \sum_{i=1}^{n_{\mathrm{test}}}
        \tau(x_i)
    \Biggr|
    .
\end{equation*}

For TCGA, let $\hat{m}_a(x,d)$ be the estimated conditional mean response.
The mean integrated squared error (MISE) is
\begin{equation*}
    \mathrm{MISE}
    =
    \frac{1}{3n_{\mathrm{test}}|\mathcal{D}|}
    \sum_{i=1}^{n_{\mathrm{test}}}
    \sum_{a=1}^3
    \sum_{d\in\mathcal{D}}
    \bigl\{
        \hat{m}_a(x_i,d)-m_a(x_i,d)
    \bigr\}^2
    .
\end{equation*}
For each unit and treatment arm, define the model-selected dosage $\hat{d}_{i,a}\in \argmax_{d\in\mathcal{D}} \hat{m}_a(x_i,d)$.
The dosage policy error (DPE) is
\begin{equation*}
    \mathrm{DPE}
    =
    \frac{1}{3n_{\mathrm{test}}}
    \sum_{i=1}^{n_{\mathrm{test}}}
    \sum_{a=1}^3
    \bigl\{
        \max_{d\in\mathcal{D}} m_a(x_i,d)
        -
        m_a(x_i,\hat{d}_{i,a})
    \bigr\}^2
    .
\end{equation*}
For the full treatment--dosage decision, let
\begin{equation*}
    (\hat{a}_i,\hat{d}_i)
    \in
    \argmax_{a\in\{1,2,3\},\ d\in\mathcal{D}}
    \hat{m}_a(x_i,d)
    .
\end{equation*}
The policy error (PE) is
\begin{equation*}
    \mathrm{PE}
    =
    \frac{1}{n_{\mathrm{test}}}
    \sum_{i=1}^{n_{\mathrm{test}}}
    \bigl[
        \max_{a\in\{1,2,3\},\ d\in\mathcal{D}} m_a(x_i,d)
        -
        m_{\hat{a}_i}(x_i,\hat{d}_i)
    \bigr]
    .
\end{equation*}

For Jobs, $\widehat{\mathrm{ATT}}$ denotes the method-specific estimate of the average treatment effect on the treated, and $\mathrm{ATT}_{\mathrm{RCT}}$ denotes the reference estimate computed from the held-out randomized NSW arms; their absolute discrepancy is $\mathrm{ATTErr}=|\widehat{\mathrm{ATT}}-\mathrm{ATT}_{\mathrm{RCT}}|$.
We also calculate an RCT-evaluated policy value.
Let $\pi:\cX\to\{0,1\}$ be the treatment rule induced by a fitted method, write $\pi_i=\pi(x_i)$ on held-out randomized covariates, and let
$\hat{p}=n_{\mathrm{RCT}}^{-1}\sum_{i\in\mathrm{RCT}}\pi_i$.
The policy value on the original earnings scale is estimated by
\begin{equation*}
    \hat{V}(\pi)
    =
    \hat{p}
    \frac{
        \sum_{i\in\mathrm{RCT}}\1\{\pi_i=1,T_i=1\}re78_i
    }{
        \sum_{i\in\mathrm{RCT}}\1\{\pi_i=1,T_i=1\}
    }
    +
    (1-\hat{p})
    \frac{
        \sum_{i\in\mathrm{RCT}}\1\{\pi_i=0,T_i=0\}re78_i
    }{
        \sum_{i\in\mathrm{RCT}}\1\{\pi_i=0,T_i=0\}
    }
    ,
\end{equation*}
with the convention that empty matched arms are replaced by the corresponding randomized arm mean on the original earnings scale.
For generative methods, conditional means are estimated from generated samples;
for deterministic baselines, we use their native mean estimates.

\paragraph{Proper scoring rule.}
We compute the continuous ranked probability score (CRPS), a strictly proper scoring rule for univariate predictive distributions \citep{gneiting2007strictly,dombry2024distributional}.
For a predictive CDF $F$ and observation $y$,
\begin{equation*}
    \mathrm{CRPS}(F,y)
    =
    \int_{\R}
    \{F(z)-\1\{y\le z\}\}^2\dd z
    =
    \E[|Y-y|]-\frac{1}{2}\E[|Y-Y'|]
    ,
    \quad
    Y,Y'\overset{\mathrm{i.i.d.}}{\sim}F
    .
\end{equation*}
For semi-synthetic datasets, we compute the expected CRPS under the true interventional law:
\begin{equation*}
    \widehat{\mathrm{CRPS}}
    =
    \frac{1}{|\mathcal{I}|}
    \sum_{s\in\mathcal{I}}
    \frac{1}{B}
    \sum_{b=1}^B
    \mathrm{CRPS}(\hat{F}_s,Y_{s,b}^\ast)
    .
\end{equation*}
For Jobs, we report factual CRPS on the held-out randomized NSW sample:
\begin{equation*}
    \widehat{\mathrm{CRPS}}_{\mathrm{fact}}
    =
    \frac{1}{n_{\mathrm{RCT}}}
    \sum_{i\in\mathrm{RCT}}
    \mathrm{CRPS}(\hat{F}_{x_i,T_i},Y_i^{\obs})
    .
\end{equation*}
We additionally compute the same factual score on the original earnings scale:
\begin{equation*}
    \widehat{\mathrm{CRPS}}_{\mathrm{earn}}
    =
    \frac{1}{n_{\mathrm{RCT}}}
    \sum_{i\in\mathrm{RCT}}
    \mathrm{CRPS}(\hat{F}^{\$}_{x_i,T_i},re78_i)
    ,
\end{equation*}
where $\hat{F}^{\$}_{x_i,T_i}$ is the generated predictive distribution after applying $re78=1000\sinh(Y)$.
CRPS is useful because it is not a Wasserstein distance, although factual CRPS is included as an auxiliary training loss for \method.

\paragraph{Energy distance.}
We compute energy distance (ED) as a sample-based distribution discrepancy with geometry different from extended Wasserstein risk \citep{szekely2013energy}.
For a semi-synthetic state $s$,
\begin{equation*}
    \mathrm{ED}(\hat{\mu}_s,\mu_s^\ast)
    =
    2\E[\|\hat{Y}-Y^\ast\|]
    -
    \E[\|\hat{Y}-\hat{Y}'\|]
    -
    \E[\|Y^\ast-Y^{\ast\prime}\|]
    ,
\end{equation*}
where $\hat{Y},\hat{Y}'\sim\hat{\mu}_s$ and $Y^\ast,Y^{\ast\prime}\sim\mu_s^\ast$ independently.
A hat on $\mathrm{ED}$ denotes its sample-based approximation using the evaluation samples. The calculated semi-synthetic value is
\begin{equation*}
    \widehat{\mathrm{ED}}
    =
    \frac{1}{|\mathcal{I}|}
    \sum_{s\in\mathcal{I}}
    \widehat{\mathrm{ED}}(\hat{\mu}_s,\mu_s^\ast)
    .
\end{equation*}
For Jobs, energy distance is computed at the randomized arm level on the original earnings scale and averaged over arms:
\begin{equation*}
    \widehat{\mathrm{ED}}
    =
    \frac{1}{2}
    \sum_{a=0}^1
    \widehat{\mathrm{ED}}
    \bigl(
        \hat{F}_a,
        \hat{F}_a^{\mathrm{RCT}}
    \bigr)
    .
\end{equation*}

\paragraph{CDF discrepancy.}
For scalar outcomes, we compute the Kolmogorov--Smirnov discrepancy (KS)
\begin{equation*}
    \mathrm{KS}
    =
    \frac{1}{|\mathcal{I}|}
    \sum_{s\in\mathcal{I}}
    \sup_z |\hat{F}_s(z)-F_s^\ast(z)|
    .
\end{equation*}
For Jobs, the same discrepancy is computed at the randomized arm level:
\begin{equation*}
    \mathrm{KS}
    =
    \frac{1}{2}
    \sum_{a=0}^1
    \sup_z
    \bigl|
        \hat{F}_a(z)
        -
        \hat{F}_a^{\mathrm{RCT}}(z)
    \bigr|
    .
\end{equation*}

\paragraph{Quantile and quantile-effect errors.}
Let $\mathcal{A}_Q = \{0.05, 0.10, 0.25, 0.50, 0.75, 0.90, 0.95\}$ denote the quantile levels.
For semi-synthetic datasets, we compute the integrated quantile error (IQE)
\begin{equation*}
    \mathrm{IQE}
    =
    \Biggl[
        \frac{1}{|\mathcal{I}||\mathcal{A}_Q|}
        \sum_{s\in\mathcal{I}}
        \sum_{\alpha\in\mathcal{A}_Q}
        \{\hat{Q}_s(\alpha)-Q_s^\ast(\alpha)\}^2
    \Biggr]^{1/2}
    .
\end{equation*}
For IHDP, we additionally compute quantile treatment-effect error (QTEErr):
\begin{equation*}
    \mathrm{QTEErr}
    =
    \Biggl[
        \frac{1}{n_{\mathrm{test}}|\mathcal{A}_Q|}
        \sum_{i=1}^{n_{\mathrm{test}}}
        \sum_{\alpha\in\mathcal{A}_Q}
        \bigl\{
            [\hat{Q}_{x_i,1}(\alpha)-\hat{Q}_{x_i,0}(\alpha)]
            -
            [Q^\ast_{x_i,1}(\alpha)-Q^\ast_{x_i,0}(\alpha)]
        \bigr\}^2
    \Biggr]^{1/2}
    .
\end{equation*}
For TCGA, the analogous dose-quantile error (DQErr) is
\begin{equation*}
    \mathrm{DQErr}
    =
    \Biggl[
        \frac{1}{3n_{\mathrm{test}}|\mathcal{D}||\mathcal{A}_Q|}
        \sum_{i,a,d,\alpha}
        \{\hat{Q}_{x_i,a,d}(\alpha)-Q^\ast_{x_i,a,d}(\alpha)\}^2
    \Biggr]^{1/2}
    .
\end{equation*}
For Jobs, we compute both arm-level integrated quantile error and arm-level quantile treatment-effect error.
The arm-level IQE is
\begin{equation*}
    \mathrm{IQE}
    =
    \Biggl[
        \frac{1}{2|\mathcal{A}_Q|}
        \sum_{a=0}^1
        \sum_{\alpha\in\mathcal{A}_Q}
        \bigl\{
            \hat{Q}_a(\alpha)
            -
            \hat{Q}_a^{\mathrm{RCT}}(\alpha)
        \bigr\}^2
    \Biggr]^{1/2}
    ,
\end{equation*}
and the arm-level quantile treatment-effect error is
\begin{equation*}
    \mathrm{QTEErr}
    =
    \Biggl[
        \frac{1}{|\mathcal{A}_Q|}
        \sum_{\alpha\in\mathcal{A}_Q}
        \bigl\{
            [\hat{Q}_{1}(\alpha)-\hat{Q}_{0}(\alpha)]
            -
            [\hat{Q}_{1}^{\mathrm{RCT}}(\alpha)-\hat{Q}_{0}^{\mathrm{RCT}}(\alpha)]
        \bigr\}^2
    \Biggr]^{1/2}
    .
\end{equation*}

\paragraph{Tail functionals.}
Distributional causal inference is often motivated by tail risk rather than central tendencies \citep{firpo2007efficient,kallus2023robust}.
Let $\mathcal{A}_T=\{0.05,0.10,0.90,0.95\}$ be the tail probability levels. For $\alpha\in\mathcal{A}_T$, we define the lower- and upper-tail means (LTM and UTM) as
\begin{equation*}
    \mathrm{LTM}_\alpha(\mu)
    =
    \E_{\mu}[Y\mid Y\le Q_\mu(\alpha)]
    ,
    \quad
    \mathrm{UTM}_\alpha(\mu)
    =
    \E_{\mu}[Y\mid Y\ge Q_\mu(\alpha)]
    .
\end{equation*}
Here $Y$ has law $\mu$.
A hat on $\mathrm{LTM}$ or $\mathrm{UTM}$ denotes a sample-based approximation of the corresponding functional using generated outcomes.
For semi-synthetic datasets, we compute the tail error (TailErr) as
\begin{equation*}
    \mathrm{TailErr}
    =
    \frac{1}{|\mathcal{I}||\mathcal{A}_T|}
    \sum_{s\in\mathcal{I}}
    \sum_{\alpha\in\mathcal{A}_T}
    \Bigl(
        |\widehat{\mathrm{LTM}}_\alpha(\hat{\mu}_s)
        -
        \mathrm{LTM}_\alpha(\mu_s^\ast)|
        +
        |\widehat{\mathrm{UTM}}_\alpha(\hat{\mu}_s)
        -
        \mathrm{UTM}_\alpha(\mu_s^\ast)|
    \Bigr)
    .
\end{equation*}
For Jobs, the same quantity is computed at the randomized arm level on the earnings scale and averaged across arms.

\paragraph{Calibration, interval width, and PIT diagnostics.}
For the nominal coverage levels
$\mathcal{C}=\{0.50,0.80,0.90,0.95\}$, we define the central predictive interval $I_s(c)=[\hat{Q}_s\{(1-c)/2\},\hat{Q}_s\{(1+c)/2\}]$, where $\hat{Q}_s(\cdot)$ denotes the estimated quantile function.
For semi-synthetic datasets, the empirical coverage is defined as
\begin{equation*}
    \widehat{\mathrm{Cov}}(c)
    =
    \frac{1}{|\mathcal{I}|B}
    \sum_{s\in\mathcal{I}}
    \sum_{b=1}^B
    \1\{Y_{s,b}^\ast\in I_s(c)\}
    .
\end{equation*}
The reported calibration error (CalErr) is
\begin{equation*}
    \mathrm{CalErr}
    =
    \frac{1}{|\mathcal{C}|}
    \sum_{c\in\mathcal{C}}
    |\widehat{\mathrm{Cov}}(c)-c|
    .
\end{equation*}
We also report the average interval width at each nominal level, which serves as a diagnostic of sharpness.
For Jobs, calibration and interval width are evaluated on factual outcomes from the held-out randomized NSW sample.
Finally, we report probability integral transform (PIT) histograms of $\hat{F}_s(Y_{s,b}^\ast)$ for semi-synthetic benchmarks and $\hat{F}_{x_i,T_i}(Y_i^{\obs})$ for Jobs.

\subsection{Additional Results}
\label{app:additional-results}

This subsection reports additional empirical results on the three benchmarks used in Section~\ref{sec:experiments}.
Each benchmark is evaluated over $100$ repetitions.
We focus on diagnostics complementary to those in Table~\ref{tab:main-exp}, including non-Wasserstein distributional metrics, quantile and tail errors, calibration, randomized-arm CDF fits, and an objective ablation.
Unless stated otherwise, reported values are means with standard errors computed over the same repetitions as in Table~\ref{tab:main-exp}.
Figures display a subset of methods for readability.

\paragraph{Robustness across distributional metrics.}
Tables~\ref{tab:app-ihdp-additional-metrics}--\ref{tab:app-jobs-additional-metrics} report the additional evaluation metrics introduced in Appendix~\ref{app:additional-evaluation-metrics}.
The results show that the performance gains of \method\ are not merely a consequence of optimizing the extended Wasserstein objective.
On TCGA, \method\ achieves the best performance across all reported distributional metrics, including CRPS, energy distance, KS distance, dose-quantile error, tail error, and calibration error.
On Jobs, \method\ also performs best on all reported distributional and factual predictive metrics except calibration, while attaining the lowest ATT error against the held-out randomized NSW sample.
On IHDP, Causal-DRF achieves the best distributional scores, followed by CoDiTE on several metrics, while \method\ attains the lowest PEHE and ATE error.
Compared with INFs, \method\ performs substantially better on integrated quantile error, quantile treatment-effect error, and tail error, whereas INFs has lower CRPS, energy distance, KS distance, and calibration error.
This distinction is important:
a method may accurately capture marginal distributional structure while still failing to recover treatment-induced distributional contrasts.

\begin{table}[tb]
    \centering
    \scriptsize
    \setlength{\tabcolsep}{3pt}
    \renewcommand{\arraystretch}{1.05}
    \caption{
    Additional evaluation metrics on IHDP.
    Lower is better for all columns.
    Bold and underline for PEHE follow the unrounded means ($0.049561$ for \method, $0.049718$ for Causal-DRF).
    }
    \label{tab:app-ihdp-additional-metrics}
    \resizebox{\textwidth}{!}{
    \begin{tabular}{lccccccccc}
        \toprule
        Method
        & CRPS
        & ED
        & KS
        & IQE
        & QTEErr
        & TailErr
        & CalErr
        & PEHE
        & ATEErr \\
        \midrule
        GANITE
        & $0.769\;(0.039)$
        & $0.899\;(0.075)$
        & $0.790\;(0.010)$
        & $0.986\;(0.038)$
        & $0.675\;(0.021)$
        & $1.840\;(0.069)$
        & $0.787\;(0.000)$
        & $0.403\;(0.031)$
        & $0.309\;(0.025)$ \\
        PO-Flow
        & $0.469\;(0.007)$
        & $0.298\;(0.008)$
        & $0.396\;(0.003)$
        & $0.595\;(0.010)$
        & $0.729\;(0.015)$
        & $0.988\;(0.016)$
        & $0.312\;(0.004)$
        & $0.457\;(0.024)$
        & $0.132\;(0.011)$ \\
        DiffPO
        & $0.391\;(0.004)$
        & $0.142\;(0.005)$
        & $0.261\;(0.003)$
        & $0.451\;(0.007)$
        & $0.688\;(0.011)$
        & $0.748\;(0.011)$
        & $0.077\;(0.004)$
        & $0.345\;(0.013)$
        & $0.536\;(0.012)$ \\
        INFs
        & $0.381\;(0.004)$
        & $0.124\;(0.003)$
        & $0.237\;(0.003)$
        & $0.569\;(0.007)$
        & $0.713\;(0.009)$
        & $0.963\;(0.013)$
        & $0.028\;(0.001)$
        & $0.279\;(0.010)$
        & $0.466\;(0.010)$ \\
        DR-Learner
        & $0.864\;(0.008)$
        & $1.089\;(0.013)$
        & $0.534\;(0.003)$
        & $1.300\;(0.012)$
        & $2.116\;(0.021)$
        & $1.907\;(0.019)$
        & $0.174\;(0.003)$
        & $4.504\;(0.095)$
        & $0.380\;(0.023)$ \\
        CoDiTE
        & $\underline{0.347}\;(0.004)$
        & $\underline{0.055}\;(0.001)$
        & $\underline{0.207}\;(0.001)$
        & $\underline{0.389}\;(0.008)$
        & $0.501\;(0.011)$
        & $\underline{0.561}\;(0.009)$
        & $\underline{0.024}\;(0.001)$
        & $0.060\;(0.002)$
        & $0.085\;(0.006)$ \\
        DNet
        & $0.642\;(0.009)$
        & $0.644\;(0.013)$
        & $0.616\;(0.003)$
        & $0.873\;(0.014)$
        & $1.075\;(0.020)$
        & $1.491\;(0.019)$
        & $0.571\;(0.002)$
        & $1.070\;(0.043)$
        & $0.268\;(0.020)$ \\
        Causal-DRF
        & $\textbf{0.346}\;(0.003)$
        & $\textbf{0.053}\;(0.001)$
        & $\textbf{0.194}\;(0.002)$
        & $\textbf{0.287}\;(0.004)$
        & $\textbf{0.375}\;(0.005)$
        & $\textbf{0.486}\;(0.006)$
        & $\textbf{0.020}\;(0.001)$
        & $\underline{0.050}\;(0.002)$
        & $\underline{0.078}\;(0.006)$ \\
        SKCD (adapted)
        & $0.372\;(0.004)$
        & $0.105\;(0.001)$
        & $0.309\;(0.001)$
        & $0.433\;(0.007)$
        & $0.604\;(0.010)$
        & $0.669\;(0.008)$
        & $0.088\;(0.002)$
        & $0.118\;(0.004)$
        & $0.084\;(0.006)$ \\
        \method
        & $0.382\;(0.004)$
        & $0.125\;(0.003)$
        & $0.288\;(0.003)$
        & $\underline{0.389}\;(0.005)$
        & $\underline{0.399}\;(0.006)$
        & $0.693\;(0.009)$
        & $0.164\;(0.003)$
        & $\textbf{0.050}\;(0.002)$
        & $\textbf{0.074}\;(0.006)$ \\
        \bottomrule
    \end{tabular}}
\end{table}

\begin{table}[tb]
    \centering
    \scriptsize
    \setlength{\tabcolsep}{3pt}
    \renewcommand{\arraystretch}{1.05}
    \caption{
    Additional evaluation metrics on TCGA.
    Lower is better for all columns.
    }
    \label{tab:app-tcga-additional-metrics}
    \resizebox{\textwidth}{!}{
    \begin{tabular}{lccccccccc}
        \toprule
        Method
        & CRPS
        & ED
        & KS
        & DQErr
        & TailErr
        & CalErr
        & MISE
        & DPE
        & PE \\
        \midrule
        SCIGAN
        & $\underline{0.420}\;(0.001)$
        & $\underline{0.306}\;(0.003)$
        & $\underline{0.503}\;(0.001)$
        & $\underline{0.558}\;(0.002)$
        & $\underline{0.901}\;(0.004)$
        & $\underline{0.510}\;(0.001)$
        & $\underline{0.155}\;(0.002)$
        & $0.212\;(0.004)$
        & $0.450\;(0.005)$ \\
        VCNet
        & $0.570\;(0.004)$
        & $0.606\;(0.009)$
        & $0.617\;(0.002)$
        & $0.788\;(0.007)$
        & $1.226\;(0.009)$
        & $0.585\;(0.002)$
        & $0.448\;(0.010)$
        & $\underline{0.134}\;(0.005)$
        & $\underline{0.399}\;(0.011)$ \\
        DRNet
        & $0.439\;(0.001)$
        & $0.344\;(0.002)$
        & $0.559\;(0.001)$
        & $0.600\;(0.002)$
        & $0.980\;(0.003)$
        & $0.605\;(0.002)$
        & $\textbf{0.149}\;(0.002)$
        & $\textbf{0.065}\;(0.003)$
        & $\textbf{0.176}\;(0.003)$ \\
        \method
        & $\textbf{0.409}\;(0.003)$
        & $\textbf{0.284}\;(0.007)$
        & $\textbf{0.382}\;(0.003)$
        & $\textbf{0.532}\;(0.007)$
        & $\textbf{0.820}\;(0.009)$
        & $\textbf{0.310}\;(0.002)$
        & $0.225\;(0.007)$
        & $0.213\;(0.005)$
        & $0.542\;(0.010)$ \\
        \bottomrule
    \end{tabular}}
\end{table}

\begin{table}[tb]
    \centering
    \scriptsize
    \setlength{\tabcolsep}{3pt}
    \renewcommand{\arraystretch}{1.05}
    \caption{
    Additional evaluation metrics on Jobs.
    Distributional metrics compare interventional arm distributions with held-out randomized NSW arm distributions; both CRPS columns and calibration use factual outcomes.
    Factual CRPS uses the transformed scale, and CRPS earn. uses the earnings scale.
    Lower is better except for policy value.
    }
    \label{tab:app-jobs-additional-metrics}
    \resizebox{\textwidth}{!}{
    \begin{tabular}{lcccccccccc}
        \toprule
        Method
        & Factual CRPS
        & CRPS earn.
        & ED
        & KS
        & IQE
        & QTEErr
        & TailErr
        & CalErr
        & ATTErr
        & Policy value \\
        \midrule
        GANITE
        & $1.463\;(0.021)$
        & $7291\;(338)$
        & $4093\;(562)$
        & $0.491\;(0.016)$
        & $6122\;(399)$
        & $5128\;(310)$
        & $10547\;(733)$
        & $0.787\;(0.000)$
        & $3979\;(283)$
        & $5330\;(74)$ \\
        PO-Flow
        & $\underline{0.801}\;(0.004)$
        & $\underline{3552}\;(26)$
        & $337\;(22)$
        & $0.176\;(0.004)$
        & $2729\;(93)$
        & $2966\;(131)$
        & $5015\;(176)$
        & $0.068\;(0.004)$
        & $2147\;(122)$
        & $\underline{5779}\;(84)$ \\
        DiffPO
        & $0.892\;(0.007)$
        & $3998\;(35)$
        & $1504\;(70)$
        & $0.281\;(0.005)$
        & $4927\;(162)$
        & $2864\;(127)$
        & $8833\;(333)$
        & $0.121\;(0.006)$
        & $2418\;(123)$
        & $5409\;(73)$ \\
        INFs
        & $1.195\;(0.007)$
        & $8198\;(30)$
        & $9688\;(93)$
        & $0.525\;(0.003)$
        & $19171\;(94)$
        & $2753\;(87)$
        & $37232\;(234)$
        & $0.159\;(0.003)$
        & $1574\;(98)$
        & $5670\;(79)$ \\
        DR-Learner
        & $0.859\;(0.004)$
        & $3769\;(37)$
        & $335\;(14)$
        & $0.208\;(0.003)$
        & $2637\;(123)$
        & $\underline{2481}\;(153)$
        & $7187\;(268)$
        & $0.200\;(0.003)$
        & $\underline{1377}\;(117)$
        & $\textbf{5946}\;(76)$ \\
        CoDiTE
        & $0.839\;(0.004)$
        & $3832\;(24)$
        & $315\;(16)$
        & $\underline{0.160}\;(0.003)$
        & $2492\;(75)$
        & $3003\;(133)$
        & $6231\;(198)$
        & $0.067\;(0.002)$
        & $2577\;(141)$
        & $5647\;(80)$ \\
        DNet
        & $0.844\;(0.005)$
        & $3649\;(29)$
        & $\underline{296}\;(16)$
        & $0.222\;(0.004)$
        & $\underline{1907}\;(78)$
        & $2586\;(148)$
        & $\underline{3490}\;(175)$
        & $0.246\;(0.005)$
        & $1731\;(119)$
        & $5735\;(84)$ \\
        Causal-DRF
        & $0.871\;(0.004)$
        & $4565\;(23)$
        & $2270\;(40)$
        & $0.258\;(0.003)$
        & $8620\;(69)$
        & $11932\;(166)$
        & $12873\;(162)$
        & $0.078\;(0.002)$
        & $9169\;(142)$
        & $5196\;(61)$ \\
        SKCD (adapted)
        & $0.854\;(0.004)$
        & $3979\;(25)$
        & $640\;(28)$
        & $0.194\;(0.004)$
        & $3495\;(80)$
        & $4674\;(168)$
        & $6281\;(190)$
        & $\textbf{0.047}\;(0.002)$
        & $3836\;(155)$
        & $5471\;(62)$ \\
        \method
        & $\textbf{0.769}\;(0.003)$
        & $\textbf{3408}\;(29)$
        & $\textbf{167}\;(10)$
        & $\textbf{0.127}\;(0.003)$
        & $\textbf{1644}\;(60)$
        & $\textbf{2107}\;(96)$
        & $\textbf{3210}\;(149)$
        & $\underline{0.058}\;(0.002)$
        & $\textbf{1068}\;(85)$
        & $5753\;(81)$ \\
        \bottomrule
    \end{tabular}}
\end{table}

\paragraph{Relation to scalar causal metrics.}
The additional metrics distinguish distributional causal learning from conventional mean-response or policy-oriented evaluation.
On IHDP, \method\ attains the lowest PEHE and ATE error, while Causal-DRF performs best on the distributional functionals.
On TCGA, DRNet performs best on mean-response and policy metrics, which is expected because it is specifically optimized for scalar dose-response estimation;
however, its distributional metrics are considerably worse.
This shows the limits of point-estimation benchmarks for evaluating interventional distribution learning.
On Jobs, DR-Learner achieves the highest policy value, whereas \method\ attains the best arm-level distributional fit, factual CRPS, QTE error, tail error, and ATT error.
\method\ thus improves estimation of the full interventional distribution while remaining competitive on scalar causal summaries.

\paragraph{Calibration and sharpness.}
Tables~\ref{tab:app-ihdp-additional-metrics}--\ref{tab:app-jobs-additional-metrics} report calibration errors, while Figures~\ref{fig:app-interval-widths} and~\ref{fig:app-pit-histograms} present interval sharpness and PIT diagnostics.
\method\ achieves the lowest calibration error on TCGA and the second-lowest on Jobs, where SKCD (adapted) performs best.
The interval-width analysis on TCGA shows that this improvement arises from capturing the broader dose-dependent uncertainty, rather than from producing overly narrow prediction intervals.
On Jobs, \method\ maintains substantially narrower factual prediction intervals than INFs and DiffPO at high coverage levels.
On IHDP, Causal-DRF has the lowest calibration error.
Among the methods displayed in Figure~\ref{fig:app-pit-histograms}, INFs exhibits the best PIT calibration on IHDP, but has larger quantile treatment-effect and tail errors than \method.
These results support the main conclusion:
calibration alone is insufficient for causal distribution learning, because the estimand is a target-design family of interventional conditional distributions and their distributional contrasts.

\begin{figure}[tb]
    \centering
    \begin{subfigure}{0.325\linewidth}
        \centering
        \includegraphics[width=\linewidth]{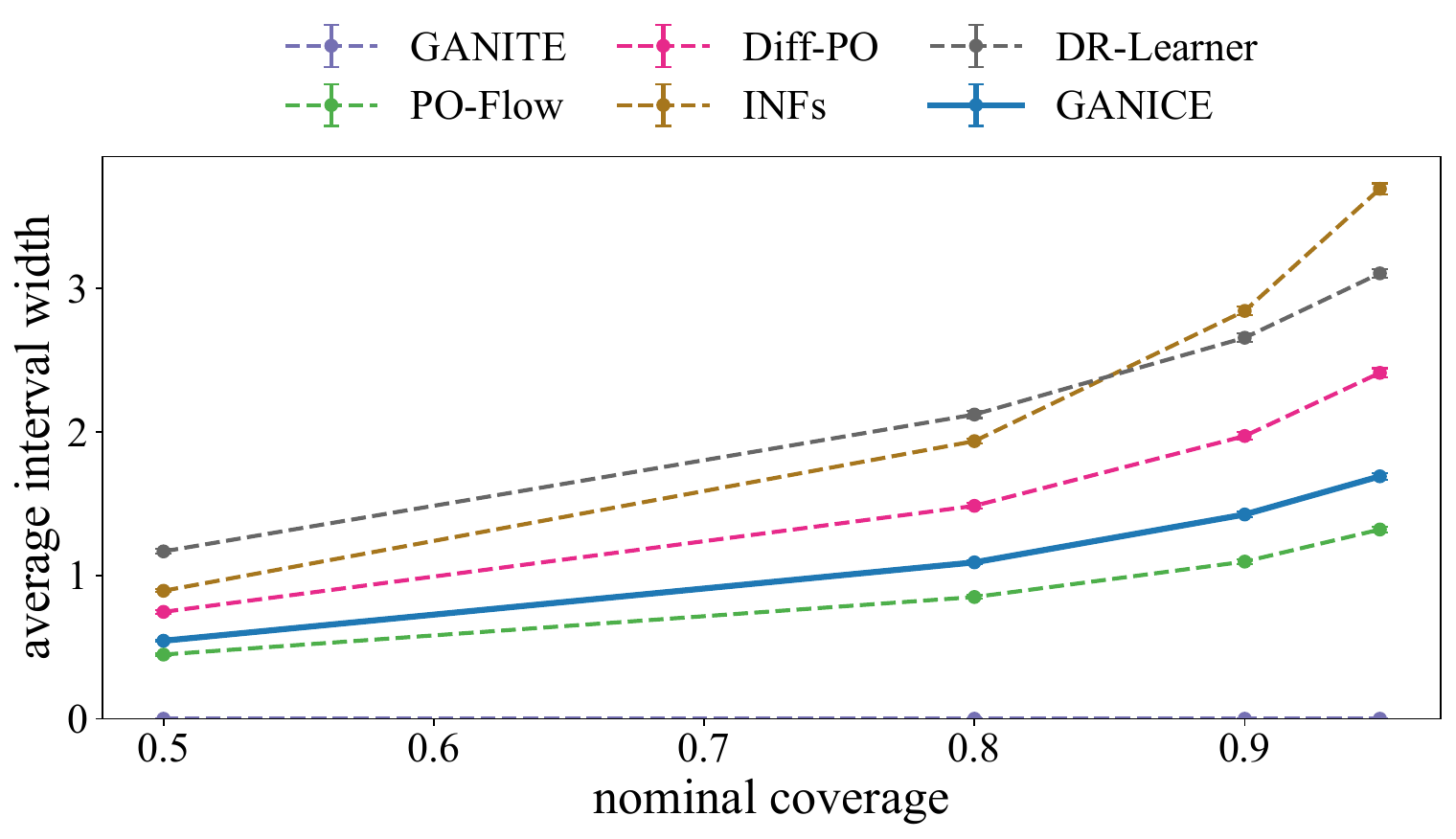}
        \caption{IHDP}
    \end{subfigure}
    \hfill
    \begin{subfigure}{0.325\linewidth}
        \centering
        \includegraphics[width=\linewidth]{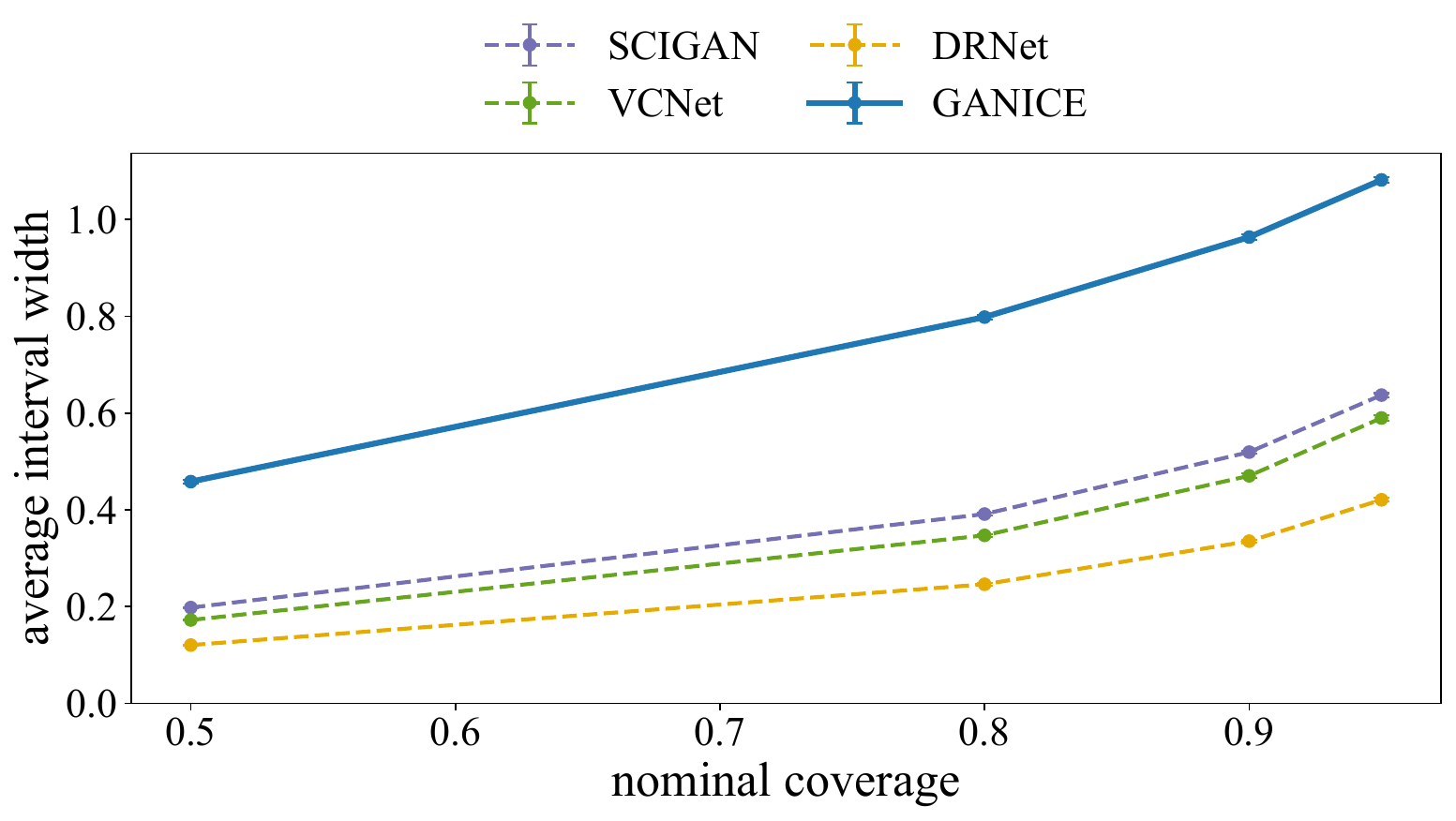}
        \caption{TCGA}
    \end{subfigure}
    \hfill
    \begin{subfigure}{0.325\linewidth}
        \centering
        \includegraphics[width=\linewidth]{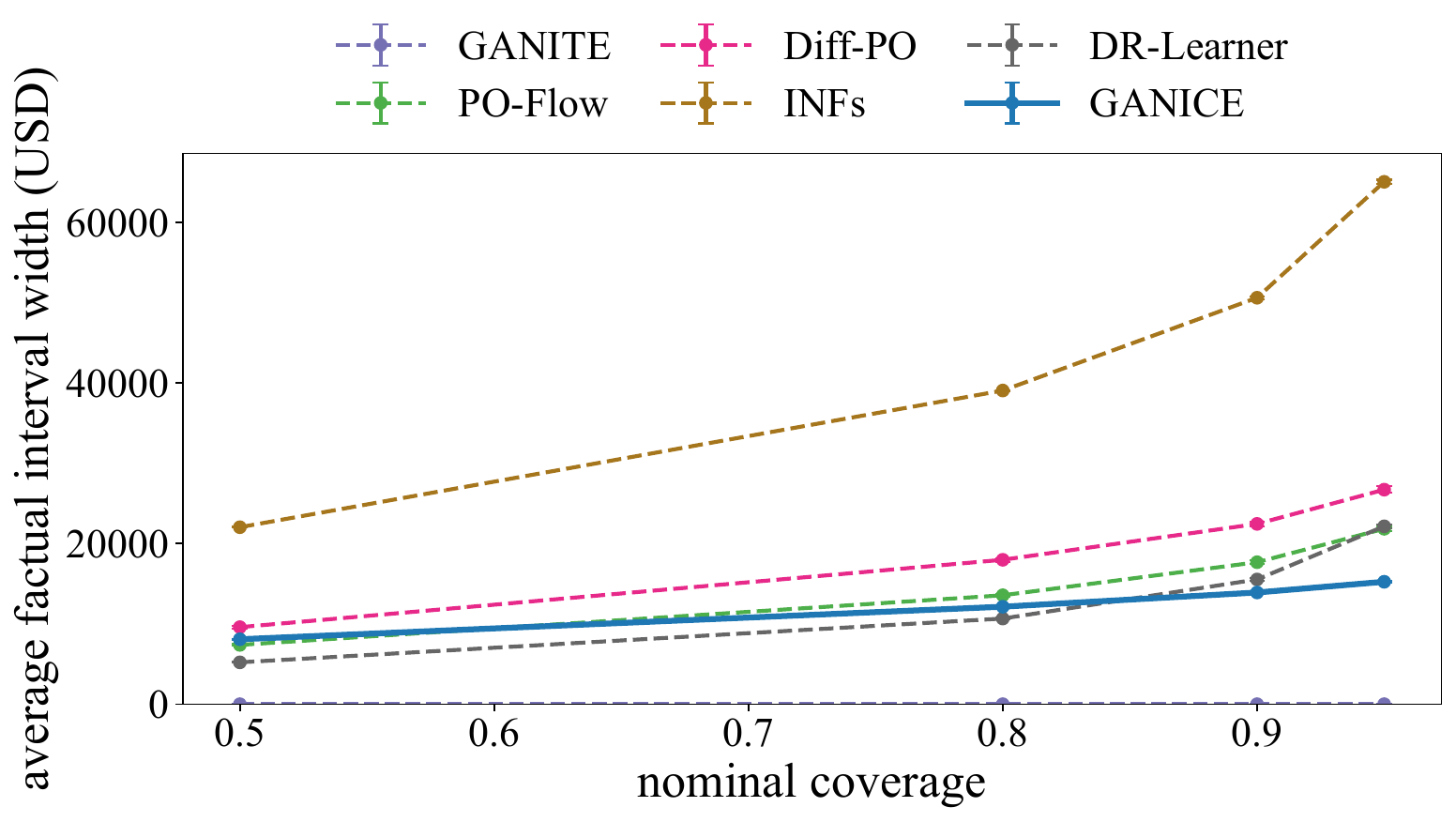}
        \caption{Jobs}
    \end{subfigure}
    \caption{
    Predictive interval widths across nominal coverage levels.
    Widths should be interpreted together with calibration error:
    narrower intervals are preferable only when empirical coverage is maintained.
    }
    \label{fig:app-interval-widths}
\end{figure}

\begin{figure}[tb]
    \centering
    \begin{subfigure}{0.325\linewidth}
        \centering
        \includegraphics[width=\linewidth]{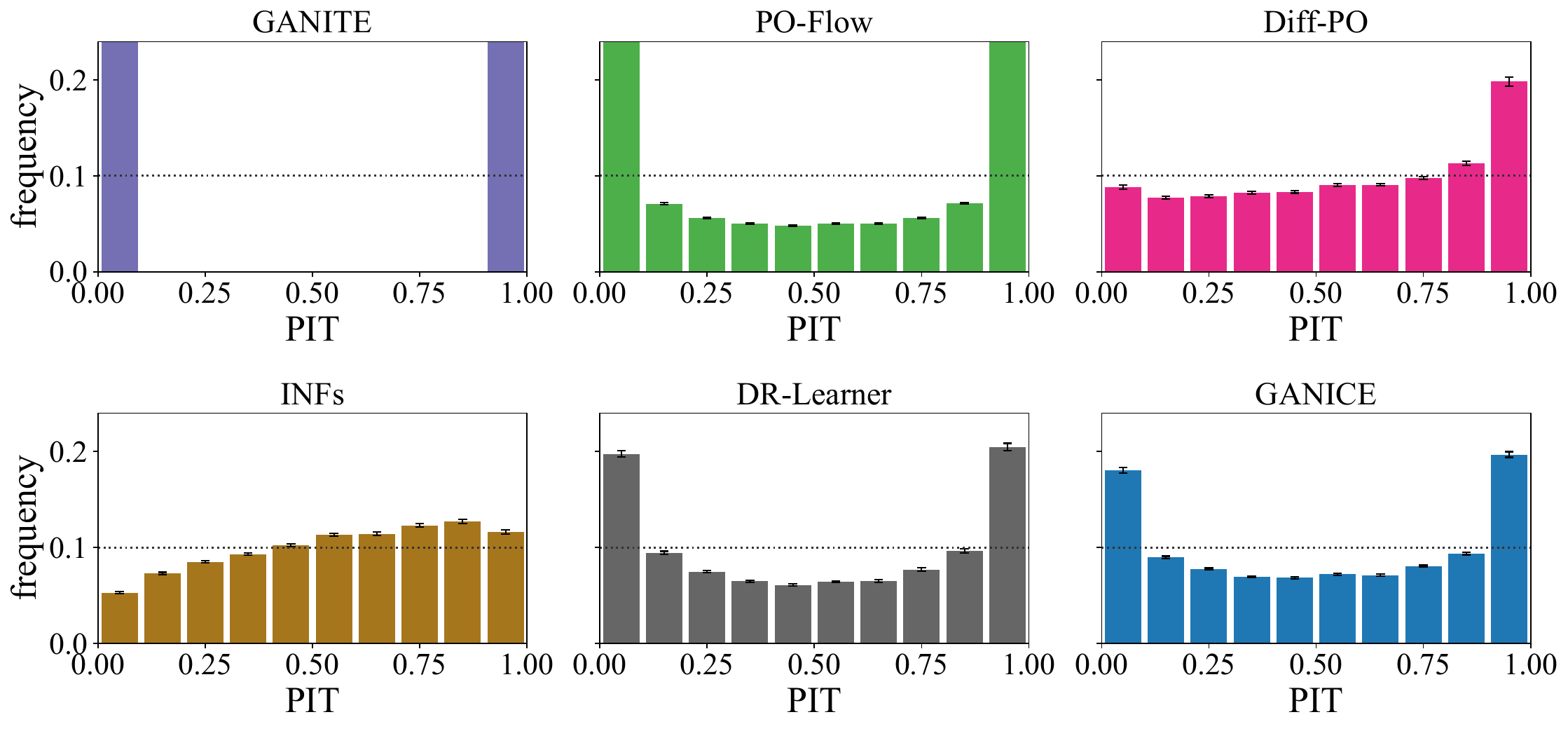}
        \caption{IHDP}
    \end{subfigure}
    \hfill
    \begin{subfigure}{0.325\linewidth}
        \centering
        \includegraphics[width=\linewidth]{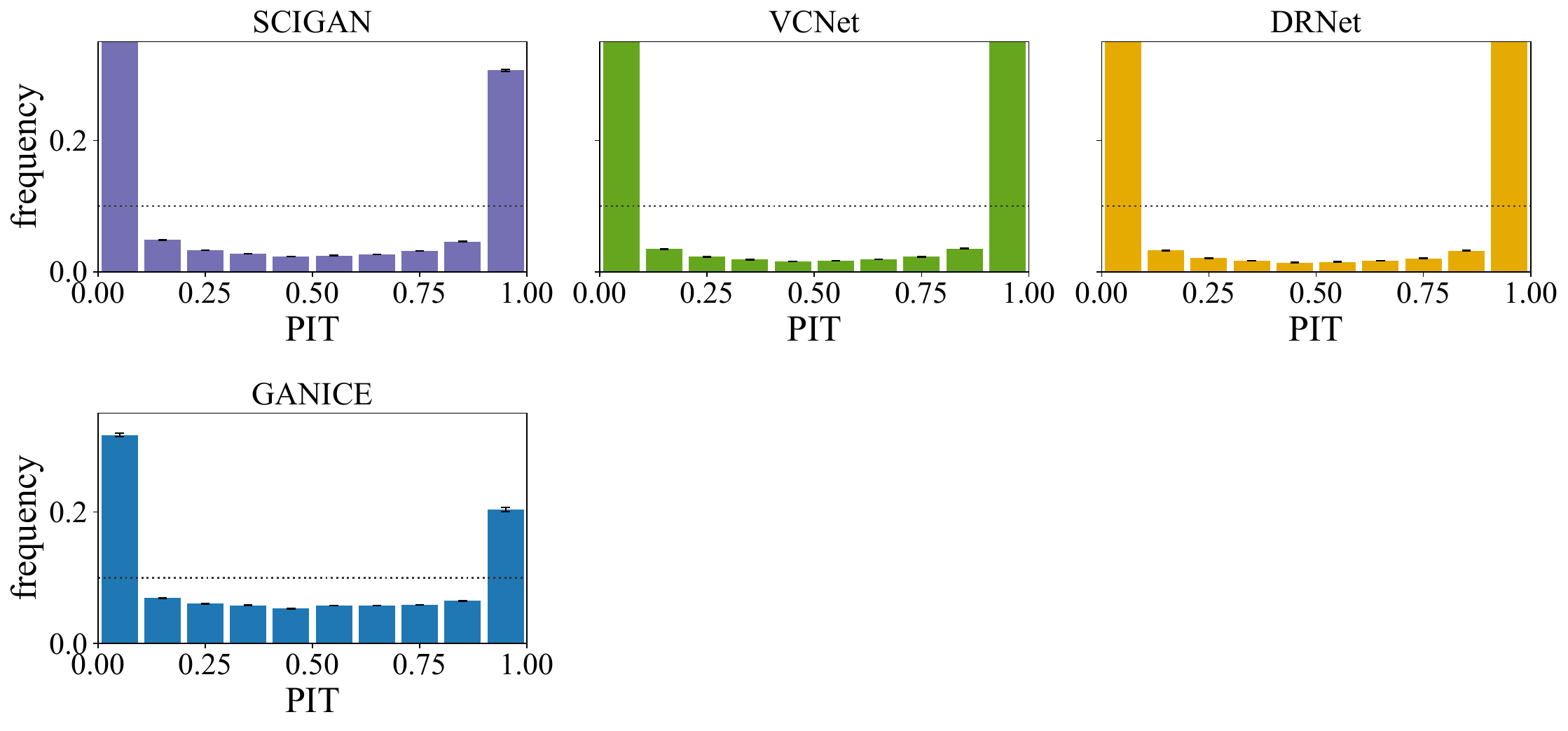}
        \caption{TCGA}
    \end{subfigure}
    \hfill
    \begin{subfigure}{0.325\linewidth}
        \centering
        \includegraphics[width=\linewidth]{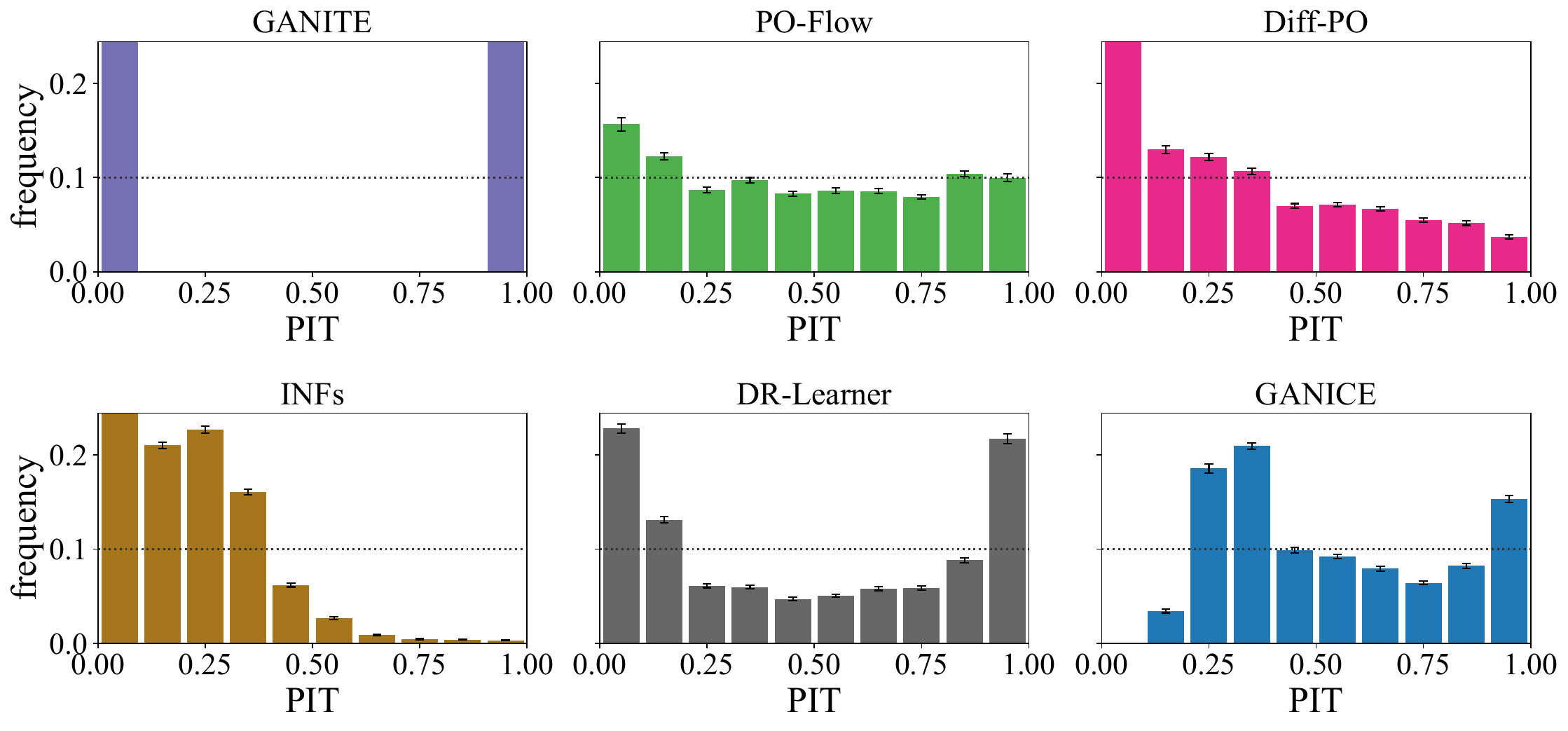}
        \caption{Jobs}
    \end{subfigure}
    \caption{
    Probability integral transform diagnostics.
    For calibrated predictive distributions with continuous CDFs, PIT histograms should be close to uniform.
    For Jobs, PIT is evaluated factually on the held-out randomized NSW sample.
    }
    \label{fig:app-pit-histograms}
\end{figure}

\paragraph{Randomized arm-level distributional fits on Jobs.}
Figure~\ref{fig:app-jobs-arm-cdfs} compares the generated treated- and control-arm CDFs with the empirical CDFs from the held-out randomized NSW sample.
The treated-arm panel complements Figure~\ref{fig:main-exp}, while the control-arm panel verifies that the improvement is not specific to a single treatment arm.
Across both arms, \method\ closely follows the randomized empirical distribution throughout the main body of the earnings distribution and avoids the large tail discrepancies observed for INFs.
PO-Flow and DR-Learner are competitive in parts of the treated-arm curve, but Table~\ref{tab:app-jobs-additional-metrics} shows that \method\ is better under integrated CDF, quantile, tail, and Wasserstein-type discrepancies.

\begin{figure}[tb]
    \centering
    \begin{subfigure}{0.495\linewidth}
        \centering
        \includegraphics[width=\linewidth]{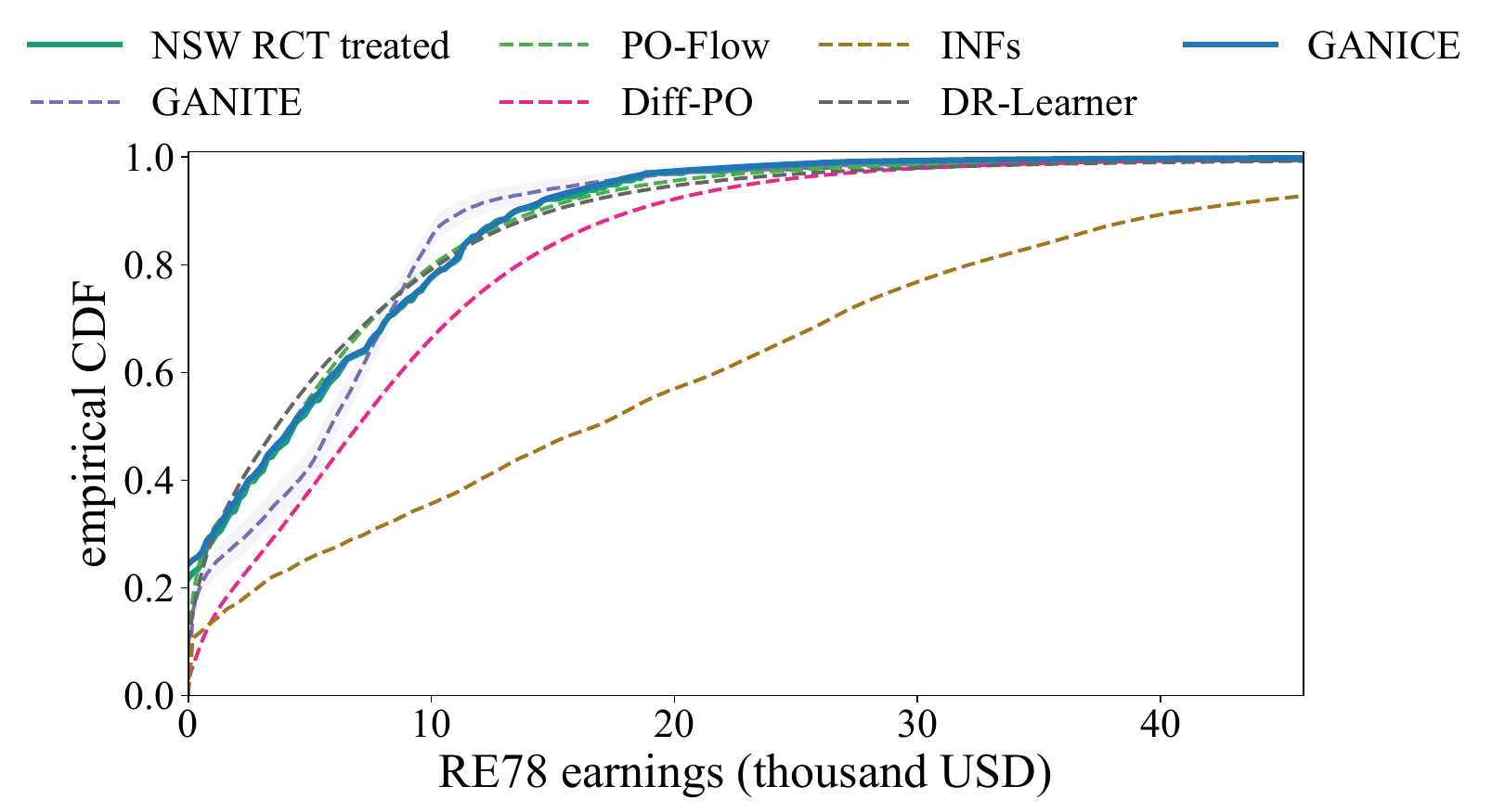}
        \caption{Treated arm}
    \end{subfigure}
    \hfill
    \begin{subfigure}{0.495\linewidth}
        \centering
        \includegraphics[width=\linewidth]{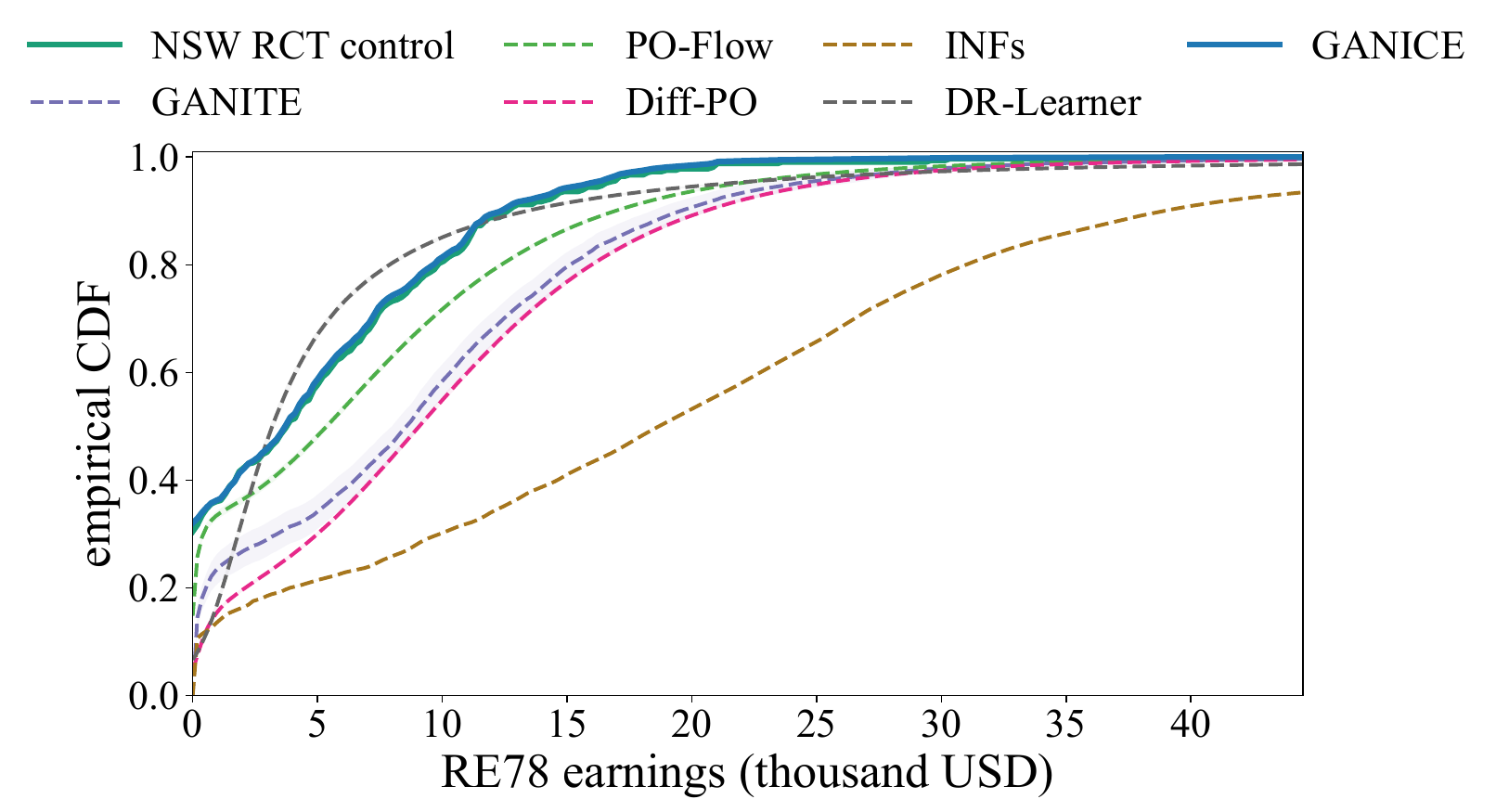}
        \caption{Control arm}
    \end{subfigure}
    \caption{
    Jobs randomized-arm CDF diagnostics.
    The panels compare model-implied interventional arm CDFs with held-out randomized NSW empirical CDFs on the earnings scale.
    }
    \label{fig:app-jobs-arm-cdfs}
\end{figure}

\paragraph{Objective ablation.}
Figure~\ref{fig:app-ihdp-objective-ablation} evaluates the role of the proposed finite-resolution causal objective using IHDP.
Replacing the statewise objective with a pooled WGAN-style objective substantially increases the empirical extended Wasserstein error, indicating that ordinary joint distribution matching can obscure statewise causal discrepancies.
The full \method\ objective and the no-cell-normalization variant perform similarly on this binary-treatment benchmark, although the full objective performs slightly better and is theoretically aligned with the target-design aggregation analyzed above.
The comparison with GANITE further suggests that adversarial imputation on factual coordinates alone is insufficient for accurate interventional distribution learning.
The ablation thus supports the two central design choices of \method: statewise extended-Wasserstein comparison and stratified target-design aggregation.

\begin{figure}[tb]
    \centering
    \includegraphics[width=0.66\linewidth]{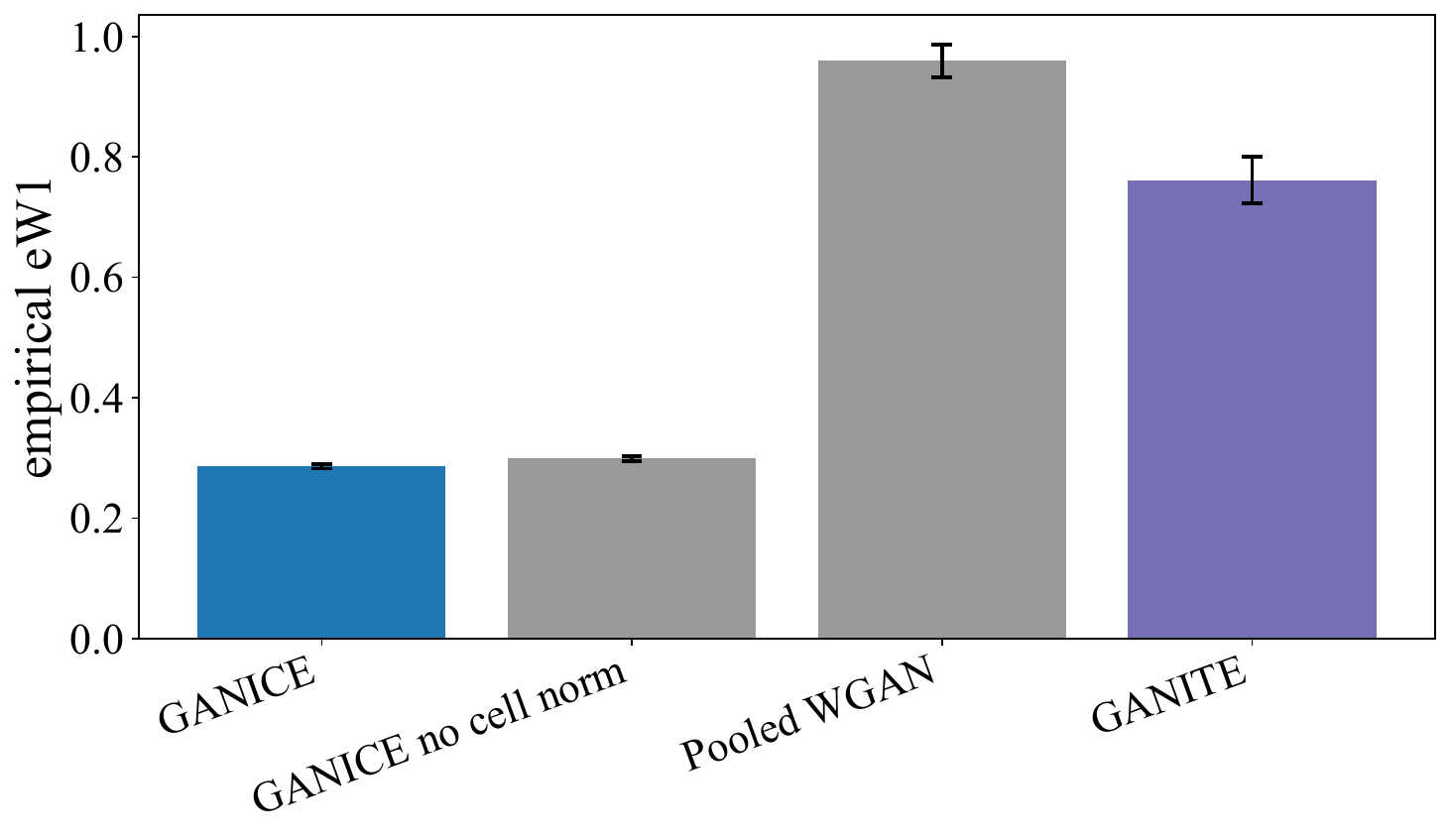}
    \caption{
    Objective ablation on IHDP.
    The full method is compared with variants that remove cell normalization or use a pooled WGAN-style objective.
    Lower empirical extended Wasserstein error is better.
    }
    \label{fig:app-ihdp-objective-ablation}
\end{figure}

\subsection{Implementation Diagnostics}
\label{app:implementation-diagnostics}

The main benchmark experiments use one shared conditional generator.
We compare it with a routed implementation, which fits one outcome generator per cell, on paired repetitions with the same split and evaluation target ($10$ for IHDP, $10$ for Jobs, and $5$ for TCGA).

\paragraph{Estimating the transfer discrepancy.}
Let $\hat{g}$ be the shared generator and $\hat{\nu}_C$ the fitted law in cell $C$ of the routed reference.
Draw $B=2{,}048$ independent samples $Y_{C,1},\ldots,Y_{C,B}$ from $\nu_{\hat{g},C}^\rho$ and an independent sample $Z_{C,1},\ldots,Z_{C,B}$ from $\hat{\nu}_C$.
Sampling the cell mixture uses $W\sim Q_\rho(\cdot\mid C)$ and independent latent noise.
For scalar outcomes, sorting these samples gives the empirical Wasserstein distance:
\begin{equation*}
    \hat{\varepsilon}
    =
    \sum_C q_C
    \frac{1}{B}
    \sum_{b=1}^{B}
    \bigl|
        Y_{C,(b)}
        -
        Z_{C,(b)}
    \bigr|
    .
\end{equation*}
Here parenthesized subscripts indicate order statistics.
The reported $\hat{\varepsilon}$ averages five independent replications of this computation.
Two independent samples from the same routed law provide a Monte Carlo self-distance.
The measured discrepancies exceeded this reference on all benchmarks.
As a sanity check, substituting the routed experts themselves for the shared generator in this computation gives zero discrepancy on every cell retained by the minimum-count rule.
For gated models using the same experts, the gate discrepancy $\eta$ in Proposition~\ref{prop:integrated-gate-transfer} gives another computable bound.

\begin{table}[tb]
    \centering
    \small
    \setlength{\tabcolsep}{4pt}
    \caption{
    Paired implementation diagnostics, with standard errors in parentheses and paired repetition counts beside benchmark names.
    Primary errors use the metrics in Table~\ref{tab:main-exp}.
    The ratio is $\hat{\varepsilon}$ divided by the shared primary error, using the displayed means.
    }
    \label{tab:implementation-diagnostics}
    \begin{tabular}{lcccc}
        \toprule
        Benchmark & Shared error & Routed error & $\hat{\varepsilon}$ & Ratio \\
        \midrule
        IHDP (10)
        & $0.292\;(0.012)$
        & $0.233\;(0.009)$
        & $0.155\;(0.012)$
        & $0.53$ \\
        Jobs (10)
        & $0.181\;(0.017)$
        & $0.177\;(0.020)$
        & $0.531\;(0.049)$
        & $2.93$ \\
        TCGA (5)
        & $0.413\;(0.015)$
        & $0.526\;(0.032)$
        & $0.258\;(0.034)$
        & $0.62$ \\
        \bottomrule
    \end{tabular}
\end{table}

Routing yields lower mean error on IHDP, similar error on Jobs, and higher error on TCGA.
For Jobs, $\hat{\varepsilon}$ compares cell distributions before arm-level calibration, whereas the primary metric compares calibrated arm distributions.
These diagnostics measure finite-sample discrepancies and do not establish the asymptotic transfer condition.

\subsection{Sensitivity Analysis}
\label{app:sensitivity-analysis}

We use a fully synthetic finite-state experiment to examine overlap, sample size, target-design shift, cell resolution, covariate dimension, and treatment--covariate heterogeneity.
The covariates are $Z\sim\Unif(\{0,1\}^d)$, with $d\ge2$, and the conditioning state is $W=(Z,T)$.
For binary treatment, $\Prob(T=1\mid Z=z)$ takes the values $\kappa/2$ and $1-\kappa/2$ on the two parity classes of $\sum_jz_j$, with $\kappa\in(0,1]$.
Define $c_1(z)=(-1)^{\sum_{j=1}^{d}z_j}$ and $c_2(z)=(-1)^{\sum_{j=1}^{d}jz_j}$.
The functions $m_t$ and $\theta_t$ determine the conditional mean and distributional shape, with heterogeneity parameter $h\in[0,1]$:
\begin{equation*}
    \begin{aligned}
        m_t(z)
        &=
        0.4c_1(z)+0.2c_2(z)
        +t\bigl\{
            0.8+0.35h\,c_1(z)
        \bigr\}
        ,\\
        \theta_t(z)
        &=
        \frac{\pi}{12}
        +
        \frac{\pi}{4}z_1z_2
        +
        \frac{\pi}{12}ht\bigl\{
            1+c_2(z)
        \bigr\}
        .
    \end{aligned}
\end{equation*}
Since $z\in\{0,1\}^d$, the product $z_1z_2$ equals $\1\{z_1=z_2=1\}$.
With independent $U_1,U_2\sim\Unif([0,1])$, independent of $(Z,T)$, the potential outcomes are
\begin{equation*}
    Y(t)
    =
    m_t(Z)
    +
    0.45
    \Bigl\{
        \sqrt{2}\cos(\theta_t(Z))\sin(2\pi U_1)
        +
        \frac{4}{\sqrt{5}}\sin(\theta_t(Z))\sin^3(2\pi U_2)
    \Bigr\}
    .
\end{equation*}
This generator is bounded and smooth in the latent variables.
For $H=\{z\in\{0,1\}^d:z_1=z_2=1\}$ and $\delta\in[0,1]$, the target design is
\begin{equation*}
    Q_\rho
    =
    \bigl\{
        (1-\delta)\Unif(\{0,1\}^d)
        +
        \delta\Unif(H)
    \bigr\}
    \otimes
    \Unif(\{0,1\})
    .
\end{equation*}
The target cell masses are known exactly, and a valid overlap constant for this design is $\kappa/(1+3\delta)$.
The nominal configuration is $(n,d,\kappa,\delta,h)=(2048,4,1,0,0.5)$ with $32$ cells.
We examine each factor separately, refit \method, and average the statewise $W_1$ errors over all $(z,t)$ using their target masses.
Each statewise $W_1$ error is computed from $1{,}024$ samples of the true law and $1{,}024$ generated samples.
Sample-size comparisons use nested prefixes of one observational sample; target-design comparisons use the same observations.

\begin{table}[tb]
    \centering
    \small
    \setlength{\tabcolsep}{4pt}
    \caption{
    Finite-state sensitivity analysis of \method.
    Entries are relative errors with the nominal configuration normalized to $1$.
    Each setting uses three optimization seeds.
    Reported standard errors of the ratios are below $0.01$, except for $d=3$, where the standard error is $0.03$.
    }
    \label{tab:stress-tests}
    \begin{tabular}{p{0.22\linewidth}p{0.33\linewidth}p{0.36\linewidth}}
        \toprule
        Factor & Settings & Relative error \\
        \midrule
        Overlap
        & $\kappa=1,0.5,0.25$
        & $1.000,1.005,1.010$ \\
        Sample size
        & $n=2048,1024,768,512$
        & $1.000,0.996,1.001,1.002$ \\
        Target-design shift
        & $\delta=0,0.45,0.9$
        & $1.000,0.888,0.426$ \\
        Cell resolution
        & $|\Pi|=2,4,8,16,32$
        & $0.998,1.000,0.996,0.997,1.000$ \\
        Covariate dimension
        & $d=2,3,4,5,6$
        & $0.267,0.753,1.000,1.026,1.030$ \\
        Heterogeneity
        & $h=0,0.25,0.5,0.75,1$
        & $0.820,0.910,1.000,1.091,1.183$ \\
        \bottomrule
    \end{tabular}
\end{table}

The errors change little over the tested overlap, sample-size, and cell-resolution ranges, suggesting an approximation or optimization floor at these sample sizes.
Errors increase with covariate dimension and heterogeneity.
At $\delta=0.9$, the error is $0.426$ of the nominal value under the corresponding target weighting.

\end{document}